\documentclass[bj, preprint]{imsart}

\usepackage[dvips]{graphics}
\DeclareGraphicsExtensions{.eps.gz,.eps,.epsi.gz,.epsi,.ps,.ps.gz}
\usepackage{epsfig}
\usepackage{color}
\graphicspath{{fig/}} 

\usepackage[]{subcaption}
\usepackage{multirow}
\usepackage{booktabs}

\usepackage[dvipsnames]{xcolor}

\usepackage{latexsym}
\usepackage{graphicx}
\usepackage{amsmath}
\usepackage{amsthm}
\usepackage{amsfonts}
\usepackage{amssymb}
\usepackage{mathrsfs}
\usepackage{multirow}
\usepackage{booktabs}
\usepackage{lscape}

\usepackage{hyperref}
\usepackage{cleveref}
\hypersetup{breaklinks=true,
    linkcolor=blue,
    citecolor=blue,
    colorlinks=true,
pdfborder={0 0 0}}

\usepackage{thmtools}
\usepackage{thm-restate}

\numberwithin{equation}{section}

\def\argmax{\mathop{\rm arg\, max}}
\def\argmin{\mathop{\rm arg\, min}}

\newcommand{\bel}{\begin{eqnarray}\label}
\newcommand{\eel}{\end{eqnarray}}
\newcommand{\bes}{\begin{eqnarray*}}
\newcommand{\ees}{\end{eqnarray*}}
\newcommand{\bei}{\begin{itemize}}
\newcommand{\eei}{\end{itemize}}
\newcommand{\beiftnt}{\begin{itemize}\footnotesize}

\def\benu{\begin{enumerate}}
\def\eenu{\end{enumerate}}

\def\argmax{\mathop{\rm arg\, max}}
\def\argmin{\mathop{\rm arg\, min}}
\def\real{{\mathbb{R}}}
\def\R{{\real}}

\def\E{{\mathbb{E}}}

\def\P{{\mathbb{P}}}

\def\complex{\mathop{{\rm I}\kern-.58em\hbox{\rm C}}\nolimits}
\def\pa{\partial}

\def\diag{\hbox{\rm diag}}

\def\rank{\hbox{\rm rank}}
\def\Rem{\hbox{\rm Rem}}
\def\sgn{\hbox{\rm sgn}}

\DeclareMathOperator{\trace}{trace}

\def\Var{\hbox{\rm Var}}

\DeclareMathOperator{\supp}{supp}

\def\mathbold{\boldsymbol} 

\def\ba{\mathbold{a}}
\def\tba{{\widetilde{\ba}}}
\def\bA{\mathbold{A}}

\def\bb{\mathbold{b}}
\def\hbb{{\widehat{\bb}}}
\def\bB{\mathbold{B}}\def\calB{{\cal B}}\def\scrB{{\mathscr B}}

\def\bD{\mathbold{D}}

\def\bfe{\mathbold{e}}

\def\bg{\mathbold{g}}

\def\bh{\mathbold{h}}
\def\tbh{{\widetilde{\bh}}}

\def\bI{\mathbold{I}}

\def\calL{{\cal L}}\def\scrL{{\mathscr L}}

\def\bM{\mathbold{M}}

\def\bp{\mathbold{p}}

\def\bP{\mathbold{P}}\def\hbP{{\widehat{\bP}}}

\def\bQ{\mathbold{Q}}\def\tbQ{{\widetilde{\bQ}}}

\def\br{\mathbold{r}}

\def\Shat{{\widehat{S}}}

\def\be{\mathbold{e}}

\def\bu{\mathbold{u}}
\def\tbu{{\widetilde{\bu}}}

\def\scrU{{\mathscr U}}

\def\bv{\mathbold{v}}
\def\tbv{{\widetilde{\bv}}}

\def\bw{\mathbold{w}}

\def\bW{\mathbold{W}}

\def\bx{\mathbold{x}}

\def\bX{\mathbold{X}}\def\tbX{{\widetilde{\bX}}}

\def\by{\mathbold{y}}

\def\bz{\mathbold{z}}
\def\hbz{{\widehat{\bz}}}\def\tbz{{\widetilde{\bz}}}

\def\bbeta{\mathbold{\beta}}\def\hbeta{\widehat{\beta}}

\def\hbbeta{\hat{\bbeta}}
\def\tbbeta{\tilde{\bbeta}}
\def\bbetabar{{\overline\bbeta}}

\def\bgamma{\mathbold{\gamma}}

\def\hbgamma{{\widehat{\bgamma}}}\def\tbgamma{{\widetilde{\bgamma}}}

\def\ep{\varepsilon}\def\eps{\epsilon}
\def\bep{ {\mathbold{\ep} }}

\def\btheta{\mathbold{\theta}}\def\htheta{\widehat{\theta}}
\def\ttheta{\widetilde{\theta}}
\def\hbtheta{{\widehat{\btheta}}}

\def\lam{\lambda}

\def\lambdabar{{\overline{\lambda}}}

\def\bSigma{\mathbold{\Sigma}}

\def\bSigmabar{{\overline\bSigma}}

\def\htau{\widehat{\tau}}

\def\bOmega{\mathbold{\Omega}}

\declaretheorem[name=Theorem,numberwithin=section]{theorem}
\declaretheorem[name=Proposition,sibling=theorem]{proposition}
\declaretheorem[name=Lemma,sibling=theorem]{lemma}
\declaretheorem[name=Corollary,sibling=theorem]{corollary}
\declaretheorem[name=Assumption,numberwithin=section]{assumption}

\declaretheorem[name=Remark,style=remark,numberwithin=section]{remark}
\declaretheorem[name=Example,style=definition,numberwithin=section]{example}

\crefname{assumption}{assumption}{assumptions}

\usepackage{marginnote}

\newcommand\pb[1]{{\color{magenta} #1}}

\usepackage[textsize=scriptsize]{todonotes}

\usepackage[color]{changebar} 
\cbcolor{Apricot}
\def\argmax{\mathop{\rm arg\, max}}
\def\argmin{\mathop{\rm arg\, min}}

\def\calB{{\mathscr B}}

\def\lasso{\hbbeta{}^{\text{\tiny (lasso)} } }
\def\lassoNoBold{\hbeta{}^{\text{\tiny (lasso)} } }

\def\tildelasso{\tbbeta{}^{\text{\tiny (lasso)} } }
\def\bhlasso{\bh{}^{\text{\tiny (lasso)} } }
\def\tildehlasso{\tbh{}^{\text{\tiny (lasso)} } }

\def\hbinit{\hbbeta{}^{\text{\tiny (init)} } }
\def\lamuniv{\lambda_{\text{\tiny univ} } }

\def\calL{\scrL}

\def\constants{\{\rho_*,\eta_2,\eta_3,\eps_1,\eps_2,\eps_3,\eps_4\}}

\begin{document}

\title{De-Biasing The Lasso With Degrees-of-Freedom Adjustment}
\runtitle{De-Biasing The Lasso With Degrees-of-Freedom Adjustment}
\date{\today}
\begin{aug}
  \author{\fnms{Pierre C.}  \snm{Bellec}\thanksref{t1}\ead[label=e1]{pierre.bellec@rutgers.edu}}
  \and
  \author{\fnms{Cun-Hui} \snm{Zhang}\thanksref{t2}\ead[label=e2]{czhang@stat.rutgers.edu}}

  \thankstext{t1}{
    Research partially supported by the NSF Grant DMS-1811976
    and DMS-1945428.
  }
  
  \thankstext{t2}{
      Research partially supported by the NSF Grants DMS-1513378,
  IIS-1407939, DMS-1721495, IIS-1741390 and CCF-1934924.  }

  \runauthor{Bellec and Zhang}

  \affiliation{Rutgers University}

  \address{Department of Statistics, Hill Center, Busch Campus, \\
      Rutgers University, Piscataway, NJ 08854, USA. \\
  }

\end{aug}

\begin{abstract}
    This paper studies schemes to de-bias the Lasso
    in sparse linear regression
    {with Gaussian design}
    where the goal is to estimate and construct confidence intervals 
    for a low-dimensional projection of the unknown coefficient vector
    in a preconceived direction $\ba_0$.
    Our analysis reveals that 
    previously analyzed
    propositions to de-bias the Lasso
    require a modification in order to enjoy 
    nominal coverage and asymptotic efficiency 
    in a full range of the level of sparsity.
    This modification takes the form of a degrees-of-freedom adjustment
    that accounts for the dimension of the model selected by the Lasso.
        The degrees-of-freedom adjustment (a) preserves the success of de-biasing methodologies
        in regimes where previous proposals were successful,
        and (b)
        repairs the nominal coverage and provides efficiency in regimes where
        previous proposals produce spurious inferences and
        provably fail to achieve the nominal coverage.
        Hence our theoretical and simulation results call for the implementation of
        this degrees-of-freedom adjustment in de-biasing methodologies. 

    Let $s_0$ denote the number of nonzero coefficients of the true coefficient vector 
    and $\bSigma$ the population Gram matrix.  
    The unadjusted de-biasing scheme may fail to achieve the nominal
        coverage 
        as soon as $s_0\ggg n^{2/3}$ if $\bSigma$ is known.
    If $\bSigma$ is unknown, the degrees-of-freedom adjustment
    grants efficiency for the contrast in a general direction $\ba_0$ when 
    $$
    \frac{s_0\log p}{n}
    + \min\Big\{\frac{s_\Omega\log p}{n}, \frac{\|\bSigma^{-1}\ba_0\|_1 \sqrt{\log p}}{\|\bSigma^{-1/2}\ba_0\|_2 \sqrt n  }\Big\} 
+ \frac{\min(s_\Omega,s_0)\log p}{\sqrt n} \to 0
    $$
    where $s_\Omega = \|\bSigma^{-1}\ba_0\|_0$.
    The dependence in $s_0,s_\Omega$ and $\|\bSigma^{-1}\ba_0\|_1$
    is optimal and closes a gap in previous upper and lower bounds.
    Our construction of the estimated score vector 
    provides a novel methodology to handle dense directions $\ba_0$.



    Beyond the degrees-of-freedom adjustment, our proof techniques yield a
    sharp $\ell_\infty$ error bound for the Lasso which is of independent
    interest.

\end{abstract}

\maketitle




\section{Introduction}

Consider a linear regression model 
\bel{LM}
\by = \bX\bbeta + \bep
\eel
with a sparse coefficient vector $\bbeta\in \R^p$, a Gaussian noise vector 
$\bep\sim N({\bf 0},\sigma^2\bI_n)$, and a Gaussian design matrix $\bX\in \R^{n\times p}$ with iid 
$N({\bf 0},\bSigma)$ rows. 
The purpose of this paper is to study the sample size requirement in de-biasing the Lasso 
for regular statistical inference of a linear contrast 
\bel{theta}
\theta  = \big\langle \ba_0, \bbeta \big\rangle 
\eel
at the $n^{-1/2}$ rate in the case of $p\gg n$ for both known and unknown $\bSigma$. 
{As a consequence of regularity, the $n^{-1/2}$ rate also corresponds to the length of confidence
    intervals for $\theta$.
} 

The problem was considered in \cite{ZhangOber11} in a general semi-low-dimensional (LD) 
approach where high-dimensional (HD) models are decomposed as 
\bel{gen-model}
\hbox{HD model = LD component + HD component}
\eel
in the same fashion as in semi-parametric inference \cite{BickelKRW98}. 
For the estimation of a real function $\theta = \theta(\bbeta)$ of a HD 
unknown parameter $\bbeta$, the decomposition in (\ref{gen-model}) was written 
in the vicinity of a given $\bbeta_0$ as 
\bel{decomp-beta}
\bbeta - \bbeta_0 = \bu_0\big(\theta - \theta_0\big) 
+ \bQ_0\big(\bbeta - \bbeta_0\big), 
\eel
where $\bu_0$ specifies the least favorable one-dimensional local sub-model giving the minimum Fisher information 
for the estimation of $\theta$, subject to $\big\langle \bu_0,\nabla\theta(\bbeta_0)\big\rangle =1$, 
and $\bQ_0 = \bI_{p\times p} - \bu_0(\nabla\theta(\bbeta_0)){}^\top$ projects $\bbeta - \bbeta_0$ 
to a space of nuisance parameters. \cite{ZhangOber11} went on to propose a 
low-dimensional projection estimator (LDPE) as a  
one-step maximum likelihood correction of an initial estimator $\hbinit$ in the direction
of the least favorable one-dimensional sub-model, 
\bel{one-step}
\htheta = \theta\big(\hbinit\big) + 
\argmax_{\phi\in\R}\,\hbox{log-likelihood}\big(\hbinit + \bu_0\phi \big),
\eel
and stated without proof that the asymptotic variance of 
such a one-step estimator achieves the lower bound given by the reciprocal of the Fisher information. 

For the estimation of a contrast (\ref{theta}) in linear regression (\ref{LM}), we have
$\nabla\theta(\bbeta_0)=\ba_0$, 
{the Fischer information in the one dimension sub-model
$\{\bbeta+\phi\bu,\phi\in\R\}$ is $\langle\bu,\bSigma \bu\rangle\sigma^{-2}$,
the least favorable sub-model is given by}
\begin{equation}
\label{u_0}
\bu_0 = \frac{\bSigma^{-1}\ba_0}{\langle\ba_0,\bSigma^{-1}\ba_0\rangle},
\quad
\text{i.e., the minimizer}
\quad
\bu_0=\argmin_{\bu\in\R^p:\langle\bu,\ba_0\rangle = 1} 
\frac{\langle\bu, \bSigma \bu\rangle}{\sigma^2}
\end{equation}
and the Fisher information for the estimation of $\theta$ is 
\bel{Fisher}
F_\theta = 1\big/\left(\sigma^2 \big\langle\ba_0,\bSigma^{-1}\ba_0\big\rangle \right).
\eel
In the linear model \eqref{LM}, the log-likelood function is
$\bb\to-\|\by - \bX\bb\|_2^2/(2\sigma^2)$ 
up to a constant term 
and the one-step log-likelihood correction (\ref{one-step})
can be explicitly written as a linear bias correction, 
\bel{LDPE}
\htheta  = \left\langle \ba_0, \hbinit \right\rangle + 
\frac{\left\langle \bz_0, \by -\bX \hbinit \right\rangle}{\|\bz_0\|_2^2}
\quad
\text{ with }
\quad
\bz_0 = \bX\bu_0.
\eel
Here, $\bz_0 = \bX\bu_0$ can be viewed as an efficient score vector for the 
estimation of $\theta$. 

In the case of unknown $\bSigma$, the efficient score 
vector $\bz_0$  has to be estimated from the data. 
For statistical inference of a preconceived regression coefficient $\beta_j$ or a linear 
combination of a small number of $\beta_j$, such one-step linear bias correction was considered in 
\cite{ZhangSteph14,BelloniCH14,Buhlmann13,GeerBR14,JavanmardM14a,javanmard2018debiasing} among others. 
The focus of the present paper is to find sharper sample size requirements,
in the case of 
Gaussian design, than the typical $n\gg (s_0\log p)^2$ required 
in the aforementioned previous studies. 
Here and in the sequel, 
\begin{equation}
    \label{support-S}
    s_0 = \big|S\big|\ \hbox{ with }\ S = \supp(\bbeta). 
\end{equation}

Our results study both known $\bSigma$---in that case the ideal
score vector $\bz_0$ can be used---and unknown $\bSigma$
where estimated score vectors $\hbz\approx\bz_0$ are used. 
The results of \cite{cai2017confidence} show that
for unknown $\bSigma$ with bounded condition number,
it is impossible to construct confidence intervals for $\theta = \ba_0^\top\bbeta$
with length of order {$n^{-1/2}\|\ba_0\|_2$} 
in the sparsity regime $s_0\ggg \sqrt n$.
Proposition 4.2 in \cite{javanmard2018debiasing} extends the
lower bound from \cite{cai2017confidence} 
to account for the sparsity and $\ell_1$ norm of $\bu_0$ in \eqref{u_0} 
as follows. Let $\Theta(s_0,s_\Omega,\rho)$ be the collection 
of all pairs $(\bbeta,\bSigma)$ such that
$\lambda_{\min}(\bSigma)^{-1} \vee \lambda_{\max}(\bSigma) \le c_0$
for some absolute constant $c_0 > 1$ and
$$\|\bSigma^{-1}\be_j\|_0 \le s_\Omega, 
\qquad 
\|\bbeta\|_0\le s_0,
\qquad
\|\bSigma^{-1}\be_j\|_1 \le 1.02\vee   \rho.$$
When $\ba_0=\be_j$ for a fixed canonical basis vector and
$s_0\le c_1 \min(p^{0.49},n/\log p)$, 
\bel{eq:lower-bound-Prop42-javanmard-montanari}
    && \sup_{(\bbeta,\bSigma)\in\Theta(s_0,s_\Omega,\rho)}
    \E_{\bbeta,\bSigma}\Big[n^{1/2}\sigma^{-1}|\hat \beta_j - \beta_j|
        \Big]
    \ge
    c_2+
    c_3 r_n(s_0,s_\Omega,\rho)
    \\ \nonumber
    &\text{ {with} }&
    r_n(s_0,s_\Omega,\rho) =
    \min\Big\{{\min(s_0,s_\Omega)}\log(p)n^{-1/2}
        ,{(\rho\vee 1.02)}\sqrt{\log p}
    \Big\}
\eel
for any estimator $\hat\beta_j$ as a measurable function of $(\by,\bX)$,
where $c_1,c_2,c_3>0$ are absolute constants. 

Hence the minimax rate of estimation of
$\beta_j$ over $\Theta(s_0,s_\Omega,\rho)$
is at least $\sigma n^{-1/2} (1+r_n(s_0,s_\Omega,\rho))$,
and {any} $(1-\alpha)$-confidence interval\footnote{(footnote)
Note that \eqref{eq:lower-bound-Prop42-javanmard-montanari} is stated slightly
differently than in \cite{javanmard2018debiasing}: It can be equivalently
stated as a lower bound on the expected length of $(1-\alpha)$-confidence
intervals for $\beta_j$ valid uniformly over $\Theta(s_0,s_\Omega,\rho)$ up to
constants depending on $\alpha$. This follows by picking as $\hat\beta_j$ any
point in the confidence interval, or by constructing a confidence interval from
an estimate $\hat\beta_j$ and its maximal expected length over
$\Theta(s_0,s_\Omega,\rho)$ by Markov's inequality.
}
for $\beta_j$ valid uniformly
over $\Theta(s_0,s_\Omega,\rho)$
must incur a length of order $\sigma n^{-1/2}(1+r_n(s_0,s_\Omega,\rho))$
up to {a constant} depending on $\alpha$.
Since the focus of the present paper is on efficiency results
and other phenomena for sparsity $s_0\ggg \sqrt n$,
these impossibility results from \cite{cai2017confidence,javanmard2018debiasing}
motivate either the known $\bSigma$ assumption 
(in \Cref{sec:2.2-known,sec-3-main-results} below)
or the sparsity assumptions on $\bSigma^{-1}\ba_0$
for unknown $\bSigma$ in \Cref{sec:unknown-Sigma} 
where we prove that the lower bound \eqref{eq:lower-bound-Prop42-javanmard-montanari}
is sharp. 
For known $\bSigma$, our analysis reveals that the de-biasing scheme \eqref{LDPE} needs
to be modified to enjoy efficiency in the regime $s_0 \ggg n^{2/3}$
when the initial estimator is the Lasso.
For unknown $\bSigma$, 
the modification of \eqref{LDPE} is also required for efficiency
when $s_0, s_\Omega$ satisfy the conditions in
\Cref{thm:unknown-sigma-lb} of
\Cref{sec:unknown-Sigma}.

The required modification of \eqref{LDPE} 
takes the form of a multiplicative adjustment to
account for the degrees-of-freedom of the initial estimator.
Interestingly, \cite{JavanmardM14b} proved that for the Gaussian design with known 
$\bSigma = \bI_{p\times p}$, the sample size $n \ge C s_0 \log(p/s_0)$ is sufficient 
in de-biasing the Lasso for the estimation of $\beta_j$ at the $n^{-1/2}$ rate.
More recently, \cite{javanmard2018debiasing} extended this result
and showed that 
$n \ge C s_0(\log p)^2$ is sufficient to de-bias the Lasso for 
the estimation of $\beta_j$ at the $n^{-1/2}$ rate for Gaussian designs with 
known covariance matrices $\bSigma$ 
when the $\ell_1$ norm of each column of $\bSigma^{-1}$ is bounded,
i.e., for some constant $\rho>0$
\begin{equation}
    \max_{j=1,...,p}\|\bSigma^{-1}\be_j\|_1 \le \rho
    \label{ell_1-norm-bounded}
\end{equation}
holds, where $(\be_1,...,\be_p)$ is the canonical basis in $\R^p$.
From this perspective, the present paper provides an extension of 
these results to more general $\bSigma$: We will see below 
that for $n \ge C s_0\log(p)^2$, the efficiency of the de-biasing scheme \eqref{LDPE}
is specific to assumption \eqref{ell_1-norm-bounded} and that
the de-biasing scheme \eqref{LDPE} requires a modification to be efficient in cases where \eqref{ell_1-norm-bounded}
is violated.

The paper is organized as follows.
\Cref{sec-2-approach-and-dof-adjustment} provides a description of our proposed estimator,
which is a modification of the de-biasing scheme \eqref{LDPE} that accounts for the degrees-of-freedom of the initial estimator.
\Cref{sec-3-main-results} describes our strongest results in linear regression with known covariance matrix for the Lasso.
This includes several efficiency results for the de-biasing scheme
modified with degrees-of-freedom adjustment
and a characterization of the asymptotic regime where this adjustment is necessary.
\Cref{sec-4-initial-bias-bounded} studies the specific situation
where bounds on the $\ell_1$ norm of $\bSigma^{-1}\ba_0$ are available,
similarly to \eqref{ell_1-norm-bounded} when $\ba_0$ is a 
canonical basis vector. 
The additional assumptions on $\bSigma^{-1}$ and the results of
\Cref{sec-4-initial-bias-bounded} explain why the necessity of degrees-of-freedom
adjustment did not appear in some previous works.
\Cref{sec-5-ell_infty} provides a new $\ell_\infty$ bound for estimation of $\bbeta$ by the Lasso
under assumptions similar to \eqref{ell_1-norm-bounded}.
    \Cref{section:efficiency} discusses efficiency and regularity,
    and shows that asymptotic normality remains unchanged under
    non-sparse $n^{-1/2}$-perturbations of $\bbeta$.
    \Cref{section:subgaussian} shows that the degrees-of-freedom adjustment
is also needed for certain non-Gaussian designs.
The proofs of the main results are given in 
\Cref{sec:proof,sec:proof-probabilistic-lemma,sec:proof-thm-infty-norm,sec:proof-th-1}. 
The proofs of intermediary lemmas 
and propositions can be found in 
\Cref{appendix:1,appendix:2,appendix:3,appendix:4}.
Our main technical tool is a carefully constructed Gaussian interpolation path described in \Cref{sec:proof-path}.

\section*{Notation}
We use the following notation throughout the paper. 
Let $\bI_d$ be the identity matrix of size $d\times d$, e.g. $d=n,p$. 
For any $p\ge 1$, let $[p]$ be the set $\{1,...,p\}$.
For any vector $\bv=(v_1,...,v_p)^\top \in\R^p$ and any set $A\subset [p]$,
the vector $\bv_A\in\R^{|A|}$ is the restriction $(v_j)_{j\in A}$.
For any $n\times p$ matrix $\bM$ with columns $(\bM_1,\ldots,\bM_p)$
and any subset $A\subset [p]$, let $\bM_A = (\bM_j, j\in A)$ be 
the matrix composed of columns of $\bM$ indexed by $A$, 
and $\bM_A^\dag$ be the Moore-Penrose generalized inverse of $\bM_A$. 
If $\bM$ is a symmetric matrix of size $p\times p$ and $A\subset [p]$,
then $\bM_{A,A}$ denotes the sub-matrix of $\bM$ with rows and columns in $A$,
and $\bM_{A,A}^{-1}$ is the inverse of $\bM_{A,A}$.
Let $\|\cdot\|_q$ denote the $\ell_q$ norm of vectors, 
$\|\cdot\|_{op}$ the operator norm (largest singular value) of matrices and $\|\cdot\|_F$ the Frobenius norm. 
We use the notation $\langle\cdot,\cdot\rangle$ for the canonical scalar product of vectors in 
$\R^n$ or $\R^p$, i.e., $\langle \ba,\bb\rangle  = \ba^\top \bb$ for two vectors $\ba,\bb$ of the same dimension. 

Throughout the paper, $C_0=\|\bSigma^{-1/2}\ba_0\|_2$, 
$\bu_0$ {is} as in \eqref{u_0} {and $F_\theta$ as in} \eqref{Fisher}.
The score vector $\bz_0$ is always defined as $\bz_0=\bX\bu_0$
and $\bQ_0$ is the matrix $\bQ_0 = \bI_{p\times p} - \bu_0\ba_0^\top$, so that
$$\bX = \bX\bQ_0 + \bz_0 \ba_0^\top$$ always holds.
{As in \eqref{support-S},}  
$S$ and $s_0$ are the support and number of nonzero coefficients
of the unknown coefficient vector $\bbeta$. 
For any event $\Omega$, denote by $I_\Omega$ its indicator function and $a_+=\max(0,a)$ for $a\in\R$.

\section{Degrees of freedom adjustment}
\label{sec-2-approach-and-dof-adjustment}

\subsection{Known $\bSigma$}
\label{sec:2.2-known}
In addition to the de-biasing scheme (\ref{LDPE}), we consider the following degrees-of-freedom adjusted version 
of it.  Suppose that the Lasso estimator $\lasso$ is used as 
the initial estimator $\hbinit$, where 
\bel{lasso}
\lasso = \argmin_{\bb\in\R^p}\left\{\|\by - \bX\bb\|_2^2/(2n) + \lam \|\bb\|_1\right\}. 
\eel
The degrees-of-freedom adjusted LDPE is defined as
\bel{LDPE-df}
\htheta_{\nu}  = \left\langle \ba_0, \lasso \right\rangle + 
\frac{\left\langle \bz_0, \by -\bX \lasso \right\rangle}{\|\bz_0\|_2^2(1-\nu/n)}, 
\eel
where $\bz_0$ is as in (\ref{LDPE}) and $\nu\in[0,n)$ is a degrees-of-freedom adjustment; $\nu$ is allowed to be random.
Our theoretical results will justify the degrees-of-freedom adjustment $\nu=|\Shat|$
where $\Shat = \supp(\lasso)$. 
The size of the selected model has the interpretation of degrees of freedom 
for the Lasso estimator
in the context of Stein's Unbiased Risk Estimate (SURE)
\cite{zou2007,Zhang10-mc+,tibshirani2012}. 

We still retain other possibilities for $\nu$ such as $\nu=0$ in order to
analyse the unadjusted de-biasing scheme (\ref{LDPE}). 
With some abuse of notation, in order to avoid any ambiguity we may sometimes use
the notation $\htheta_{\nu=0}$ for the unadjusted \eqref{LDPE}
and $\htheta{}_{\nu=|\Shat|}$ for \eqref{LDPE-df} with 
$|\Shat|$ 
being the size of the support of the Lasso.

Our main results will be developed in \Cref{sec-3-main-results}.
Here is a simpler version of the story.  

\begin{theorem}
    \label{thm:teaser}
Let $s_0,n$ and $p$ be positive integers satisfying $p/s_0\to\infty$ 
and $(s_0/n)\log(p/s_0) \to 0$. 
Assume that $\bSigma_{jj}\le1$ for all $j\in[p]$ and that 
the spectrum of $\bSigma$ is uniformly 
bounded away from 0 and $\infty$; e.g. $\max(\|\bSigma\|_{op},\|\bSigma^{-1}\|_{op})\le 2$.
Let $\lambda = 1.01\sigma\sqrt{2\log(8 p/s_0)/n}$.

(i) Then $|\Shat| = O_{\P}(s_0) = o_{\P}(n)$ and for $\nu=|\Shat|$ we have for every $\ba_0$
\begin{equation}
 \sqrt{nF_\theta} \left( 1- |\Shat|/n\right)\left(\htheta_{\nu = |\Shat|}-\theta\right) = T_n + o_\P(1)
\label{teaser-asymptotic-efficiency}
\end{equation}
where $T_n = \sqrt{nF_\theta}\langle \bz_0,\bep\rangle/\|\bz_0\|_2^2$ 
has the $t$-distribution with $n$ degrees of freedom.
Thus the estimator (\ref{LDPE-df}) enjoys asymptotic efficiency 
when $\nu = |\Shat|$.

(ii) {The} quantity
$\sqrt{n F_\theta} (\htheta_{\nu=0}-\theta) - T_n$ is unbounded 
for certain $\bbeta$ satisfying $n/\log(p/s_0) \ggg s_0 \ggg n^{2/3}/\log(p/s_0)^{1/3}$ 
and $\ba_0$ depending on $S$ and $\bSigma$ only. 
Consequently, the unadjusted \eqref{LDPE} cannot be efficient.
\end{theorem}

\Cref{thm:teaser}(ii) implies that with $\nu=0$,
the unadjusted \eqref{LDPE} cannot be efficient
in the whole range $\{s_0: s_0\log(p/s_0)\lll n\}$ of sparsity levels
unless extra assumptions are made on the covariance matrix
$\bSigma$ such as \eqref{ell_1-norm-bounded}.
\Cref{thm:teaser}(i) shows that using the adjustment
$\nu=|\Shat|$ repairs this: The efficiency in \eqref{teaser-asymptotic-efficiency} then
holds in the whole range $\{s_0: s_0\log(p/s_0)\lll n\}$ of sparsity levels.
\Cref{thm:teaser}(i) is proved after \Cref{corollary:well-adjusted-|Shat|} below
while (ii) is a consequence of 
{the following proposition}.

\begin{restatable}{proposition}{propositionUnbounded}
    \label{prop:unbounded}
{Let the setting and assumptions of \Cref{thm:teaser} be fulfilled
    and let $\nu$ be a random variable with $\nu\in[0, n)$ almost surely.
Then}
\bel{second-term-lower-bound}
     \sqrt{nF_\theta}(\htheta_{\nu}-\theta)
  &=&
  T_n + o_\P(1) +  n^{-1}(\nu-|\Shat|) \Lambda_\nu
\eel
where
$\Lambda_\nu 
= \sqrt{nF_\theta}(1-\nu/n)^{-1}(1-|\Shat|/n)^{-1} \|\bz_0\|^{-2} \langle \bz_0,\by-\bX{\lasso}\rangle.
$
Furthermore 
for any $(s_\Omega,s_0)$ with
    $s_\Omega\le s_0 =o(n/\log(p/s_0))$, 
    and any $\ba_0$ with $\|\bSigma^{-1/2}\ba_0\|_2=1$ and 
    $\|\bSigma^{-1}\ba_0\|_0=s_\Omega$, 
    there exists $\bbeta$ with $\|\bbeta\|_0=s_0$ such that
\begin{equation}
\P\bigl[
    |\Lambda_{\nu}| \ge \|\bSigma^{-1}\ba_0\|_1 
    \sqrt{\log(8p/s_0)}
\bigr]\to 1,
\quad
\P\bigl[|\Shat|\ge s_0\bigr]\to 1. 
\label{eq:Lambda_nu-approaching-one-conclusion}
\end{equation}
In particular, it is possible to pick $\ba_0$ satisfying in addition 
    $\|\bSigma^{-1}\ba_0\|_1\ge s_\Omega^{1/2}/\|\bSigma\|_{op}^{1/2}$.
\end{restatable}

\Cref{prop:unbounded} is proved in 
\Cref{sec:proofs-unknown-Sigma}.
\Cref{thm:teaser}(ii) is implied by \Cref{prop:unbounded}
with $s_\Omega = s_0$:
If $s_0^{3/2}\sqrt{\log(8p/s_0)}/n\to +\infty$
then $n^{-1}|\Shat|\Lambda_{\nu=0}$ is unbounded with probability
approaching one by \eqref{eq:Lambda_nu-approaching-one-conclusion},
while the other terms in \eqref{second-term-lower-bound}
are stochastically bounded.

{
\begin{example}
    \label{example:e_j}
It is informative to unpack from the proof of
\Cref{prop:unbounded} how $(\ba_0,\bu_0,\bbeta)$ is constructed
so that \eqref{eq:Lambda_nu-approaching-one-conclusion} holds.
\Cref{thm:teaser}
and \Cref{prop:unbounded} apply to any $\bSigma$ with bounded
spectrum and $\Sigma_{ii}\le1$. Let $\bbeta$ be an $s_0$-sparse vector
with large enough non-zero coefficients and $\beta_j > 0$
for some index $j=1,...,p$ that is fixed throughout this example.
Then let $\bv\in \{-1,0,1\}^p$ be an $s_\Omega$-sparse vector
with $\supp(\bv)\subset \supp(\bbeta)$
and $v_k=\sgn(\beta_k)$ for all $k\in\supp(\bv)$.
Consider
\begin{equation*}
    \bOmega = \bI_p +  (1/4) s_\Omega^{-1/2}[\be_j \bv^\top + \bv \be_j^\top].
\end{equation*}
which has bounded spectrum since $\|\bOmega - \bI_p\|_{op} \le 1/2$
and set $\bSigma = \bOmega^{-1} \kappa$ for some constant $\kappa>0$
such that $\max_{j=1,...,p}\Sigma_{jj} = 1$.
Since $\bOmega$ has bounded spectrum, $\kappa$ is also bounded
and the spectrum of $\bSigma$ is bounded as required.
From the proof of \Cref{prop:unbounded}, we see
that the requirement for
$\bu_0$ is that
\begin{equation}
\langle \bu_0, \sgn(\bbeta) \rangle = \|\bu_0\|_1
\label{eq:condition-u0}
\end{equation}
must hold.
For the $\bSigma$ just defined, set
$\bu_0 
=
\bOmega \be_j
=
\bigl(1+(1/4)s_\Omega^{-1/2}v_j\bigr)\be_j
+ (1/4) s_{\Omega}^{-1/2}\bv$.
Since $\bbeta$ was chosen with $\beta_j > 0$, we have $v_j\ge0$
by definition of $\bv$ and $\bu_0$
satisfies \eqref{eq:condition-u0}.
These quantities $(\bSigma,\bbeta,\bu_0)$, when $\bbeta$
has large enough coefficients, satisfy
\eqref{eq:Lambda_nu-approaching-one-conclusion}
by the proof of \Cref{prop:unbounded}.
Finally, from \eqref{u_0} there is a one-to-one correspondence
between $\bu_0$ and $\ba_0$ given by
$\ba_0 = \bSigma \bu_0 / \langle \bu_0, \bSigma \bu_0\rangle$.
This implies $\ba_0
= \bOmega^{-1} \bu_0 /\langle \bu_0,\bOmega^{-1}\bu_0\rangle
$ and since $\bu_0=\bOmega\be_j$, the direction
$\ba_0$ for this example is proportional to the canonical basis vector $\be_j$.
\Cref{prop:unbounded} thus proves 
the necessity of the degrees-of-freedom adjustment
with $\ba_0$ proportional to $\be_j$.
\Cref{fig-e_j} illustrates this phenomenon on simulated data.
\end{example}
}

The adjustment in (\ref{LDPE-df}) was proposed by \cite{JavanmardM14b} in the form of 
\bel{est-JM} 
\hbbeta = \lasso + \frac{\bSigma^{-1}\bX^\top(\by -\bX \lasso)}{n-\nu}
\eel
based on 
heuristics of the replica method from statistical physics and a theoretical justification 
in the case of 
$\bSigma = \bI_p$. 
As $\bz_0 = \bX\bu_0$ with $\bu_0=\bSigma^{-1}\ba_0/\big\langle\ba_0,\bSigma^{-1}\ba_0\big\rangle$ in (\ref{LDPE}), 
$\E\|\bz_0\|_2^2/n 
= 1/\big\langle \ba_0,\bSigma^{-1}\ba_0\big\rangle$ and 
\bes
\frac{\left\langle \bz_0, \by -\bX \hbinit \right\rangle}{\big(\E \|\bz_0\|_2^2\big)(1-\nu/n)} 
= \left\langle \ba_0, \frac{\bSigma^{-1}\bX^\top(\by -\bX \lasso)}{n-\nu} \right\rangle. 
\ees
Thus, the plug-in estimator
\bel{htheta-JM}
\htheta_{\nu}  = \left\langle \ba_0, \hbbeta \right\rangle\ \hbox{with the $\hbbeta$ in (\ref{est-JM})}, 
\eel
is equivalent to replacing $\|\bz_0\|_2^2$ with its expectation 
in the denominator of the bias correction term 
in (\ref{LDPE-df}). 
Another version of the estimator, akin to the version of the LDPE proposed in \cite{ZhangSteph14}, is 
\bel{est-ZZ}
\htheta_{\nu}  = \left\langle \ba_0, \lasso \right\rangle + 
\frac{\left\langle \bz_0, \by -\bX \lasso \right\rangle}
{\big\langle \bz_0, \bX\bu\big\rangle (1-\nu/n)}
\eel
with a vector $\bu\in\R^p$ satisfying $\langle \bu,\ba_0\rangle=1$. 
Since $\E \big\langle \bz_0, {\bX\bu}\big\rangle =\E \|\bz_0\|_2^2$, 
the estimator \eqref{est-JM} 
{also} 
corresponds to \eqref{est-ZZ} with $\langle \bz_0,{\bX\bu}\rangle$ replaced
by its expectation in the denominator of the bias correction term.

Let $\bhlasso = (\lasso-\bbeta)$.
It is worthwhile to mention here that when $\|\bX\bhlasso\|_2/\sqrt n = o_\P(1)$
based on existing results on the Lasso, 
the asymptotic distribution of (\ref{LDPE-df}) adjusted at the $n^{-1/2}$ rate
does not change 
when $\|\bz_0\|_2^2$ is replaced by a quantity of type $\|\bz_0\|_2^2(1+O(n^{-1/2}))$ 
in the denominator of the bias correction term. 
Indeed, 
\bel{hat-theta-equivalence}
&& \sqrt{n F_\theta}
\bigg|
\frac{\left\langle \bz_0, \by -\bX \lasso \right\rangle}{\|\bz_0\|_2^2(1-\nu/n)}
-
\frac{\left\langle \bz_0, \by -\bX \lasso \right\rangle}{\|\bz_0\|_2^2(1+O(n^{1/2}))(1-\nu/n)}
\bigg| \\
&\le&
O(1)(1-\nu/n)^{-1}\left(
    |T_n| n^{-1/2}
+
\|\bX\bhlasso\|_2/(\sigma C_0\|\bz_0\|_2)
\right).
\nonumber
\eel
The right-hand side converges to 0 in probability
if $(1-\nu/n)^{-1}=O_\P(1)$
and
$\|\bX\bhlasso\|_2/\sqrt n = o_\P(1)$
since $T_n$ has the $t$-distribution with $n$ degrees of freedom.
Thus, as (\ref{LDPE-df}), (\ref{htheta-JM}) and (\ref{est-ZZ}) are asymptotically equivalent, 
the most notable feature of these estimators is the degrees-of-freedom adjustment 
with the choice $\nu = |\Shat|$, as proposed in  \cite{JavanmardM14b}, compared with 
earlier proposals with $\nu=0$. 
While the properties of these estimators for general $\bbeta$ and $\bSigma$ will be studied in the 
next section, we highlight in the following theorem the requirement of either a degrees-of-freedom 
adjustment or some extra condition on the bias of the Lasso
in the special case where the Lasso is sign consistent. 

\begin{restatable}{theorem}{thmNecessityDofGaussianDesign}\label{th-1}
Suppose that the Lasso is sign consistent in the sense of 
\bel{selection}
\P\Big\{\sgn(\lasso) = \sgn(\bbeta)\Big\}\to 1. 
\eel
Let $C_0=\|\bSigma^{-1/2}\ba_0\|_2$ and $C_{\bbeta}= \|\bSigma_{S,S}^{-1/2}\sgn(\bbeta_S)\|_2/\sqrt{s_0}$. 
Suppose that $\sqrt{(1\vee s_0)/n}+C_{\bbeta}\sqrt{s_0}(\lam/\sigma)\big) \le \eta_n$
for a sufficiently small $\eta_n<1$.
Let $F_\theta = 1/(\sigma C_0)^2$ be the Fisher information as in (\ref{Fisher}), and 
$T_n = \sqrt{nF_\theta}\langle \bz_0,\bep\rangle/\|\bz_0\|_2^2$ so that $T_n$ has the $t$-distribution with 
$n$ degrees of freedom. 
Let $\htheta_{\nu}$ be as in (\ref{LDPE-df}) or (\ref{htheta-JM}). 
Then, 
\bel{th-1-1}
(1-\nu/n)\sqrt{n F_\theta}\big(\htheta_{\nu} - \theta\big) = T_n + O_{\P}(\eta_n)
\eel
for a random variable $\nu\in [0,s_0]$ if and only if 
\bel{th-1-2a}
\sqrt{F_\theta/n}\big(s_0 - \nu\big)\big\langle \ba_0, \lasso - \bbeta \big\rangle 
= O_{\P}(\eta_n), 
\eel
if and only if 
\bel{th-1-2}&& 
\sqrt{F_\theta/n}\big(s_0 - \nu\big)\big\langle (\ba_0)_S, \lam(\bX_S^\top\bX_S/n)^{-1}\sgn(\bbeta_S)\big\rangle 
= O_{\P}(\eta_n). 
\eel
The conclusion also holds for the $\htheta_{\nu}$ in (\ref{est-ZZ}) 
when {$C_0\|\bSigma^{1/2}\bu\|_2=O(1)$.}
\end{restatable}

The proof is given in 
\Cref{sec:proof-th-1}.
{
    \Cref{th-1} provides an alternative negative result,
    similar in flavor to \Cref{thm:teaser}(ii) and
    \Cref{prop:unbounded} above. The settings may not match exactly
    since the tuning parameter $\lambda$
    required for sign consistency is larger than the one featured
    in \Cref{thm:teaser}.
    Compared with \Cref{prop:unbounded}, the sign consistency lets
    us derive the two explicit conditions \eqref{th-1-2a}-\eqref{th-1-2}
    for efficiency that are useful to pinpoint situations, such as those
    described in the next two paragraphs, where efficiency does not hold.
}

Theorem \ref{th-1} implies that for efficient statistical inference of $\theta$ at the $n^{-1/2}$ rate, 
the unadjusted de-biasing scheme \eqref{LDPE} requires either a degrees-of-freedom 
adjustment or the extra condition that the bias of the initial Lasso estimator of $\theta$, 
{given by $\big\langle (\ba_0)_S, \lam(\bX_S^\top\bX_S/n)^{-1}\sgn(\bbeta_S)\big\rangle$,} 
is of order
$o_{\P}(n^{1/2}/s_0)$, even when the initial Lasso estimator is sign-consistent. 
For example, if $(\ba_0)_{S^c}=0$ and 
$(\ba_0)_S = \sgn(\bbeta_S)/\|\bSigma^{-1/2}\sgn(\bbeta_S)\|_2$, 
then $\ba_0$ is standardized with $\|\bSigma^{-1/2}\ba_0\|_2=1$ and condition (\ref{th-1-2}) on the bias 
can be written as 
\begin{equation*}
(\lam/\sigma)n^{-1/2}(s_0-\nu)\|\bSigma_{S,S}^{-1/2}\sgn(\bbeta_S)\|_2=O_\P(\eta_n)
\end{equation*}
because the singular values of the Wishart matrix 
$\bSigma_{S,S}^{-1/2}(\bX_S^\top\bX_S/n) \bSigma_{S,S}^{-1/2}$
are bounded away from 0 and $+\infty$ with high probability.
For $\nu=0$, this is equivalent to 
$C_{\bbeta} (\lam/\sigma) s_0^{3/2}/\sqrt n = O(\eta_n)$.
If $C_{\bbeta}$ is of order of a constant and $\eta_n< 1$, 
this implies that
the unadjusted de-biasing scheme \eqref{LDPE}
cannot be efficient in the asymptotic regime when 
\begin{equation}
    (\lam/\sigma) s_0^{3/2}/\sqrt n \ggg 1.
    \label{regime-unadjusted-cannot-be-efficient}
\end{equation}
Interestingly, the condition $(\lam/\sigma) s_0^{3/2}/\sqrt n =O(1)$ 
is weaker than 
the typical sample size requirement $n\gg (s_0\log p)^2$ in the case of unknown $\bSigma$. 

Another enlightening situation is
$\ba_0 = \be_j$ the $j$-th canonical basis vector for some $j\in S$,
    $S=\{1,...,s_0\}$ and $\bSigma^{-1}$ diagonal by block with two blocks
    $$
    \bSigma^{-1} = \begin{pmatrix}
        \bI_{|S|} + 
        (1/4) |S|^{-1/2}
        [\sgn(\bbeta_S) \be_j^\top + \be_j \sgn(\bbeta_S)^\top]& \mathbf{0}_{|S|,p-|S|} \\
    \mathbf{0}_{p-|S|, |S|} & \bI_{p-|S|}
    \end{pmatrix}.
    $$
    The eigenvalues of $\bSigma^{-1/2}$ belong to $[1/2,3/2]$ by construction
    since $(|S|)^{-1/2} \sgn(\bbeta_S) \be_j^\top$ has operator norm equal to one. Again using properties of
the singular values of the Wishart matrix 
$\bSigma_{S,S}^{-1/2}(\bX_S^\top\bX_S/n) \bSigma_{S,S}^{-1/2}$,
the left hand-side of condition \eqref{th-1-2} is of order
$$\sqrt{F_\theta/n}(s_0 - \nu) \lambda
\be_j^\top(\bSigma_{S,S})^{-1} \sgn(\bbeta_S) 
\asymp
(\lambda/\sigma)n^{-1/2} (s_0-\nu) \sqrt{s_0}.
$$
Similarly to the previous paragraph,
this implies that with $\nu=0$
the unadjusted de-biasing scheme \eqref{LDPE}
cannot be efficient 
if \eqref{regime-unadjusted-cannot-be-efficient} holds.
Up to a multiplicative constant in $\bSigma$, this example is similar
to \Cref{example:e_j} with $s_\Omega=s_0$. 

The novelty of our contributions 
resides in the $s_0^2\ggg n$ regime up to logarithmic factor,
in the sparsity range where the transition \eqref{regime-unadjusted-cannot-be-efficient} happens.
The necessity of the degrees-of-freedom adjustment can be seen 
in simulated data {as follows.} 
\Cref{fig:comparison} presents the distribution of
$\sqrt n(\hat\theta_\nu - \theta)$ with and without the adjustment
for $\bSigma=\bI_p,\sigma=1$ for $(n,p)=(4000,6000)$ and $s_0=20,40,80,120$.
Although
classical results on de-biasing in the regime $s_0^2\lll n$ 
proves that $\sqrt n(\hat\theta_\nu - \theta)\approx N(0,1)$
\cite{ZhangSteph14,JavanmardM14a,GeerBR14}
with $\nu=0$, simulations reveal that
$\sqrt n(\hat\theta_\nu - \theta)$ is substantially biased (downward in \Cref{fig:comparison}),
and any confidence interval constructed from $\sqrt n(\hat\theta_\nu -
\theta)\approx N(0,1)$ would not correctly control Type-I error
due to this substantial bias.
This substantial bias is present for sparsity as small as $s_0=20$
(for which $s_0^2/n = 0.1$).
On the other hand, the adjustment $\nu=|\Shat|$ repairs this,
as shown both in the simulation in \Cref{fig:comparison}
and by the theory in \Cref{thm:teaser} and in the next sections.
Thus our novel results on the necessity of the degrees-of-freedom adjustment is not only 
{theoretical; It} 
also explains the gap between
simulations and the predictions
from the early literature on de-biasing 
\cite{ZhangSteph14,JavanmardM14a,GeerBR14} 
{where the degrees-of-freedom adjustment is absent.} 

\begin{figure}
\begin{subfigure}[b]{0.475\textwidth}
        \centering
        \includegraphics[width=1.0\textwidth]{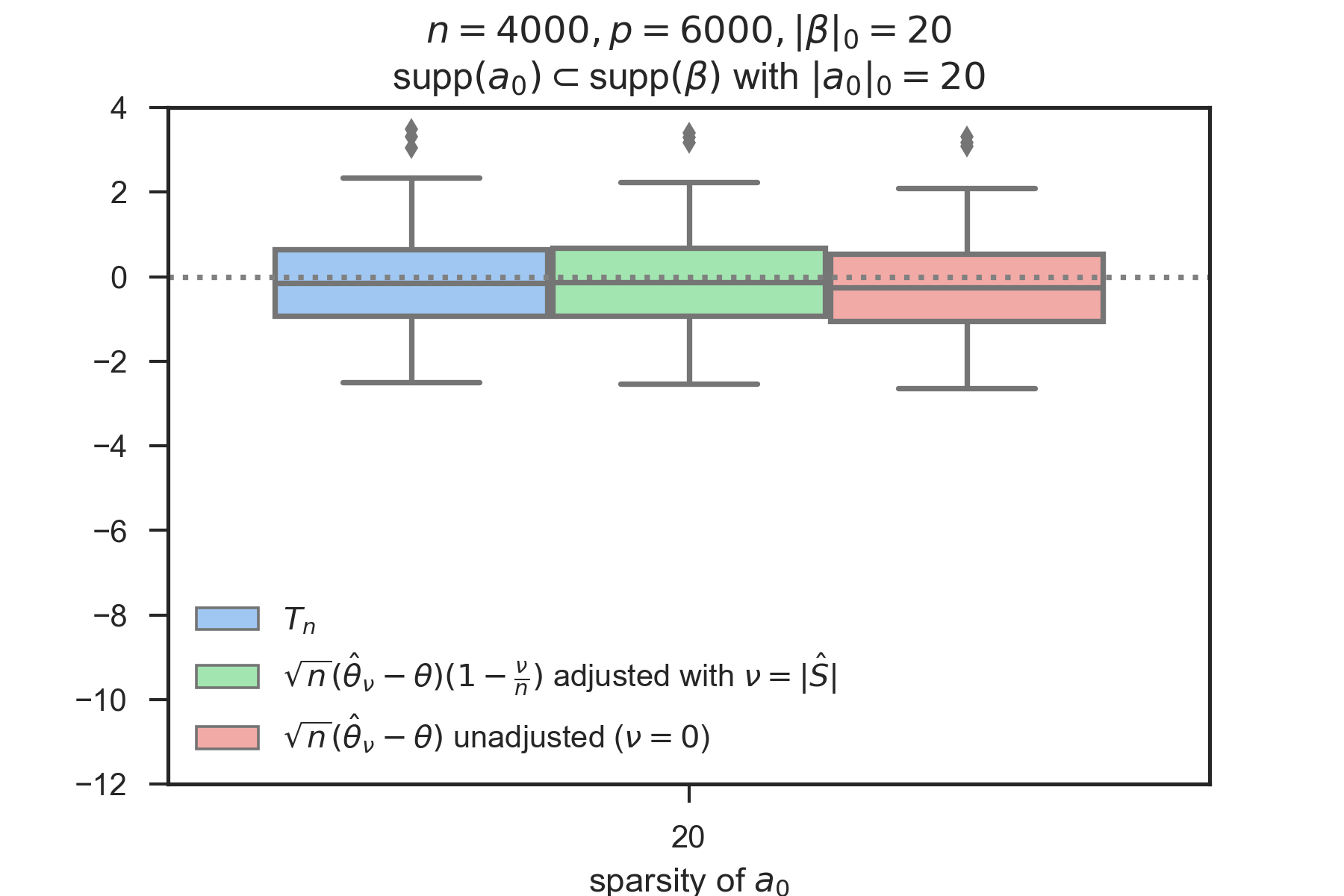}
\end{subfigure}
\begin{subfigure}[b]{0.475\textwidth}
        \centering
        \includegraphics[width=1.0\textwidth]{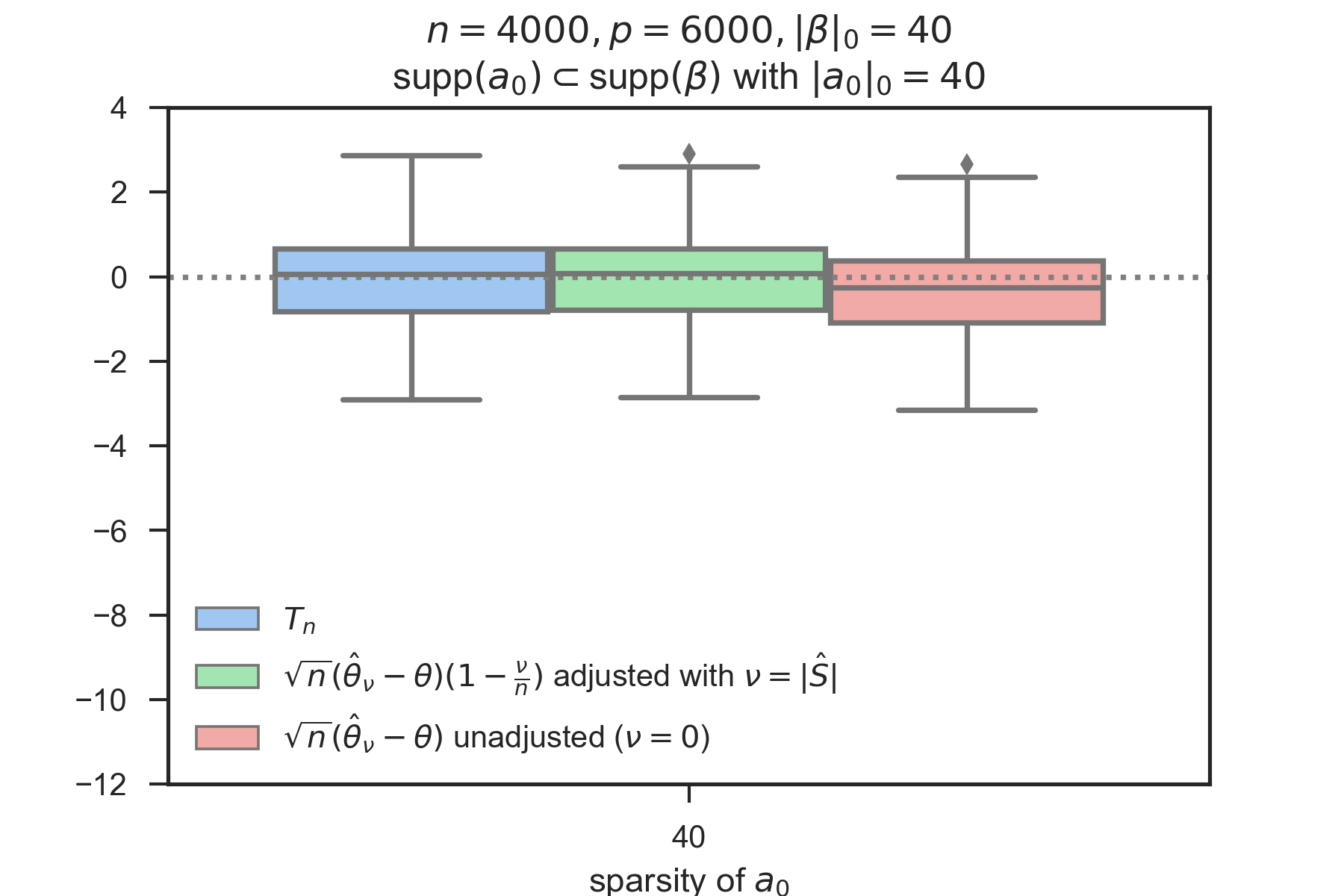}
\end{subfigure}
\vspace{0.3in}
\begin{subfigure}[b]{0.475\textwidth}
        \centering
        \includegraphics[width=1.0\textwidth]{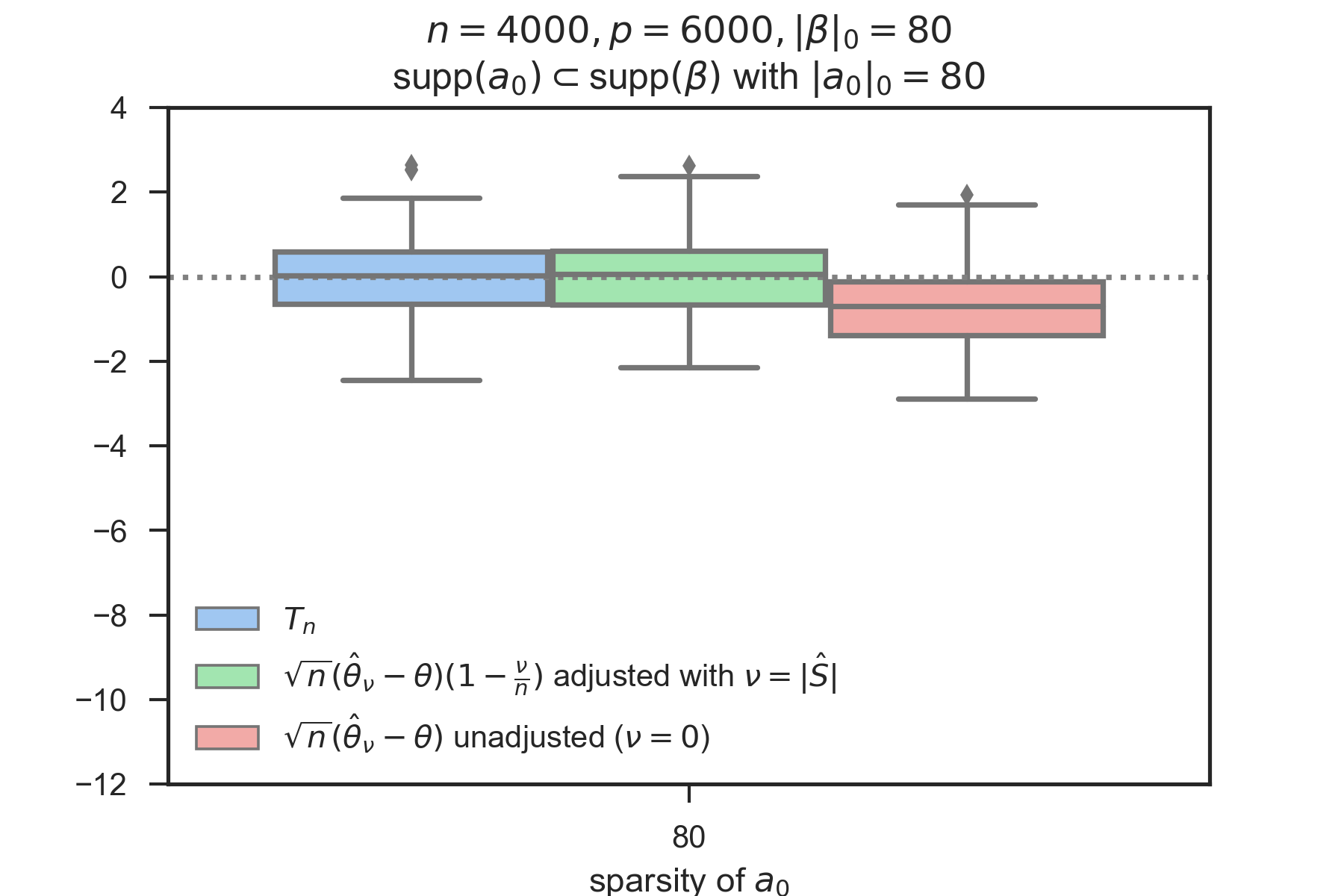}
\end{subfigure}
\begin{subfigure}[b]{0.475\textwidth}
        \centering
        \includegraphics[width=1.0\textwidth]{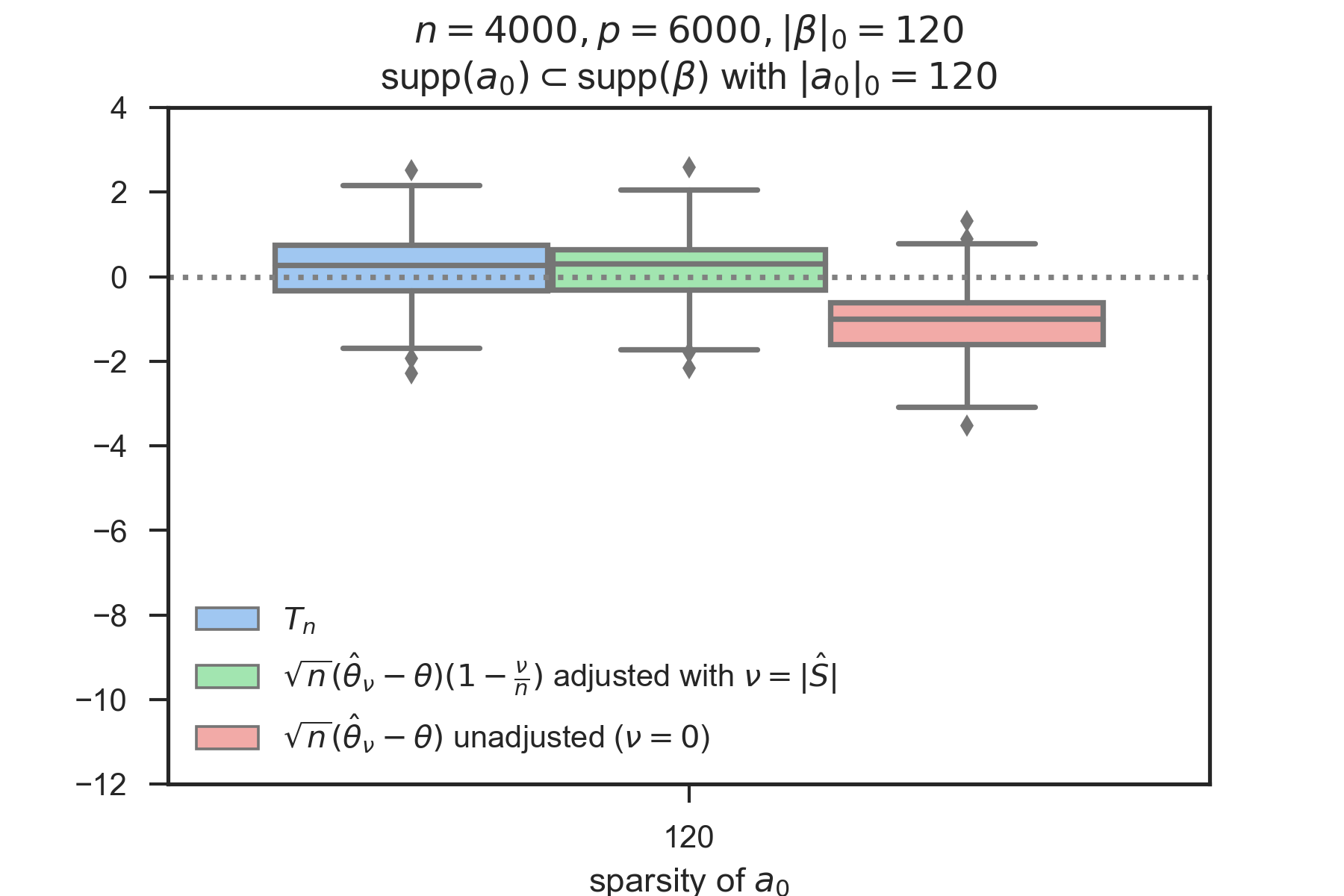}
\end{subfigure}
    \caption{\scriptsize
        Distribution of $\sqrt n(\hat\theta_\nu - \theta)$ in the adjusted ($\nu = |\Shat|$) and unadjusted $\nu = 0$ cases.
        For comparison,
        $T_n$ has the t-distribution with 
        $n$ degrees-of-freedom. Here $\ba_0$ is proportional to $\sgn(\bbeta)$
        normalized with $\|\bSigma^{-1/2}\ba_0\|_2=1$.
        Experiments were replicated 200 times. A two-sided t-test rejects
        that the mean of $\sqrt{n}(\hat\theta_\nu - \theta)$
        is zero in the unadjusted  case $(\nu = 0)$ with p-value $0.0048$ for $s_0=20$,
        p-value $0.00028$ for $s_0=40$, p-value $7\cdot 10^{-22}$ for $s_0=80$
        and p-value $2\cdot 10^{-31}$ for $s_0=120$.
        \label{fig:comparison}
}
\end{figure}

Guided by \Cref{th-1}, one can easily exhibit situations
with correlated $\bSigma$
and $\ba_0$ proportional to $\be_j$ (a canonical basis vector),
such that
the unadjusted estimate leads to spurious inference: 
One just needs to find problem instances
such that \eqref{th-1-2} is large.
As an example, \Cref{fig-e_j} shows boxplots
of the situation with $\ba_0=\be_j/(\bSigma^{-1})_{jj}^{1/2}$
sparsity $s_0 = \|\bbeta\|_0= 120$, $p=6000$, $n=4000$, $\sigma=1$ and 
$\bSigma$ is correlated
of the form $\bSigma^{-1}=\bI_p + 0.07(\sgn( \bbeta ) \be_j^\top + \be_j \sgn( \bbeta )^\top)$. In the un-adjusted case,
the pivotal quantity $n^{1/2}(\hat \theta_{\nu=0}-\theta)$
is biased downward and would produce misleading confidence intervals
with incorrect coverage.
The adjustment $\nu=|\Shat|$ exactly repairs this.

\begin{figure}[ht]
    \includegraphics[width=3.5in]{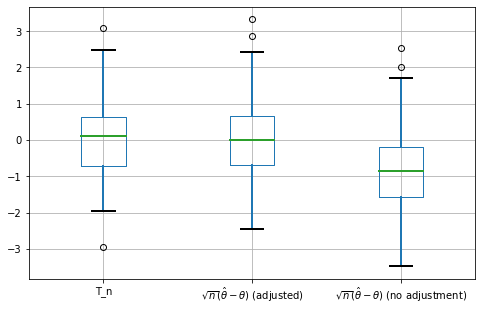}.
    \caption{\scriptsize
        Boxplots of pivotal random variables $\sqrt n(\hat\theta_n -\theta)$
        for $\nu=0$ (unadjusted) and $\nu=|\Shat|$ (adjusted) 
        when
        $\ba_0=\be_j/(\bSigma^{-1})_{jj}^{1/2}$,
$s_0 = \|\bbeta\|_0= 120$, $p=6000$, $n=4000$, $\sigma=1$ and 
$\bSigma^{-1}=\bI_p + 0.07(\sgn( \bbeta ) \be_j^\top + \be_j \sgn( \bbeta )^\top)$.
For comparison, $T_n$ has the t-distribution with $n$ degrees-of-freedom.
    \label{fig-e_j}
    }
\end{figure}

\Cref{th-1} requires sign consistency of the Lasso in \eqref{selection}. 
Sufficient conditions for the sign consistency of the Lasso were given in 
\cite{MeinshausenB06,Tropp06,ZhaoY06,Wainwright09}. 
In particular, \cite{Wainwright09} gave the following sufficient conditions for (\ref{selection})
in the case of linear regression (\ref{LM}) with Gaussian design: 
For certain positive $\gamma$, $\delta$ and $\phi_p\ge 2$, 
\bes
& \big\|\bSigma_{S^c,S}\bSigma_{S,S}^{-1}\sgn(\bbeta_S)\big\|_\infty \le 1 - \gamma,
\cr & \lam = \gamma^{-1}\sigma\sqrt{\phi_p\rho(2/n)\log p},
\cr & \rho(C_{\min}\gamma^2)^{-1}(2s_0/n)\log(p-s_0) 
+(\phi_p\log p)^{-1}\log(p-s_0)  < 1 - \delta,
\ees
with $\rho = \max_{j\in S^c}\big(\bSigma_{j,j} - \bSigma_{j,S}\bSigma_{S,S}^{-1}\bSigma_{S,j} \big)$ and 
$C_{\min} = \min_{\|\bu\|_2=1}\big\|\bSigma_{S,S}^{-1/2}\bu\big\|_2$, and 
\bes
\hbox{$\min_{j\in S}$}|\beta_j| \ge \big(1+n^{-1/2}c_n\big)\lam 
\max_{\|\bu\|_\infty =1}\big\|\bSigma^{-1/2}_{S,S}\bu\big\|_\infty^2
+ 20\big(\sigma^2(C_{\min}n)^{-1}\log s_0\big)^{1/2}, 
\ees
for some $c_n\to \infty$. 

\subsection{
    Unknown $\bSigma$
}

\label{sec:unknown-Sigma}

In the case of unknown {$\bu_0$,} one needs to estimate the ideal score vector
$\bz_0=\bX\bu_0$ as well as the variance level $\|\bz_0\|^2/n$ in \eqref{LDPE}. 
{In view of \eqref{u_0}, {we} consider 
\begin{equation}\label{bQ}
\bz = \bX\bu,\  \bQ = \bI_p - \bu\ba_0^\top\ \hbox{ with $\bu$ satisfying}\ \langle\bu,\ba_0\rangle = 1.
\end{equation}
As $\bQ^2=\bQ$, by algebra and the definitions of $\bu_0$ and $\bz_0$ in \eqref{u_0} and \eqref{LDPE}, 
\bel{eq:LM-gamma}
\bz = - \bX\bQ \bu_0  + \bz_0= \bX\bQ \bgamma  + \bz_0
\eel
with $\bgamma = - \bQ\bu_0$ 
and $\E\big[(\bX\bQ)^\top\bz_0\big] = \bQ^\top\bSigma\bu_0 = {\bf 0}$. Hence,} 
\eqref{eq:LM-gamma} is a linear model
with response vector ${\bz}\in\R^n$,
Gaussian design matrix  $\bX{\bQ}\in\R^{n\times p}$
with $n$ independent rows,
true coefficient vector $\bgamma$,
and Gaussian noise
$\bz_0 \sim N({\bf 0 }, C_0^{-2} \bI_n)$ independent of $\bX{\bQ}$,
where {$C_0 = \|\bSigma^{-1/2}\ba_0\|_2$.} 
Note that since $\bQ$ is rank deficient, 
    the linear model  \eqref{eq:LM-gamma} is unidentifiable:
    For both $\tbgamma=-\bu_0$ and $\bgamma=-\bQ \bu_0$
    we have $\bX\bQ \tbgamma = \bX\bQ \bgamma$ so that both $\tbgamma,\bgamma$
    can be regarded as the true coefficient vector in the model \eqref{eq:LM-gamma}.
    To solve this identifiability issue, we view the parameter
    space of \eqref{eq:LM-gamma} as the image of $\bQ$
    and the true coefficient vector as $\bgamma=-\bQ\bu_0$. 

It is thus natural to estimate {$\bz_0$} in the linear model \eqref{eq:LM-gamma},
as was already suggested previously for $\ba_0=\be_j$
\cite{ZhangSteph14,GeerBR14,javanmard2018debiasing}.
Given an estimator {$\hbgamma$ of $\bgamma$,} 
we define the estimated score vector 
\begin{equation}
    \label{estimated-score-vector-z}
    \hbz = {\bz - \bX\bQ\hbgamma} 
\end{equation}
and the corresponding de-biased estimate
\begin{equation}
    \label{de-biased-estimate-estimated-score-z}
    \hat\theta_{\nu,\hbz}
    =
    \langle \ba_0,\lasso\rangle
    + \frac{\langle \hbz, \by-\bX\lasso\rangle}{{(1-\nu/n)\langle \hbz, \bz\rangle}}.
\end{equation}
This corresponds to \eqref{est-ZZ} with the ideal score vector $\bz_0$
replaced by $\hbz$. 

    The vector $\bu$ in \eqref{bQ} that defines the linear model
    \eqref{eq:LM-gamma}
    should be picked carefully to yield small prediction error
    $\|\bz_0-\hbz\|_2^2/n=\|\bX\bQ(\hbgamma-\bgamma)\|_2^2/n$ 
    in the linear model \eqref{bQ}. 
As $\bz_0=\bX\bu_0$ with a high-dimensional $\bu_0$, 
it would be reasonable to {expect that} a sparsity condition on $\bu_0$ 
would ensure proper convergence of $\hbz$ to $\bz_0$. 
However, this requires a connection between the sparsity of 
$\bgamma = - \bQ\bu_0$ to that of $\bu_0$.
To this end, we pick 
\bel{bu}
\bu = \bfe_{j_0}/(a_0)_{j_0}\ \hbox{with }\ j_0 = \argmax_{j=1,\dots,p} \big|(a_0)_j\big|. 
\eel
For the above choice of $\bu$, 
\bel{bu-properties}
&& \langle\bu,\ba_0\rangle = 1,\quad \|\bQ\bh\|_0\le 1+\|\bh\|_0,\quad \|\bQ\bh\|_1\le 2\|\bh\|_1,\quad \forall\ \bh\in\R^p,  
\eel
so that the sparsity of $\bu_0$ implies that of $\bgamma$.
This leads to the Lasso estimator 
\begin{equation}
\label{gamma-hat}
    {\hbgamma = \bQ \hbb\ 
    \text{ with }}\ 
    \hbb = \argmin_{\bb\in\R^p}\Big\{\|\bz - \bX\bQ\bb\|_2^2/(2n) + \htau A\lamuniv\|\bQ\bb\|_1\Big\} 
\end{equation}
where $\lamuniv = \sqrt{(2/n)\log p}$, $A$ is an upper bound for $\max_{j=1,\dots,p} \|\bX\bQ\bfe_j\|_2/n^{1/2}$ and 
$\htau$ is an estimate of the noise level $C_0^{-1}$ in the regression model \eqref{eq:LM-gamma}. 
We note the delicate difference between \eqref{gamma-hat} and the usual Lasso 
as the estimator and penalty are both restricted to the image of $\bQ$. 
To the best of our knowledge, the regression model \eqref{eq:LM-gamma} in the direction \eqref{bu}, 
which plays a crucial role in our analysis, 
provides a new way of dealing with dense direction $\ba_0$ in de-biasing the Lasso.
We note that the natural choice $\tbu = \ba_0\|\ba_0\|_2^{-2}$ satisfies $\langle \ba_0,\tbu\rangle=1$, but for certain dense $\ba_0$ the corresponding projection matrix
    $\tbQ = \bI_p - \tbu \ba_0^\top$
does not preserve sparsity as in \eqref{bu-properties}.

For the purpose of the asymptotic normality result in \Cref{thm:unknown-sigma} below,
we will consider estimators $\hbgamma$ satisfying 
\begin{align}
& \|\bQ(\hbgamma - \bgamma)\|_1 
    =
    O_\P(C_0^{-1}) 
    \min\big\{\|\bgamma\|_0 \lamuniv, C_0\|\bgamma\|_1\big\},
    \label{eq:hbu-ell_1-rate}
\\ 
  & \sup\big\{\langle \hbz, \bX\bQ\bh\rangle/n:\|\bQ\bh\|_1=1\big\} 
   = O_\P( C_0^{-1}\lamuniv). 
    \label{eq:hbu-ell_infty}
\end{align}
Inequality \eqref{eq:hbu-ell_1-rate} is the usual $\ell_1$ estimation
rate when {$\bgamma$} is sparse or {$\bgamma$} has small $\ell_1$ norm.
%
Condition \eqref{eq:hbu-ell_infty} 
holds automatically for the Lasso estimator \eqref{gamma-hat} when $C_0\htau=O_{\P}(1)$ 
as a consequence of the KKT conditions as explained in the following
proposition.

\begin{restatable}{proposition}{PropDantzigUnknownSigma}
    \label{prop:dantzig}
Let $\bz=\bX\bu$ and $\bQ=\bI_p-\bu\ba_0^\top$ be as in \eqref{bQ} with the $\bu$ in \eqref{bu}. 
    Assume that $\bSigma_{j,j}\le 1\,\forall j$, $\phi_{\min}(\bSigma)$ is bounded away from 0, 
    and $\min\{\|\bQ\bu_0\|_0\log(p)/n,C_0\|\bQ\bu_0\|_1\sqrt{\log(p)/n}\} = o(1)$. 
    Let $\bQ\hbgamma$ be as in \eqref{gamma-hat} with a constant $A > 2$ and $\htau$ satisfying 
    one of the following conditions: 
    \begin{enumerate}
        \item 
            $\htau = \|\bz - \bX\bQ\hbgamma\|_2/n^{1/2}$ is the recursive solution of 
            \eqref{gamma-hat} as scaled Lasso \cite{sun2012scaled}, 
        \item 
            or $\htau$ is any estimator satisfying 
            $1+o_{\P}(1)\le C_0\htau \le O_{\P}(1)$. 
    \end{enumerate}
    Then, the requirements
    \eqref{eq:hbu-ell_1-rate}-\eqref{eq:hbu-ell_infty}
    are satisfied.

\end{restatable}

\Cref{prop:dantzig} is proved in 
\Cref{sec:proofs-unknown-Sigma}.
The following is our main result for unknown $\bSigma$.

\begin{restatable}{theorem}{thmUnknownSigma}
    \label{thm:unknown-sigma}
Assume that $\bSigma_{jj}\le1$ for all $j\in[p]$ and 
that the spectrum of $\bSigma$ is uniformly 
bounded away from 0 and $\infty$; e.g. $\max(\|\bSigma\|_{op},\|\bSigma^{-1}\|_{op})\le 2$.
Let $\lambda = 1.01\sigma\sqrt{2\log(8 p/s_0)/n}$ for the Lasso \eqref{lasso}
in the linear model \eqref{LM}. 
Let $\eps_n>0$ with $\eps_n=o(1)$ and $\scrB_n$ be the class of 
$(\bbeta,\ba_0)\in \R^{p\times 2}$ {satisfying} 
\bel{sparsity-conditions-consistency}
    \frac{\|\bbeta\|_0\log p}{n}
    + \min
    \Big\{
        \frac{\|\bu_0\|_0\log p}{n},
        \frac{C_0\|\bu_0\|_1\sqrt{\log p}}{\sqrt n}
    \Big\}
    \le \eps_n 
\eel
{and $\ba_0\ne {\bf 0}$,} 
where $\bu_0$ is as in \eqref{u_0}. 
Given $\ba_0\neq {\bf 0}$, let $\bu$ be as in \eqref{bu}, $\bQ = \bI_p-\bu\ba_0^\top$,  
$\bQ\hbgamma$ an estimator of $\bgamma = - \bQ\bu_0$ in the linear model
\eqref{eq:LM-gamma} satisfying \eqref{eq:hbu-ell_1-rate}-\eqref{eq:hbu-ell_infty}, 
{$\hbz$} the estimated score vector in \eqref{estimated-score-vector-z}, 
and $\hat\theta_{\nu,\hbz}$ the de-biased estimate in
\eqref{de-biased-estimate-estimated-score-z}. 
If $\nu=\|\lasso\|_0$, then uniformly for $(\bbeta,\ba_0)\in \scrB_n$ 
\bes
    \sqrt{nF_\theta} (\htheta_{\nu=|\Shat|,\hbz}-\theta) &=& Z_n + O_\P({r_n}) 
\ees
holds, where
$Z_n \to^d N(0,1)$ and
\begin{equation}
\label{eq:r_n-minimum-of-three-sparsity} 
    r_n = r_{n,p}(\bbeta,\ba_0) = 
    \min\Big\{\frac{(\|\bbeta\|_0\wedge  \|\bu_0\|_0)\log(p)}{\sqrt n}, C_0 \|\bu_0\|_1 \sqrt {\log p}\Big\}.
\end{equation}
Consequently, for all $(\bbeta,\ba_0)\in \scrB_n$ satisfying $r_n\to 0$, 
\bes
\sqrt{nF_\theta} (\htheta_{\nu=|\Shat|,\hbz}-\theta)\to^d N(0,1).
\ees
\end{restatable}

\Cref{thm:unknown-sigma} is proved in
\Cref{sec:proofs-unknown-Sigma}.
The sparsity condition \eqref{sparsity-conditions-consistency}
is mild: it only requires that the squared prediction rate
for $\bbeta$
and $\bgamma$ converge to 0.
Under this condition, \Cref{thm:unknown-sigma} 
shows that estimation of $\theta$ is possible, 
for general directions $\ba_0\neq 0$, 
at the rate $n^{-1/2}(1+ r_n)$
where $r_n$ is given by \eqref{eq:r_n-minimum-of-three-sparsity}.
The rate $n^{-1/2}(1+r_n)$ is optimal as it matches the lower
bound in Proposition 4.2 of \cite{javanmard2018debiasing} 
for the estimation of $\theta = \beta_j$ in the canonical basis
directions $\ba_0=\bfe_j$ stated in \eqref{eq:lower-bound-Prop42-javanmard-montanari}.
Before \Cref{thm:unknown-sigma}, it was unknown whether
the lower bound \eqref{eq:lower-bound-Prop42-javanmard-montanari} 
can be attained (cf. for instance the discussion in Remark 4.3 of \cite{javanmard2018debiasing}).
\Cref{thm:unknown-sigma} closes this gap, extends the upper bound to general direction $\ba_0$, 
 and relaxes the $\ell_1$ bound on $\bSigma$ imposed in \cite{javanmard2018debiasing}.

The recent work \cite{cai2019individualized} proposes an alternative
construction of a score vector for general direction $\ba_0$
based on a quadratic program. This quadratic program is similar to the
construction in \cite{ZhangSteph14,JavanmardM14a}, with a modification to handle general
direction $\ba_0$, see \cite[equation (7), (8) and (10)]{cai2019individualized}.
The upper bounds in \cite[Corollaries 3 and 4]{cai2019individualized}
require $\|\bbeta\|_0\lesssim \sqrt{n}/\log p$ in contrast with
\Cref{thm:unknown-sigma} where $\|\bbeta\|_0 \ggg \sqrt n$ is allowed.

Another recent line of research 
\cite{bradic2018testability,zhu2018significance,zhu2018linear}
consider the construction of confidence intervals for $\ba_0^\top\bbeta$
for general directions $\ba_0$ without sparsity assumption on $\bbeta$.
These works consider the setting where $\bbeta$ is arbitrary but bounded in the sense
that $\|\bSigma^{1/2}\bbeta\|_2^2\le C$ for some constant $C\asymp\sigma^2$
independent of $n,p$.
In this setting, $\|\bbeta\|_0\log(p)/n\to 0$ is violated and consistent estimation of
$\bbeta$ or $\bQ_0\bbeta$ is not possible.
Assuming $\|\bSigma^{1/2}\bbeta\|_2^2\le C$
instead of a sparsity assumption on $\bbeta$ leads to different minimax rates:
The rate in \cite[Corollary 5]{bradic2018testability} does not depend on
$\|\bbeta\|_0$ but depends implicitly on $\|\bSigma^{1/2}\bbeta\|_2$ instead;
hence the rate in \Cref{thm:unknown-sigma} and
\eqref{eq:lower-bound-Prop42-javanmard-montanari}
is not directly comparable to theirs.
On a higher level, this line of research is fundamentally different
than the present work: 
\cite{bradic2018testability,zhu2018significance,zhu2018linear} leverage
the assumption that the nuisance part of the signal, $\bX\bQ_0\bbeta$,
is bounded with componentwise variance of the same order as that of the noise,
without attempting to estimate the nuisance part of the signal.
In contrast, \Cref{thm:unknown-sigma} attemps to estimate the nuisance parameter
and the nuisance part of the signal
$\bX\bQ_0\bbeta$ is allowed to have arbitrarily large componentwise variance.

Next, we 
prove that the de-biased estimator in \Cref{thm:unknown-sigma} for unknown $\bSigma$, 
and the ideal $\htheta_{\nu}$ in \eqref{LDPE-df} for known $\bSigma$ as well, 
would not achieve the same rate without the degrees-of-freedom adjustment. 
Compared with \Cref{th-1}, {\Cref{thm:unknown-sigma-lb} below is somewhat less explicit but} 
the sign consistency of the Lasso is no longer required. 

\begin{restatable}{theorem}{thmUnknownSigmaLowerBound}
    \label{thm:unknown-sigma-lb}
Let $\bSigma$, $\eps_n$, $\scrB_n$, $r_n$, $\hbz$ and $\hat\theta_{\nu,\hbz}$ be as in \Cref{thm:unknown-sigma}. 
Let $s_0$ and $s_\Omega$ be positive integers satisfying $s_0\log(p)/n\le \eps_n$, $s_\Omega\le s_0$ and 
\bel{condition-unbounded-s_0-s_Omega}
n^{-1} s_0 s_{\Omega}^{1/2} \sqrt{\log(8p/s_0)}\ggg 1+ s_\Omega \log(p)/n^{1/2}. 
\eel
If $\nu=0$, which means no degrees-of-freedom adjustment in \eqref{de-biased-estimate-estimated-score-z}, 
then there exist $(\bbeta,\ba_0)\in \scrB_n$ such that $\|\bbeta\|_0=s_0$, 
$\|\bu_0\|_0=\|\bSigma^{-1}\ba_0\|_0=s_\Omega$, and 
$\sqrt{nF_\theta} (\htheta_{\nu,\hbz}-\theta)/(1+r_n)$ is stochastically unbounded. 
Moreover, the above statement also holds when $\hat\theta_{\nu,\hbz}$ is replaced by 
$\htheta_{\nu}$ in \eqref{LDPE-df}. 
\end{restatable} 

\Cref{thm:unknown-sigma-lb} is proved in 
\Cref{sec:proofs-unknown-Sigma}.
As an example, 
if $s_\Omega = \eps_n \sqrt{n}/\log(p)$ and $s_0 \ge \eps_n^{-1} n^{3/4}$
for some $\eps_n\to 0$, then
\begin{itemize}
    \item
        \eqref{condition-unbounded-s_0-s_Omega} holds so that,
without adjustment, $\sqrt{n F_\theta}(\htheta_{\nu=0,\hbz}-\theta)$
is unbounded by \Cref{thm:unknown-sigma-lb}
for some $(\bbeta,\ba_0)\in\calB_n$ with 
$\|\bbeta\|_0=s_0$ and $\|\bu_0\|_0 = s_\Omega$.
\item
    $r_n\to 0$ hence 
$\sqrt{n F_\theta}(\htheta_{\nu=|\Shat|,\hbz}-\theta)\to^d N(0,1)$
        by \Cref{thm:unknown-sigma} and the de-biased estimate
        adjusted with $\nu=\|\lasso\|_0$ is efficient
        for all $(\bbeta,\ba_0)\in\calB_n$
        with $\|\bbeta\|_0\le s_0, \|\bu_0\|_0\le s_\Omega$.
\end{itemize}

\subsection{Unknown $\bSigma$ and canonical basis directions $\ba_0=\be_j$}

For convenience we provide here the notation and corollary of \Cref{thm:unknown-sigma}
in the case of canonical basis vector $\ba_0=\be_j$ for some $j\in\{1,\dots,p\}$.
{We} denote $(\bu_0,\bz_0,\hbz)$ by $(\bu_j,\bz_j,\hbz_j)$ 
{and write the linear model
\eqref{eq:LM-gamma} as} 
\begin{equation}
    \label{eq:LM-gamma-j}
    \bX\be_j = \bX^{(-j)} ~ \bgamma^{(j)} + \bz_j
\end{equation}
where $\bX^{(-j)}\in\R^{n\times (p-1)}$ is the matrix $\bX$ with $j$-th column removed,
$\bgamma^{(j)}\in\R^{p-1}$.
The corresponding vector $\bu_0$ is 
$\bu_j = (\bSigma^{-1})_{jj}^{-1}\bSigma^{-1}\be_j
$ which is related to $\bgamma^{(j)}$ by
$(\bu_j)_j = 1$ and $(\bu_j)_{(-j)} = - \bgamma^{(j)}$.
The ideal score vector $\bz_0$ becomes $\bz_j=(\bSigma^{-1})_{jj}^{-1}\bX\bSigma^{-1}\be_j$
and has iid $N(0,(\bSigma^{-1})_{jj}^{-1})$ entries independent of
$\bX^{(-j)}$. 
For a given estimator $\hbgamma^{(j)}$ of $\bgamma^{(j)}$,
the score vector \eqref{estimated-score-vector-z} is then
$\hbz_j = \bX\be_j -\bX^{(-j)} \hbgamma^{(j)}$ and the de-biased estimate
\eqref{de-biased-estimate-estimated-score-z} reduces to
\begin{equation}
    \label{de-biased-estimate-hbeta-j}
\hat \beta_j^{\text{\tiny (de-biased)}} = 
(\lasso)_j + \frac{\langle \hbz_j, \by-\bX\lasso\rangle}{(1-\nu/n)\langle \hbz_j, \bX\be_j\rangle}.
\end{equation}
which corresponds to the proposal in \cite{ZhangSteph14} modified
with the degrees-of-freedom adjustment $(1-\nu/n)$.
For $\ba_0=\be_j$, the Lasso estimator \eqref{gamma-hat} becomes 
\begin{equation}
    \hbgamma^{(j)} = 
\argmin_{\bgamma\in\R^{p-1}}
\left\{
    \frac{1}{2 n}
    \|\bX\be_j - \bX^{(-j)}\bgamma\|_2^2
    + \hat\tau_j \lambdabar \|\bgamma\|_1
\right\}.
\label{scaled-lasso-gamma-j}
\end{equation}
with recursive solution $\hat\tau_j = \|\bX\be_j - \bX^{(-j)}\bgamma\|_2/n^{1/2}$ in 
the scaled Lasso \cite{sun2012scaled} or any estimate $\hat\tau_j$ satisfying 
$1+o_{\P}(1)\le (\bSigma^{-1})_{j,j}\hat\tau_j^{2}\le O_{\P}(1)$. 
As the choice of $\bu$ in \eqref{bu} for $\ba_0=\bfe_j$ is $\bu=\bfe_j$, 
the proof of \Cref{thm:unknown-sigma} can be modified to allow {$\lambdabar = A\lamuniv$ with} $A>1$,
since in this case 
$\E \|\bX\bQ\ba_0\|_2^2/n$ 
    is bounded from the above by 1.

\begin{corollary}
Assume that $\bSigma_{jj}\le1$ for all $j\in[p]$ and 
that the spectrum of $\bSigma$ is uniformly 
bounded away from 0 and $\infty$; e.g. $\max(\|\bSigma\|_{op},\|\bSigma^{-1}\|_{op})\le 2$.
Let $\lambda = 1.01\sigma\sqrt{2\log(8 p/s_0)/n}$ for the Lasso \eqref{lasso}.
Consider the Scaled Lasso in \eqref{scaled-lasso-gamma-j}
with $\lambdabar = 1.01\sqrt{2\log(p)/n}$,
the corresponding score vector $\hbz_j$ and de-biased estimate 
$\hbeta_j^{\text{\tiny (de-biased)}}$ in \eqref{de-biased-estimate-hbeta-j}
with $\nu=\|\lasso\|_0$. 
Then for any $j$,
\begin{equation}
    \frac{(\|\bbeta\|_0 \vee \|\bSigma^{-1}\be_j\|_0)     \log(p)}{n} \to 0
    \text{ and }
    \frac{(\|\bbeta\|_0 \wedge \|\bSigma^{-1}\be_j\|_0)   \log(p)}{\sqrt n}\to 0
\end{equation}
implies 
$\sqrt{n}(\bSigma^{-1})_{jj}^{-1/2} (\hbeta_j^{\text{\tiny (de-biased)}} - \beta_j) 
\to^d N(0, \sigma^2)$.
\end{corollary}

\begin{remark}
    \label{remark:tuning}
    {
    The tuning parameters of the present
    section are chosen as 
    $\lambda=1.01\sigma\sqrt{2\log(8p/s_0)/n}$ for simplicity of the presentation. As the results of the present section are consequences of \Cref{thm:main}
    in the next section, more general
    tuning parameters
    of the form \eqref{eq:lambda-larger-than-sigma-lambda_0-over-alpha}
    are also allowed and the resulting
    constants in the theorems would 
    then
    depend on certain constants $\eta_2\in (0,1),\eta_3>0$.
    }
\end{remark}

\section{Theoretical results for known $\bSigma$}
\label{sec-3-main-results}

In this section, we prove that the degrees-of-freedom adjusted LDPE 
in (\ref{LDPE-df}) indeed removes the 
bias of the Lasso for the estimation of a general linear functional 
$\theta = \langle \ba_0,\bbeta\rangle$
when $(s_0/n)\log(p/s_0)$ is sufficiently small and a sparse Riesz condition (SRC) \cite{ZhangH08} holds 
on the population covariance matrix $\bSigma$ of the Gaussian design. 

The SRC is closely related to the restricted isometry property (RIP) \cite{CandesT05,CandesT07}. 
While the RIP is specialized for nearly uncorrelated design variables in the context of compressed 
sensing, the SRC is more suitable in analysis of data from observational studies or experiments 
with higher correlation in the design. For example, the SRC allows an upper sparse eigenvalue greater than 2. 
For $p\times p$ positive semi-definite matrices $\bM$, integers $1\le m\le p$ and 
a support set $B\subset\{1,\ldots,p\}$, define a lower sparse eigenvalue as 
\bel{sparse-eigen}
\phi_{\min}(m,B;\bM) = \min_{A\subset[p]:|A\setminus B|=m}\phi_{\min}\Big(\bM_{A,A}\Big)
\eel
and an upper sparse eigenvalue as 
\bel{sparse-eigen+} 
\phi_{\max}(m,B;\bM) = \max_{A\subset[p]: |A\setminus B|=m}\phi_{\max}\Big(\bM_{A,A}\Big), 
\eel
where $\phi_{\min}(\bM)$ and $\phi_{\max}(\bM)$ are respectively 
the smallest and largest eigenvalues of symmetric matrix $\bM$. 
Define similarly the sparse condition number by
\bel{def-condition-number} &&
\phi_{\rm cond}(m;B,\bM) = \max_{A\subset [p] : |A \setminus {B}|\le (1\vee m)}
\big\{\phi_{\max}(\bM_{A,A})/\phi_{\min}(\bM_{A,A})\big\}. 
\eel

Recall that $S$ is the support of $\bbeta$ and $s_0=|S|$.
For a precise statement of the sample size requirement for our main results,
we will assume the following.

\begin{assumption}
    \label{assumption:main}
    Assume that $\bSigma$ is invertible with diagonal elements at most 1, i.e., 
    $\max_{j=1,...,p}\Sigma_{jj}\le 1$.
    Consider positive integers $\{m,n,p,k\}$ and positive constants
    $\constants$ with $\eta_2, \eta_3 \in(0,1)$. 
    Set the tuning parameter of the Lasso by 
    \begin{align}
        \label{eq:lambda-larger-than-sigma-lambda_0-over-alpha}
            \lam = \eta_2^{-1}(1+\eta_3)\sigma\lam_0,
            \qquad
            \text{where}
            \qquad
            \lam_0=\sqrt{(2/n)\log(8p/k)}.
    \end{align}
Define $\{\tau_*,\tau^*\}$ by
$\tau_* = (1-\eps_1-\eps_2)^2$, $\tau^*=(1+\eps_1+\eps_2)^2$
and assume that
\bel{SRC-population-final}
{s_0}+k 
< \frac{(1-\eta_2)^2 2m}{(1+\eta_2)^2 \big\{(\tau^*/\tau_*)\phi_{\rm cond}(m+k;S,\bSigma) - 1 \big\} }
\eel
and $\rho_*\le \phi_{\min}(m+k,S;\bSigma)$ hold.
Assume that $\lambda_0\sqrt{s_*} \le 1$ where $s_*=s_0+m+k$, as well as
\begin{align}
    &2(m+k)+s_0+1 \le (n-1)\wedge (p+1),
    \label{condition-for-full-rank-2m}
    \\
    &\eps_1+\eps_2<1,
    \quad
    \eps_3+\eps_4=\eps_2^2/8,
    \label{conditions-epsilons-1}
    \\
    &s_0+m+k +1 \le \min(p+1,\eps_1^2n/2),
    \quad
    \log \binom{p - s_0}{m+k}\le\eps_3n. 
    \label{conditions-epsilons-2}
\end{align}
\end{assumption}
Typical values of $k, m$ and $\constants$ are given after
\Cref{corollary:well-adjusted-|Shat|} below.
As will become clear in the proofs in
\Cref{sec:proof-probabilistic-lemma},
the integer $k$ above is an upper bound
on the cardinality of the set 
\begin{equation}
    \label{eq:set-B-highly-correlated}
B=\{j\in[p]: |\bep^\top\bx_j|/n \ge \eta_2\lambda\},
\end{equation}
i.e., the set of covariates that correlate highly with the noise.
If $k=1$ then $\lambda = \eta_2^{-1}(1+\eta_3)\sigma\sqrt{(2/n)\log(8p)}$
and the set $B$ is empty with high probability.
The integer $m$ is, with high probability, an upper bound on the cardinality of the set
$\supp(\lasso)\setminus (S\cup B)$. In other words,
the support of $\lasso$ contains at most $m$ variables
that are neither in the true support $S$ nor in the set $B$ of highly correlated
covariates.
These statements are made rigorous in
\Cref{sec:proof-false-positive-1,sec:proof-false-positive-2}. 
Results of the form $|\Shat| = O_\P(s_0)$
have appeared before 
{for the Lasso, see for instances 
\cite[Theorem 1]{ZhangH08}, 
\cite[Eq. (7.9)]{bickel2009simultaneous},
\cite[Corollary 2 (ii)]{zhang2012general} and
\cite[Theorem 3]{belloni2014}. 
Among these existing bounds, the theory derived in the present paper is closest to 
\cite[Corollary 2 (ii)]{zhang2012general} where a bound of the form $|\Shat| = O_\P(s_0)$ 
is derived under a condition on the upper sparse eigenvalue \eqref{sparse-eigen+} 
after a prediction error bound under a weak restricted eigenvalue condition. 
They depart from other existing bounds of the form $|\Shat| = O_\P(s_0)$ 
in several ways.
The bounds in \cite[Theorem 1]{ZhangH08}} 
requires the tuning parameter
to be set as a function of the sparse eigenvalues of $\bX^\top\bX/n$.
The bound from
\cite{bickel2009simultaneous} involves $\phi_{\max}(\bX^\top\bX/n)$
which is unbounded if $p/n\to+\infty$ for Gaussian designs.
The bound \cite[Theorem 3]{belloni2014} tackles tuning parameters
larger than $\sigma\sqrt{2\log(p)/n}$ but does not provide guarantees
for smaller tuning parameters of order $\sigma\sqrt{2\log(8p/k)/n}$.
The theory developed for the present paper in
\Cref{sec:proof-probabilistic-lemma}
improves upon
these aforementioned references:
The theory only requires bounds on sparse condition number
(cf. the SRC condition \eqref{SRC-population-final}),
the tuning parameters need not depend on the sparse eigenvalues,
and small tuning parameters of order $\sigma\sqrt{2\log(8p/k)/n}$
are allowed.
Furthermore, the theory in 
\Cref{sec:proof-probabilistic-lemma}
clearly separates the roles
of $s_0,k$ and $m$: $k$ is an upper bound on the cardinality
of the set \eqref{eq:set-B-highly-correlated} of covariates highly correlated
with the noise, $m$ is an upper bound on $\supp(\lasso)\setminus(S\cup B)$,
and consequently $\|\lasso\|_0\le s_0 + k + m$.

\paragraph{Stochastically bounded $O_\P(\cdot)$ notation}
In the following results, we consider an asymptotic regime with growing $\{s_0, m, k, n, p\}$
such that
\begin{equation}
    \label{asymptotic-regime}
    p/k\to +\infty, \qquad s_*\lambda_0^2 \to 0
\end{equation}
where $s_*=s_0+m+k$.
This means that we consider a sequence of regression problems \eqref{LM}
indexed by $n$ and $\{s_0, m, k, p\}$ are functions of $n$
such that \eqref{asymptotic-regime} holds and \Cref{assumption:main} is satisfied 
for all $n$ with constants $\constants$ independent of
$n$.
For a deterministic sequence $a_n$, we write $W_n=O_\P(a_n)$
if the sequence of random variables $(W_n)$ is such that for any arbitrarily small $\gamma>0$,
there exists constants $K,N$ depending on $\gamma$ and $\constants$
such that for all $n\ge N$, $\P(W_n > K) \le \gamma.  $
We also write $W_n=o_\P(1)$ if $W_n=O_\P(a_n)$ for some $a_n\to 0$.
Under the above \Cref{assumption:main}, our main result is the following.

\begin{theorem}
    \label{thm:main}
    Let \eqref{asymptotic-regime} and \Cref{assumption:main} be fulfilled.
    Let $F_\theta = 1/(\sigma C_0)^2$ be the Fisher information as in (\ref{Fisher}), and 
    $T_n = \sqrt{nF_\theta}\langle \bz_0,\bep\rangle/\|\bz_0\|_2^2$ so that $T_n$ has the $t$-distribution with 
    $n$ degrees of freedom. 
    For any random degrees-of-freedom adjustment $\nu\in[0,n]$
    we have
    \begin{equation*}
        \sqrt{n F_\theta}
        (1-\nu/n)(\htheta_{\nu}-\theta)
        = T_n +  \sqrt{F_\theta / n}
        \left\langle \ba_0, \lasso-\bbeta \right\rangle \left(|\Shat| - \nu\right)
        + O_\P\left(\lambda_0 \sqrt{s_*} \right)
        .
    \end{equation*}
    If the condition number $\phi_{\rm cond}(p;\emptyset,\bSigma)=\|\bSigma\|_{op}\|\bSigma^{-1}\|_{op}$ of the population 
    covariance matrix $\bSigma$ is bounded, then $O_\P\left(\lambda_0 \sqrt{s_*} \right)$ above 
    can be replaced by $O_\P\left(\lambda_0 \sqrt{s_0+k} \right)$ [by 
    $O_\P\left(\lambda_0 \sqrt{s_0} \right)$ when the penalty is chosen with 
    $k\lesssim s_0$ in \eqref{eq:lambda-larger-than-sigma-lambda_0-over-alpha}].
\end{theorem}
The result is proved in \Cref{sec:proof-main-part-(ii)}.
If $\lambda_0\sqrt{s_*}\to 0$ and $k/p\to 0$,
the above result implies that
$\sqrt{n F_\theta} (1-\nu/n)(\htheta_{\nu}-\theta)$ 
is within $o_\P(1)$ of $T_n$ of 
the $t$-distribution with $n$ degrees of freedom
if and only if
\begin{equation}
    \sqrt{F_\theta/n}\left\langle \ba_0, \lasso-\bbeta \right\rangle \left(|\Shat| - \nu\right)=o_\P(1).
    \label{remaining-bias}
\end{equation}
The left hand side of \eqref{remaining-bias} is negligible
either because the modified de-biasing scheme \eqref{LDPE-df}
is correctly adjusted with $\nu=|\Shat|$ (or $\nu\approx|\Shat|$) to account for the degrees of freedom
of the initial estimator $\lasso$,
or because the estimation error of the initial estimator
$\langle \ba_0, \lasso-\bbeta \rangle$ is significantly small.

The choice of degrees-of-freedom adjustment $\nu=|\Shat|$ 
ensures that the quantity \eqref{remaining-bias} is always equal to 0. 
This leads to the following corollary.

\begin{corollary}
    \label{corollary:well-adjusted-|Shat|}
    Let \eqref{asymptotic-regime} and \Cref{assumption:main} be fulfilled.
    With the notation from \Cref{thm:main},
    if $\nu=|\Shat|$ then
    \begin{equation}
        \sqrt{n F_\theta}
        \left(1-|\Shat|/n\right)\left(\htheta_{\nu=|\Shat|}-\theta\right)
        = T_n  + O_\P\left(\lambda_0 \sqrt{s_*}\right)
        .
        \label{eq:thm-main-second}
    \end{equation}
\end{corollary}

Hence if $\lambda_0\sqrt{s_*}\to 0$ and $k/p\to0$,
the de-biasing scheme \eqref{LDPE-df} correctly adjusted with $\nu=|\Shat|$
enjoys asymptotic efficiency.
To highlight this fact and give an example of typical
values for $m,k$ and $\constants$ in \Cref{assumption:main},
let us explain how \Cref{corollary:well-adjusted-|Shat|}
implies \eqref{teaser-asymptotic-efficiency} of \Cref{thm:teaser}. 
Set $\eta_2^{-1}=\sqrt{1.01}$, $\eta_3=\sqrt{1.01}-1$ and $k=s_0$, so  that
the tuning parameter \eqref{eq:lambda-larger-than-sigma-lambda_0-over-alpha}
is equal to $\lambda$ defined in \Cref{thm:teaser}.
Set also $\eps_1=\eps_2=1/4$ so that $\tau_*=1/4, \tau^*=9/4$.
Under the assumptions of \Cref{thm:teaser}, the spectrum of $\bSigma$
is bounded away from 0 and $\infty$ (e.g. a subset of $[1/2,2]$) 
and the sparse condition number appearing in \eqref{SRC-population-final} is 
bounded (e.g. at most 4 respectively).
Next, set $m = C s_0$ for some large enough absolute constant $C>0$ chosen so that \eqref{SRC-population-final} holds;
this gives $s_*=s_0+m+k=(C+2)s_0$.
The conditions in \Cref{assumption:main} are satisfied
thanks to $\lambda_0\sqrt{s_*}\to 0$ and $k/p\to 0$.
By \Cref{lm-7-probability-of-Omega_1-Omega_2} we get $|\Shat|=O_\P(s_0)$. 
Then \eqref{teaser-asymptotic-efficiency}
is a direct consequence of \eqref{eq:thm-main-second}.

By \Cref{thm:main}, the unadjusted de-biasing scheme \eqref{LDPE} enjoys asymptotic efficiency 
for all fixed $\ba_0$ and $\bbeta$ with $\|\bbeta\|_0\le s_0$ 
if and only if \eqref{remaining-bias} holds with $\nu=0$,
i.e., if
\begin{equation}
    \sqrt{F_\theta/n}\left\langle \ba_0, \lasso - \bbeta\right\rangle |\Shat| = o_\P(1).
    \label{remaining-quantity-unbiased}
\end{equation}
By the Cauchy-Schwarz inequality,
$|\langle \ba_0, \lasso - \bbeta\rangle| \le C_0 \|\bSigma^{1/2}(\lasso - \bbeta)\|_2$. 
Under \Cref{assumption:main} or other typical conditions on the restricted eigenvalues of $\bSigma$
and the sample size, the population risk 
$\|\bSigma^{1/2}(\lasso - \bbeta)\|_2$
is of order $O_\P(\sigma\lambda_0\sqrt{s_*})$
which grants \eqref{remaining-quantity-unbiased} if $\lambda_0\sqrt{s_*} s_*/\sqrt n\to 0$.
This is the content of the following corollary which is formally proved in \Cref{sec:proof-corollary-unadjusted}.

\begin{corollary}[Unadjusted LDPE]
    \label{corollary:unadjusted}
    Let \eqref{asymptotic-regime} and \Cref{assumption:main} be fulfilled.
    With the notation from \Cref{thm:main},
    if $\nu=0$ then
    \begin{equation}
        \sqrt{n F_\theta}
        (\htheta_{\nu=0}-\theta)
        = T_n 
        + O_\P\left(\lambda_0\sqrt{s_*}\left(1+\frac{s_*}{\sqrt n}\right)\right)
        .
        \label{eq:corollary-unadjusted}
    \end{equation}
\end{corollary}

If $\lambda_0^2 (s_*)^{3}/n \to 0$ then
the right hand side of \eqref{eq:corollary-unadjusted}
converges in probability to $T_n$.
In this asymptotic regime, the degrees-of-freedom adjustment is not necessary
and the unadjusted \eqref{LDPE} enjoys asymptotic efficiency.
Note that although the adjustment
$\nu=|\Shat|$ that leads to the efficiency
of $\htheta_{\nu}$ in \Cref{corollary:well-adjusted-|Shat|}
is not necessary in this particular asymptotic regime, such adjustment does not harm either.
Since the practitioner cannot establish whether the asymptotic regime $\lambda_0^2 (s_*)^{3}/n \to 0$
actually occurs because $s_0$ and $s_*$ are unknown, it is still recommended to use the adjustment $\nu=|\Shat|$ as in \Cref{corollary:well-adjusted-|Shat|}
to ensure efficiency for the whole range of sparsity.

An outcome of \Cref{th-1} is that the unadjusted de-biasing scheme \eqref{LDPE} cannot
be efficient in the regime \eqref{regime-unadjusted-cannot-be-efficient}.
By \Cref{th-1} and the discussion surrounding \eqref{regime-unadjusted-cannot-be-efficient}
on the one hand,
and \Cref{corollary:unadjusted} and the discussion of the previous paragraph 
on the other hand, we have established the following phase transition:
\begin{itemize}
    \item If $\lambda_0^2 (s_*)^{3}/n \lll 1$, the unadjusted de-biasing scheme \eqref{LDPE} is efficient for every $\ba_0$, by \Cref{corollary:unadjusted}.
    \item If $\lambda_0^2 s_0^3/n \ggg 1$, the unadjusted de-biasing scheme \eqref{LDPE} cannot be efficient
        for certain specific $\ba_0$.
\end{itemize}
In other words, there is a phase transition at $s_* \asymp n^{2/3}$ (up to a logarithmic factor)
where degrees-of-freedom adjustment becomes necessary to achieve asymptotic efficiency
for all preconceived directions $\ba_0$.
Condition $s_* \lll n^{2/3}$ is a weaker requirement than the assumption $s_*
\lll \sqrt n$ commonly made in the literature on de-biasing. 

\section{De-biasing without degrees of freedom adjustment
under additional assumptions on $\bSigma$}
\label{sec-4-initial-bias-bounded}

The left hand side of \eqref{remaining-quantity-unbiased}
quantifies the remaining bias of the unadjusted de-biasing scheme \eqref{LDPE}.
Under an additional assumption on $\bSigma$, namely a bound on
$\|\bSigma^{-1}\ba_0\|_1$,
the initial bias of the Lasso $\langle \ba_0, \lasso-\bbeta\rangle$
is small enough to grant asymptotic efficiency to the unadjusted de-biasing scheme \eqref{LDPE}.
The following theorem makes this precise.

\begin{theorem}
    \label{thm:unadjusted-ell1-additional-assumption}
    Let \eqref{asymptotic-regime} and \Cref{assumption:main} be fulfilled.
    Suppose 
    \begin{equation}
        \label{condition-K_0}
        \|\bSigma^{-1}\ba_0\|_1/\|\bSigma^{-1/2}\ba_0\|_2 
        \le K_{0,n,p} = K_{1,n,p}\sqrt{n/s_*}
    \end{equation}
    for some quantities $K_{0,n,p}$ and $K_{1,n,p}$. Then, 
    $\sqrt{F_\theta}|\langle \ba_0, \lasso-\bbeta \rangle |  = O_\P( \lam_0 K_{0,n,p} )$
    and
    \begin{equation*}
        \sqrt{n F_\theta} (\htheta_{\nu=0}-\theta) = T_n 
       + O_{\P}\big( (1+K_{1,n,p})\lam_0\sqrt{s_*} + s_*/n \big). 
    \end{equation*}
    This implies that
    $\sqrt{n F_\theta} (\htheta_{\nu=0}-\theta) = T_n+o_\P(1)$ 
    when $K_{1,n,p}=O(1)$. 
\end{theorem}

The proof is given in 
\Cref{sec:proof-thm-infty-norm}.
In other words,
the unadjusted de-biasing scheme
\eqref{LDPE} is efficient and
degrees-of-freedom adjustment is not needed for efficiency if the
$\ell_1$ norm of $\bSigma^{-1}\ba_0$ is bounded from above as in 
$$
\|\bSigma^{-1}\ba_0\|_1/\|\bSigma^{-1/2}\ba_0\|_2 = O(\sqrt{n/s_*})
$$
with $s_*/p\to 0$ and $(s_*/n)\log(p/s_*)\to 0$.
This 
improves by 
a logarithmic factor the condition 
$\|\bSigma^{-1}\ba_0\|_1/ \|\bSigma^{-1/2}\ba_0\|_2=O(1)$ required for efficiency
in \cite{javanmard2018debiasing}.

The above result explains why the necessity of degrees-of-freedom adjustment
did not appear in previous analysis such as \cite{javanmard2018debiasing}; 
$\sqrt{F_\theta}|\langle \ba_0, \lasso-\bbeta \rangle |  = O_\P(\lambda_0)$ 
when $K_{0,n,p}=O(1)$ in \eqref{condition-K_0}, and 
the unadjusted de-biasing scheme \eqref{LDPE} is efficient 
when $K_{1,n,p}=O(1)$ in \eqref{condition-K_0}. 
However, by \Cref{th-1} and the discussion surrounding \eqref{regime-unadjusted-cannot-be-efficient}, there exist certain $\ba_0$ with large
$\|\bSigma^{-1}\ba_0\|_1/ \|\bSigma^{-1/2}\ba_0\|_2$ 
such that the unadjusted de-biasing scheme cannot be efficient.
For such $\ba_0$, degrees-of-freedom adjustments are necessary to achieve
efficiency.

\section{An $\ell_\infty$ error bound for the Lasso}
\label{sec-5-ell_infty}
The idea of the previous section can be applied
to $\ba_0=\be_j$ simultaneously for all vectors $\be_j$
of the canonical basis $(\be_1,...,\be_p)$.
This yields the following $\ell_\infty$ bound on the error of the Lasso.

\begin{theorem}
    \label{thm:infty-norm}
    Let \Cref{assumption:main} be fulfilled, and further assume that $\log p< n$.
    Then the Lasso satisfies simultaneously for all $j=1,...,p$
    \begin{equation}
        \label{bound-Lasso-on-coefficient-j}
        \left| \lassoNoBold_j - \beta_j \right| 
        \le
        \frac{
         M_5^2  \|\bSigma^{-1}\be_j\|_1 \lambda 
        +
        {\sigma\|\bSigma^{-1/2}\be_j\|_2}
        \sqrt{\log p/n}
        \left(
            2 M_5
            + 3\bar M \lambda_0\sqrt{s_*}
    \right)}{1-s_*/n}
    \end{equation}
    on an event $\Omega_{\ell_\infty}$ such that $\P(\Omega_{\ell_\infty}^c)\to 0$ 
    when \eqref{asymptotic-regime} holds,
    where
    $s_* = s_0+m+k$, $M_5=1/(1-\eta_3)$ and $\bar M$ is a constant
    that depends on $\constants$ only.
    Consequently, since $\|\bSigma^{-1/2}\be_j\|_2 \le \|\bSigma^{-1}\be_j\|_1$,
    on the same event we have
    \[
        \|\lasso - \bbeta \|_\infty
        \le
        \rho(\bSigma)\left(\frac{ M_5^2 + 2 M_5  + 4 \bar M \lambda_0\sqrt{s_*}}{1-s_*/n}\right)
        \max\left(\lambda, \sigma\sqrt{\frac{\log p}{n}}\right)
    \]
    where $\rho(\bSigma) = \max_{j=1,...,p} \|\bSigma^{-1}\be_j\|_1$.

\end{theorem}
The proof is given in
\Cref{sec:proof-thm-infty-norm}.
The above result asserts that if the $\ell_1$-norms
of the columns of $\bSigma^{-1}$ are bounded from above by some constant $\rho(\bSigma)>0$
then
$$\|\lasso-\bbeta\|_\infty \le C(\bSigma) \max(\lambda,\sigma\sqrt{\log(p)/n} )$$
holds with overwhelming probability 
for some constant $C(\bSigma)\lesssim \rho(\bSigma)$. 

    %
    Although some $\ell_\infty$ bounds for the lasso have
    appeared previously in the literature, we are not aware
    of previous results comparable to \Cref{thm:infty-norm}
    for $s_0\ggg\sqrt{n}$.
     The result of
     \cite{lounici2008sup} and
     \cite[Theorem 2(2)]{belloni2013least}
    requires incoherence conditions on the
    design, i.e., that non-diagonal elements of $\bX^\top\bX/n$ are
    smaller than $1/s_0$ up to a constant.
    This assumption is strong and cannot be satisfied in the regime
    $s_0\ggg \sqrt n$, even for the favorable
    $\bSigma=\bI_p$: for $\bSigma=\bI_p$ the standard deviation
    of the $i,j$-th entry is $\E[(\bX^\top\bX/n)_{ij}^2]^{1/2} = 1/\sqrt n$.
    In a random design setting comparable to ours,
    Section 4.4 of \cite{van2016estimation} explains that
    $\|{\lasso}-\bbeta\|_{\infty} \lesssim \max_j \|\bSigma^{-1}\be_j\|_1 \sigma
    \sqrt{\log(p)/n}
    (1 + \|{\lasso}-\bbeta\|_1/\sigma)$.
    This bound is only comparable to our $\ell_\infty$ bound
    in the regime  $\|{\lasso}-\bbeta\|_1 = O_P(1)$, i.e., in the
    regime $s_0 \lesssim \sqrt n$ (up to logarithmic factors)
    since $\|{\lasso}-\bbeta\|_1 \approx \lambda s_0 \approx \sigma s_0\sqrt{\log(p)/n}$.
    Again this result is not applicable (or substantially worse
    than \Cref{thm:infty-norm}) in the more challenging
    regime $s_0\ggg \sqrt n$ of interest here.


\section{Regularity and asymptotic efficiency}
\label{section:efficiency}

\Cref{thm:teaser}(i) shows that the test statistic
$\sqrt{nF_{\theta}}(1-|\Shat|/n)(\hat\theta_\nu-\ba_0^\top\bbeta)$,
properly adjusted with $\nu=|\Shat|$,
converges in distribution to $N(0,1)$, where
$F_{\theta} = 1/\{ \sigma^2 C_0^2\}$
and $C_0 = \|\bSigma^{-1/2}\ba_0\|$.
This holds
under any sequence of distributions $\{\mathbb P_{0}^{n} \}_{n\ge 1}$
defined by $\|\bbeta\| = s_0$, 
$s_0\log(p/s_0)/n\to 0$,
$\max(\|\bSigma\|_{op},\|\bSigma^{-1}\|_{op}) \le K$
for some constant $K$ independent of $n,p$, and
$$
\bX \text{ has iid rows } N({\bf 0},\bSigma), \qquad
\by|\bX \sim N(\bX\bbeta,\sigma^2\bI_n).
$$
Here, we denote the unknown parameter $\ba_0^\top\bbeta$
by $\theta(\bP^{n}_0)$ to avoid confusion with the probability
measures defined in the next paragraph. 
By Slutsky's theorem, since $|\Shat|/n$ converges to 0 in probability
by \Cref{thm:teaser}(i), we have
\begin{equation}
\mathcal L
\left(\sqrt{nF_{\theta}}(\hat\theta_\nu-\theta(\P_0^{n}))
~;~
\mathbb P_{0}^{n} \right)
\to N(0,1).
\label{limiting-distribution-P0}
\end{equation}

Given $\ba_0\in \R^p$, a positive-definite matrix $\bSigma\in \R^{p\times p}$ 
and $\mathcal B_n\subset \R^p$ as a parameter space, 
let $F_\theta$ be as in \eqref{Fisher}, 
\bes&&
\mathcal U_n \subseteq \Big\{\bu \subset \R^p: \bu^\top\ba_0=1,\, 
\bbeta+t\bu/\sqrt{nF_\theta}\in \mathcal B_n\, \forall t\in [0, t_{\bu}],\, t_{\bu} \to\infty\Big\} 
\ees
as a collection of directions of univariate sub-models 
$\{\bbeta+t\bu/\sqrt{nF_\theta}: 0\le t\le  t_{\bu}\}$.  
For $t>0$ and $\bu\in \mathcal U_n$ let 
$\P_{t,\bu}^n$ be probabilities under which 
\bel{alt-probab}
\by|\bX \sim N\left(\bX(\bbeta+t\bu/\sqrt{nF_\theta}),\sigma^2\bI_n\right) 
\eel
(for either deterministic or possibly non-Gaussian random $\bX$) and 
\bes
\theta(\P^{n}_{t,\bu}) = \left\langle \ba_0,\bbeta+t\bu/\sqrt{nF_\theta}\right\rangle 
= \left\langle \ba_0,\bbeta\right\rangle + t/\sqrt{nF_\theta}. 
\ees 
That is, under $\mathbb P_{t,\bu}^{n}$
the vector $\bbeta$ is perturbed with the additive term
$t\bu/\sqrt{nF_\theta}$, resulting a perturbation of the parameter of interest 
with $t/\sqrt{nF_\theta}$. 
In the above framework, an estimator $\ttheta$ is regular 
(in the directions $\bu\in \mathcal U_n$) if 
\begin{equation}
    \label{regular-estimator}
    \mathcal L
    \left(
    \sqrt{nF_{\theta}}(\ttheta-\theta(\P_{t,\bu}^{n}))
    ~;~
    \P_{t,\bu}^{n}
    \right)
    \to G
\end{equation}
for all fixed $t>0$ and $\bu\in \mathcal U_n$ and 
some distribution $G$ not depending on $t$ and $\bu$. 
That is, the limiting distribution is stable under the small perturbation 
as defined above. 

Our first task is to show that $\htheta_\nu$ is regular in all directions
with the same limiting distribution as in \eqref{limiting-distribution-P0}, 
i.e. \eqref{regular-estimator} holds with $\mathcal U_n=\R^p$ and $G\sim N(0,1)$.   
For $t=0$, \eqref{limiting-distribution-P0} is implied
by \Cref{thm:teaser}(i).
However \Cref{thm:teaser}(i) does not directly imply \eqref{regular-estimator}
for $t\ne 0$ because ${\bu\in\mathcal U_n}$, as well as the unknown regression vector
$\bbeta+{t\bu/\sqrt{nF_\theta}}$ under $\P_{{t,\bu}}^{n}$,
may not be sparse.
The following device due to Le Cam shows that \eqref{regular-estimator} still
holds with the perturbation $t{\bu}/\sqrt{nF_\theta}$
for any fixed $t\ne 0$ independent of $n,p$.

The likelihood-ratio $L_n$ between
$\mathbb P^{n}_{t,{\bu}}$
and $\mathbb P^{n}_{0}$ is given by
\bes
\log L_n
&=& \{ -\|\by-\bX(\bbeta+t{\bu/\sqrt{nF_\theta}})\|^2 + \|\by-\bX\bbeta\|^2\}/(2\sigma^2).
\\&=&
- t^2 C_0^2 \|{\bX\bu}\|^2/(2n) - \langle \bep,{\bX\bu}\rangle t C_0/(\sigma \sqrt n).
\ees
Under $\mathbb P^{n}_{0}$, the random variable
$\sqrt{nF_\theta}(\htheta_\nu-\theta(\P_0^{n}))$
can be written as $\langle \bep,\bz_0\rangle C_0/(\sqrt n\sigma)+o_{\mathbb P}(1)$ 
so that the vector
$(\sqrt{nF_\theta}(\htheta_\nu-\theta(\P_0^{n})), \log L_n)^\top$
converges in distribution under $\mathbb P^{n}_{0}$ to
a bivariate normal vector with mean
{
$(0, -t^2C_0^2\langle\bu,\bSigma\bu\rangle/2 )^\top$ and covariance 
$$
\Big(\begin{smallmatrix} 
    1 & t C_0^2\langle \bu_0,\bSigma\bu\rangle\\
    t C_0^2\langle \bu_0,\bSigma\bu\rangle\quad & t^2C_0^2\langle \bu,\bSigma\bu\rangle
    \end{smallmatrix}
\Big)
=
\Big(\begin{smallmatrix} 
    1 & t \\
    t & \quad t^2C_0^2\langle \bu,\bSigma\bu\rangle
    \end{smallmatrix}
\Big)
,
$$
where the equality is due to $\bu_0=C_0^{-2}\bSigma^{-1}\ba_0$ and
$\langle\ba_0,\bu\rangle =1$.}
It directly follows by Le Cam's third lemma (see, for instance, 
\cite[Example 6.7]{vaart2000asymptotic}) that
$\sqrt{nF_\theta}(\htheta_\nu-\theta(\P_0^{n}))$ converges to $N(t,1)$
under 
$\{\mathbb P^{n}_{t,{\bu}} \}_{n\ge 1}$
and that \eqref{regular-estimator} holds.
For more details, see also \cite[Section 7.5]{vaart2000asymptotic}
about situations where the log-likelihood ratio converges to normal
distributions of the form $N(-a^2/2,a^2)$.

Hence, properly adjusted with $\nu=|\Shat|$, the estimator $\hat\theta_\nu$
is regular and asymptotic normality still holds
if the sparse coefficient vector $\bbeta$ is replaced
by $\bbeta+t{\bu}/\sqrt{nF_\theta}$ for constant $t\in\R$,
even if the perturbation ${\bu}$ is non-sparse.
By the Le Cam-Hayek convolution theorem (see, for instance,
\cite[Theorem 8.8]{vaart2000asymptotic}),
the asymptotic variance of $\sqrt n(\hat\theta_\nu-\theta)$
must be at least $1/F_\theta$ and our estimator $\htheta_\nu$
is efficient, i.e., it achieves the smallest possible
asymptotic variance among regular estimators.

Note that the above reasoning does not inherently rely on
the Gaussian design assumption.
As soon as 
the second moment of the row of $\bX$ exists,
$\|{\bX\bu}\|^2/(n\langle\bu,\bSigma\bu\rangle)\to 1$
{and
$\langle\bX\bu,\bX\bu_0\rangle/(n\langle\bu_0,\bSigma\bu\rangle)\to 1$
almost surely
}
by the law of large numbers. If additionally
{$\bep\sim N({\bf 0},\sigma^2\bI_n)$ and}
$\bX$ is such that $\sqrt{nF_\theta}(\hat\theta_\nu-\ba_0^\top\bbeta)= \langle\bep,\bz_0\rangle C_0/(\sigma \sqrt n) + o_{\mathbb P^{n}_0}(1)$ for sparse $\bbeta$,
the argument of the previous paragraph
is applicable and $\hat\theta_\nu$ is regular
in the sense of \eqref{regular-estimator}.
For instance, if $\ba_0$ is a canonical basis vector, equation 
$\sqrt{nF_\theta}(\hat\theta_{\nu=0}-\ba_0^\top\bbeta)= \langle\bep,\bz_0\rangle C_0/(\sigma \sqrt n) + o_{\mathbb P^{n}_0}(1)$ can be obtained 
for sub-gaussian design and $s_0\lll \sqrt n$
using an $\ell_1/\ell_\infty$ duality inequality, cf.
\cite{ZhangSteph14,GeerBR14,JavanmardM14a}.
In such asymptotic regime,
the argument of the previous paragraph shows that
$\hat\theta_{\nu=0}$ is stable for non-sparse perturbations of the form
$t{\bu}/\sqrt{n F_\theta}$.

We formally state the above analysis and existing lower bounds,  

\begin{proposition}\label{prop-Fisher} 
Let $\mathcal V_n$ be the linear span of $\mathcal U_n$ as a tangent space. 
Suppose 
\bes
\P_{0}^n\Big\{ \Big|\|\bX\bu\|_2^2/(n\bu^\top\bSigma\bu) - 1\Big| > \eps\Big\} = o(1), 
\qquad
{
    \bu\in \mathcal U_n, 
}
\ees
and dim$(\mathcal V_n) = O(1)$. Let $\bu_0$ be as in \eqref{u_0} and 
$\tau = \tau(\mathcal V_n)  = \tbu_0^\top\bSigma\tbu_0/F_\theta$ with 
\bes
\tbu_0
= \argmin\Big\{\bu^\top\bSigma\bu: \bu\in \mathcal V_n, \langle \ba_0,\bu\rangle =1\Big\}. 
\ees
(i) Let $\ttheta$ be a regular estimator  
in the sense of \eqref{regular-estimator} with a limiting distribution $G$. 
Let $\xi\sim G$.   
Then, (a) $\Var(\xi)\ge 1/\tau$; (b) If $\Var(\xi)=1/\tau$, then $\xi\sim N(0,1/\tau)$; 
(c) If ${\tbu_0} = a_1\bu_1+a_2\bu_2$ {for two $\bu_1,\bu_2\in \mathcal U_n$}
and 
$\{a\bu_1+(1-a)\bu_2: 0\le a\le 1\}\subseteq \mathcal U_n$, 
then $\xi = \xi_1+\xi_2$ where $\xi_1 \sim N(0,1/\tau)$ and $\xi_2$ is independent of $\xi_1$.\\
(ii) If $\bu_0\in\mathcal V_n$, then $\tbu_0=\bu_0$ and 
$\tau=\tau(\mathcal V_n)=1$. \\ 
(iii) If \eqref{limiting-distribution-P0} holds, then $\htheta_{\nu}$ is regular 
and locally asymptotically efficient  
in the sense of \eqref{regular-estimator} with $\mathcal B_n=\mathcal U_n=\R^p$. 
\end{proposition}

The above statement is somewhat more general than the usual version 
as we wish to accommodate general parameter space $\mathcal B_n$, 
cf. \cite[Theorem 8.8]{vaart2000asymptotic} for 
$\mathcal U_n = \{\bu\in \mathcal V_n: \langle\ba_0,\bu\rangle=1\}$
and \cite{schick1986asymptotically} and \cite[Theorem 6.1]{zhang2005estimation} for general $\mathcal U_n$. 
We note that the condition on ${\tbu_0}$ in Proposition~\ref{prop-Fisher}(i)(c), 
known as the convolution theorem, 
is equivalent to the convexity of $\mathcal U_n$ and ${\tbu_0}\in \mathcal V_n$. 
The minimum Fisher information is sometimes defined as 
$\min\{{\sigma^{-2}}\bu^\top\bSigma\bu: \langle\ba_0,\bu\rangle=1, \bu \in \mathcal U_n\}$. 
However, when this minimum over $\mathcal U_n$ is strictly larger than the minimum 
over its linear span $\mathcal V_n$, the larger minimum information is not attainable 
by estimators regular with respect to $\mathcal U_n$ in virtue of (i)(a) above. 

In Proposition \ref{prop-Fisher}, the parameter $\tau=\tau(\mathcal V_n)$ 
can be viewed as the relative efficiency for the tangent space $\mathcal V_n$ 
generated by the collection $\mathcal U_n$ of directions of univariate sub-models. 
As the minimization for $\tbu_0$ is taken over no greater a space compared with 
{\eqref{u_0}}, $\tau\ge 1$ always holds. 
When the parameter space $\mathcal B_n$ is strictly smaller than $\R^p$ 
or the regularity (stability of the limiting distribution) is required only for 
deviations from the true $\bbeta$ in a small collection of directions, 
$\tau>1$ may materialize and an estimator regular and efficient relative to 
$\mathcal U_n$ would become super-efficient in the full model with 
$\mathcal B_n = \mathcal V_n = \R^p$. 
According to Le Cam's local asymptotic minimax theorem, 
in the full model, 
such a super-efficient estimator would perform strictly worse than 
a regular efficient estimator  
when the true $\bbeta$ is slightly perturbed in a certain direction. 

The super-efficiency 
was observed in \cite{vandegeer2017efficiency}
where an estimator, also based on the de-biased lasso, achieves 
asymptotic variance strictly smaller than $1/F_\theta$.
The construction of \cite[Theorem 2.1]{vandegeer2017efficiency} goes as follows:
Consider a sequence $\lambda_n^\sharp$
and a sequence of sub-regions $\mathcal B_n\subset \R^p$ of the parameter space
such that the Lasso satisfies uniformly over all $\bbeta\in\mathcal B_n$
both
$$\|\bSigma^{1/2}(\lasso-\bbeta)\|_2 = o_{\mathbb P}(1),
\qquad 
\sqrt n \lambda^\sharp_n\|\lasso-\bbeta\|_1 = o_{\mathbb P}(1).
$$
Then \cite{vandegeer2017efficiency} constructs an asymptotically normal
estimator of the first component $\beta_1$ of $\bbeta$.
However, this estimator depends on a fixed sub-region $\mathcal B_n$
that achieves a particular $\ell_1$ convergence rate given by $\lambda^\sharp_n$,
and the estimator would need to be changed to satisfy asymptotic
normality on a superset of $\mathcal B_n$.
Hence this construction is a super-efficiency phenomenon:
it is possible to achieve a strictly smaller variance than the Fisher information 
lower bound with the $F_\theta$ in \eqref{Fisher} as the estimators are only required to 
perform well on 
a specific parameter space $\mathcal B_n$.
Additionally, the estimator from \cite{vandegeer2017efficiency} cannot
be regular on perturbations of the form $\bbeta+t\bu_0/\sqrt{nF_\theta}$
for non-sparse $\bu_0$,
otherwise that estimator would not be able to achieve an asymptotic variance
smaller than $1/F_\theta$ according to Proposition~\ref{prop-Fisher}. 

\section{Necessity of the degrees-of-freedom adjustment 
in a more general setting}
\label{section:subgaussian}

This section extends \Cref{th-1} to subgaussian designs. It shows
that the degrees-of-freedom adjustment is necessary when the Lasso is
sign-consistent.
\begin{restatable}{theorem}{subGaussianTheoremDofNecessary}
    \label{thm:1-subgaussian}
    Let $S$ be a support of size $s_0=o(n)$
    and assume that $\bX_S \bSigma_{S,S}^{-1/2}$ has iid entries
    from a mean-zero, variance one and subgaussian distribution.
    Assume that $(\bbeta,\ba_0)$ follows a prior independent of $(\bX,\bep)$
    with $\supp(\bbeta)=S$, $\bbeta$ has iid random signs on $S$
    and fixed amplitudes $\{|\beta_j|,j\in S\}$,
    and set $\ba_0=\bSigma\sgn(\bbeta)_S/\sqrt{s_0}$.
    Then on the selection event $\{\Shat = S, \sgn(\lasso) = \sgn(\bbeta)\}$,
    the de-biased estimate  $\htheta_\nu$ in \eqref{htheta-JM}
    with adjustment $\nu$ satisfies
    \bes
    &&\sqrt{n}(1-\nu/n)(\htheta_\nu-\theta) 
    -
    \sqrt n (1-\nu/n)
    \langle \ba_0, (\bX_S^\top\bX_S)^{-1} \bX_S^\top \bep\rangle
    \\&=&
    -
    (s_0-\nu)
    \left(\lambda \sqrt n
    \ba_0^\top
    (\bX_S^\top\bX_S)^{-1}
    \sgn(\bbeta)_S
    \right)
    \\&&+  O_{\mathbb P}\left(
        \lambda\sqrt{s_0\log s_0}
        +
        \phi_{\rm cond}(\bSigma_{S,S})^{1/2}
        \lambda {\sqrt{s_0}}
    \right).
    \ees
    Furthermore, 
    $
    \lambda\sqrt n \ba_0^\top\left[
    (\bX_S^\top\bX_S)^{-1}
    \right]
    \sgn(\bbeta)_S
    = \lambda\sqrt{s_0/n}(1-o_P(1))$
    when $\phi_{\rm cond}(\bSigma_{S,S}) \le C$ for some constant $C>0$
    independent of $n,p,s_0$.
    Consequently, if $\nu=0$ and $s_0^{3/2} \ge n$,
    the right-hand side above is unbounded.
\end{restatable}

The proof is given in
\Cref{sec:appendix-proof-subgaussian}.
In conclusion, for designs with subgaussian independent entries
and under sign-consistency for the Lasso,
the unadjusted $\htheta_\nu$ with $\nu=0$ is not asymptotically normal
as soon as $s_0\ggg n^{2/3}$, similarly to the Gaussian design case
and the conclusion of \Cref{th-1}.

\section{Outline of the proof}
\label{sec:proof}

\subsection{The interpolation path}
\label{sec:proof-path}

Throughout the sequel, let $\bhlasso=\lasso-\bbeta$. It follows from the definition of $\htheta_{\nu}$ in (\ref{LDPE-df}) that 
\bes
(1 - \nu/n)\big(\htheta_{\nu} - \theta\big) = 
\frac{\big\langle \bz_0, \bep\big\rangle}{\|\bz_0\|_2^2}
- (\nu/n)\big\langle \ba_0,\bhlasso\big\rangle 
- \frac{\big\langle \bz_0,\bX\bQ_0\bhlasso\big\rangle}{\|\bz_0\|_2^2}
\ees
with $\bz_0 = \bX\bu_0$ and $\bQ_0 = \bI_{p\times p} - \bu_0\ba_0^\top$, 
where $\bu_0 = \bSigma^{-1}\ba_0/\langle\ba_0,\bSigma^{-1}\ba_0\rangle$. 

In the above expression, $\bz_0$ is independent of $(\bX\bQ_0,\bep)$ but not of
$\lasso$. 
If $\bz_0$ were independent of $\bX\bQ_0\bhlasso$, we would have 
\begin{align}
\calL\Big(\big\langle \bz_0,\bX\bQ_0\bhlasso\big\rangle 
\Big| \bX\bQ_0\bhlasso\Big)
&\sim 
N\Big(0,C_0^{-2}\|\bX\bQ_0\bhlasso\|_2^2\Big)
\nonumber
\\
    &= O_\P(1/ C_0) \|\bX\bQ_0\bhlasso\|_2, 
\label{if-z_0-were-independent-of-h^Lasso}
\end{align}
where $\calL(\xi |\zeta)$ denotes the conditional distribution of $\xi$ given $\zeta$ 
and $C_0=\|\bSigma^{-1/2}\ba_0\|_2$.
Our idea is to decouple $\bz_0$ and $\lasso$ by replacing $\bz_0$ with 
an almost independent copy of itself in the definition of $\lasso$.  

We proceed as follows. Let $\bg\sim N({\bf 0},\E[\bz_0\bz_0^\top])$ be a random vector
independent of $(\bep,\bz_0,\bX)$ such that $\bg$ and $\bz_0$ have
the same distribution. Next, define the random vector
$$
\tbz_0 = \bP_\bep \bz_0 + \bP^\perp_\bep \bg, \quad \text{ where }
\ 
\bP_\bep =\|\bep\|^{-2}  \bep\bep^\top
\ 
\text{ and }
\ 
\bP_\bep^\perp = \bI_n - \bP_\bep.$$
Conditionally on $\bep$,
the random vectors $\bz_0$ and $\tilde\bz_0$ are identically distributed, 
so that $\tbz_0$ is independent of $(\bX\bQ_0,\bep)$. 

Next, let $\tbX = \bX\bQ_0 + \tbz_0\ba_0^\top$
and let $\tildelasso$ be the Lasso solution with 
$(\bX,\by)$ replaced by $(\tbX,\tbX\bbeta+\bep)$. 
Conditionally on $\bep$, 
the random vector $\bP_\bep^\perp\bz_0$ is normally distributed and independent of $\bX \bQ_0 \tildehlasso$ by construction, so that
\bes
\left|\big\langle \bz_0,\bX\bQ_0\tildehlasso\big\rangle  \right|
& \le & \left|
        \big\langle \bP_\bep^\perp \bz_0,\bX\bQ_0\tildehlasso\big\rangle 
        \right|
        +
        \|\bP_\bep \bz_0\|\; \; \|\bP_\bep\bX\bQ_0\tildehlasso\|
, \\
& \le& 
    O_\P(1/ C_0)\left(\|\bP_\bep^\perp \bX\bQ_0\tildehlasso\|
    +    \|\bP_\bep \bX\bQ_0\tildehlasso\|\right)
,
\ees
where the last inequality is a consequence of 
$\E \|\bP_\bep\bz_0\|_2^2=\E \|\bz_0\|_2^2/n = 1 / C_0^2$. 
The above inequalities are formally proved in \Cref{lemma:W'''}.
Although $\tilde\bz_0$ and $\bz_0$ are not independent, conditionally on $\bep$,
their $(n-1)$-dimensional projections
$\bP_\bep^\perp\bz_0$ and $\bP_\bep^\perp\tilde\bz_0$ are independent and
the quantity
$
\big\langle \bz_0,\bX\bQ_0\tildehlasso\big\rangle 
$
is of the same order as in \eqref{if-z_0-were-independent-of-h^Lasso}
where  $\bX\bQ_0\bhlasso$ and $\bz_0$
were assumed independent.

This motivates the expansion
\bel{new-2}
&&(1 - \nu/n)\big(\htheta_{\nu} - \theta\big) = 
\frac{\big\langle \bz_0, \bep\big\rangle}{\|\bz_0\|_2^2}
- \frac{\big\langle \bz_0,\bX\bQ_0 \tildehlasso \big\rangle}{\|\bz_0\|_2^2} + \Rem_\nu, 
\eel
with $\Rem_\nu = \|\bz_0\|_2^{-2}\big\langle \bz_0,\bX\bQ_0(\tildelasso - \lasso)\big\rangle
 - (\nu /n)\big\langle \ba_0, \bhlasso\big\rangle$. 

The key to our analysis is to bound $\Rem_\nu$ by differentiating 
a continuous solution path of the Lasso from $\lasso$ to $\tildelasso$. To this end, define for any $t\in\R$
\bel{bz(t)}
\bz_0(t) &=& \bP_\bep \bz_0 + \bP_\bep^\perp \left[(\cos t)\bz_0+(\sin t) \bg \right], 
\\ \nonumber \bX(t) &=& \bX\bQ_0 + \bz_0(t)\ba_0^\top, 
\eel
and the Lasso solution corresponding to the design $\bX(t)$ and noise $\bep$, 
\begin{equation}
    \label{beta(t)}
\hbbeta(t)
= \argmin_{\bb\in\R^p}\Big\{ \|\bep +\bX(t)\bbeta - \bX(t)\bb\|_2^2/(2n) + \lam\|\bb\|_1\Big\}. 
\end{equation}
For each $t$, by construction, $(\bz_0(t),\bX(t), \hbbeta(t))$ has the same distribution
as $(\bz_0, \bX, \lasso)$.
The above construction defines a continuous path of Lasso solutions along which the distribution of $(\bz_0(t),\bX(t), \hbbeta(t))$ is invariant.
Furthermore,
\begin{align*}
    \text{at } t=0, \qquad &
    \bz_0(0) = \bz_0 \text{ and } \hbbeta(0) = \lasso, \\
    \text{while at } t=\tfrac \pi 2, \qquad &
    \bz_0(\tfrac \pi 2) = \tilde\bz_0 \text{ and } \hbbeta(\tfrac \pi 2) = \tildelasso.
\end{align*}
Thus, with $\dot{\bz}_0(t)=(\pa/\pa t)\bz_0(t)=\bP_{\bep}^\perp[(-\sin t)\bz_0+(\cos t)\bg]$ 
and $\bD(t) = (\pa/\pa \bz_0(t))\hbbeta(t)^\top\in\R^{n\times p}$, 
an application of the chain rule yields  
\bel{Rem_nu}
&& 
\Rem_\nu 
= \int_0^{\pi/2}
\frac{\big\langle \bz_0,\bX\bQ_0\bD^\top (t)\bP_{\bep}^\perp{\dot \bz}_0(t)\big\rangle}{\|\bz_0\|_2^{2}} dt
 - (\nu /n)\big\langle \ba_0, \bhlasso\big\rangle. 
\eel

We will prove in Lemma \ref{lm-gradient-of-Lasso} below that the above calculus is legitimate with  
\bel{matrix-formula}
&& \bX\bQ_0\bD^\top (t)\bP_{\bep}^\perp
\\ \nonumber &=& - \Big\{\bw_0(t) - \bz_0(t)\big\|\bw_0(t)\|_2^2\Big\}
\big(\bP_{\bep}^\perp\bX(t)\bh(t)\big)^\top
\\ \nonumber && - \Big\{\hbP(t)
- \bz_0(t)\big(\bw_0(t)\big)^\top\Big\}
\bP_{\bep}^\perp
\big\langle \ba_0,\bh(t)\big\rangle, 
\eel
where $\Shat(t)=\supp(\hbbeta(t))$, 
$\hbP(t)$ is the orthogonal projection onto the linear span of $\{\bX_j(t),j\in \Shat(t)\}$,
$\bw_0(t)=\bX_{\Shat(t)}(t)\big(\bX_{\Shat(t)}^\top(t)\bX_{\Shat(t)}(t)\big)^{-1}(\ba_0)_{\Shat(t)}$, 
and $\bh(t)=\hbbeta(t)-\bbeta$. 
We note that the $n\times n$ matrix in (\ref{matrix-formula}) 
is a function of $(\bX(t),\bep)$ and 
\bes
\bz_0=\bP_{\bep}\bz_0+\bP_{\bep}^\perp\big[(\cos t)\bz_0(t)-(\sin t){\dot\bz}(t)\big] 
\ees
with $\bz_0(t)=\bX(t)\bu_0$. 
Thus, as ${\dot\bz}_0(t)$ is a $N({\bf 0},\bP_{\bep}^\perp/C_0^2)$ 
vector given $(\bX(t),\bep)$, 
the mean and variance of the integrand 
$\big\langle \bz_0,\bX\bQ_0\bD^\top (t)\bP_{\bep}^\perp{\dot \bz}_0(t)\big\rangle$ 
in (\ref{Rem_nu}) can be readily computed conditionally on $(\bX(t),\bep)$ as a 
quadratic form in ${\dot \bz}_0(t)$. 
This would provide an upper bound for the remainder in (\ref{Rem_nu}) based 
on the size of $\Shat(t)$ and the prediction error $\bX(t)\bh(t)$. 
For example, the main term in this calculation is 
\bes
&& \Big(\E \|\bz_0\|_2^2\Big)^{-1}\int_0^{\pi/2}
\E\Big[\big\langle \bz_0,-\hbP(t)\bP_{\bep}^\perp
{\dot \bz}_0(t)\big\rangle \big\langle \ba_0,\bh(t)\big\rangle
\Big|\bX(t),\bep\Big] dt
\cr &=& \frac{1}{n}\int_0^{\pi/2} (\sin t)\Big\{\big|\Shat(t)\big| 
- \trace\Big(\bP_{\bep}\hbP(t)\bP_{\bep}\Big)\Big\}\big\langle \ba_0,\bh(t)\big\rangle dt, 
\ees
which has approximately the same mean as $(\nu /n)\big\langle \ba_0, \bhlasso\big\rangle$ 
when $\nu =\big|\Shat(0)\big|=\big|\Shat\big|$.  

\begin{remark}
    \label{remark:comparison-leave-one-out}
    {
    For a fixed $j$-th column
    the leave-one-out technique explained in \cite[Section 6.1]{javanmard2018debiasing} studies the modified estimate
    \begin{equation}
    \hbtheta^{(j)} = \argmin_{\bb\in\R^p:b_j=\beta_j}\|\bX\bb-\by\|_2^2/(2n) + g(\bb)
    \label{hbtheta-j}
    \end{equation}
    with the constraint $b_j=\beta_j$, so that the design matrix
    in the quadratic term is replaced by $\bX_{-j}$.
    The study of this perturbed $\hbtheta^{(j)}$
    allows \cite{javanmard2018debiasing} to prove efficiency
    under the condition $\max_{j=1,...,p} \|\Sigma^{-1}e_j\|_1 \le \rho$.
    This differs from our construction in at least three major ways:
    \begin{enumerate}
    \item
    The $\hbtheta^{(j)}$ of \cite{javanmard2018debiasing} does not have the same distribution as $\lasso$,
    while with our construction $\tildelasso$ as well as $\hbbeta(t)$ for each $t\in[0,\pi/2]$ all
    have the same distribution as the Lasso $\hbbeta$ itself;
    \item 
    In our construction the decomposition $\bX=\bX\bQ_0 + \bz_0\ba_0^\top$
    has two independent terms {$\bX\bQ_0$} and $\bz_0\ba_0^\top$, while in
    the construction \eqref{hbtheta-j} above, $\bX=\bX_{-j} + {\bX}\be_j$
    but $\bX_{-j}$ is not independent of the $j$-th column $\bX\be_j$;
    \item
    Our construction allows for general direction $\ba_0$, while the analogue
    of \eqref{hbtheta-j} with constraint $\ba_0^\top\bb$ for dense $\ba_0$,
    namely $ \hbtheta{}^{(0)} =
    \argmin_{\bb\in\R^p:\ba_0^\top(\bbeta-\bb)=0}\|\bX\bb-\by\|_2^2/(2n) +
    g(\bb)$, leads to an estimator that is \emph{not} a Lasso estimator, and
    its analysis would not be straightforward.
    \end{enumerate}
    }
\end{remark}

\subsection{The Lasso prediction error and model size}
Our next task is to show that with high probability,
simultaneously for all $t$ along the path,
the Lasso solutions $\hbbeta(t)$ enjoy guarantees in terms of prediction error
and model size similar to the bounds available for a single Lasso problem.
Define the event $\Omega_1$ by
\begin{equation}
    \label{def-Omega_1}
    \Omega_1 = \left\{0<\inf_{t,t'\ge 0} \phi_{\min} 
    \left(
        \frac 1 n \left(\bX(t)^\top\bX(t)\right)_{\Shat(t')\cup\Shat(t),\Shat(t')\cup\Shat(t)}
    \right). 
\right\}
\end{equation}
Define also $\bh^{(noiseless)}(t) = \bbeta^{(noiseless)}(t)-\bbeta$ where 
$\bbeta^{(noiseless)}(t)$ is the Lasso solution for design matrix $\bX(t)$ in
the absence of noise, that is,
\bel{noiseless}
\bbeta^{(noiseless)}(t)= \argmin_{\bb\in\R^p}\left\{\|\bX(t)(\bbeta - \bb)\|_2^2/(2n) + \lam \|\bb\|_1\right\}. 
\eel
Consider the following conditions: For a certain $s_*\in [s_0\vee 1,n]$ 
and positive $\lam_0$, 
\begin{equation}
    \begin{split}
 & \|\bX(t)\bh(t)\|_2 \le M_1\sqrt{ns_*}\sigma\lam_0, 
\\ & \|\bX(t)\bh^{(noiseless)}(t)\|_2 \le M_1\sqrt{ns_*}\sigma\lam_0,
\\ & \|\bSigma^{1/2}\bh(t)\|_2\le M_2 \sqrt{s_*}\sigma\lam_0,
\\ & |\Shat(t)| \le s_* \le M_{3}(s_0+k), 
\\ & \Big\|\Big(\bSigma_{\Shat(t),\Shat(t)}^{-1/2}\bX_{\Shat(t)}^\top(t)\bX_{\Shat(t)}(t)
\bSigma_{\Shat(t),\Shat(t)}^{-1/2}/n\Big)^{-1}\Big\|_{op}\le M_4,
\\ & (\|\bep\|_2/\sigma) \vee (C_0\|\bz_0(t)\|_2) \vee (n/(C_0\|\bz_0(t)\|_2)) \le M_5 \sqrt{n},
 \label{conds-limited}
\end{split}
\end{equation}
where $M_1, M_2, M_{3}, M_4,  M_5  > 0$ are constants to be specified.
Define the event $\Omega_2$ by
\begin{equation}
    \Omega_2(t)
    = \big\{\ \hbox{(\ref{conds-limited}) holds for $t$}\ \big\}
    \ \hbox{ and }\ 
    \Omega_2=\cap_{t\ge 0}\Omega_2(t). 
    \label{def-Omega_2}
\end{equation}
For a single and fixed value of $t$, the fact that the Lasso
enjoys the inequalities \eqref{conds-limited} under conditions on the design $\bSigma$
can be obtained using known techniques.
For instance,
the first and third inequalities in \eqref{conds-limited} describe
the prediction rate of the Lasso with respect to the empirical covariance matrix
and the population covariance matrix when the tuning parameter
of the Lasso is proportional to $\sigma\lambda_0$.
For the purpose of the present paper, however, we require the above inequalities
to hold with high probability simultaneously for all $t$.
The following lemma shows that this is the case: 
$\Omega_1\cap\Omega_2$ has overwhelming probability under
\Cref{assumption:main}.

\begin{lemma}
    \label{lm-7-probability-of-Omega_1-Omega_2}
    Let the setting and conditions of \Cref{assumption:main} be fulfilled.  
    Set $M_1= (1+\eta_2) \eta_2^{-1}(1+\eta_3)/\sqrt{\rho_*\tau_*}$, 
    $M_2=M_1/\sqrt{\tau_*}$, 
    	\bes
	M_{3} = 1 + \frac{(\tau^*/\tau_*)\phi_{\rm cond}(p;\emptyset,\bSigma) - 1}
	{2(1-\eta_2)^2/(1+\eta_2)^2},
	\ees
    $M_4=1/\tau_*$, $M_5 = 1/(1-\eta_3)$.
    Then the events $\Omega_1,\Omega_2$ defined in
    \eqref{def-Omega_1} and \eqref{def-Omega_2} satisfy 
\begin{equation}
    \begin{split}
    \label{upper-bound-proba-Omega_1_cap_Omega_2}
    1-\P(\Omega_1\cap\Omega_2) \le
    &\quad 2 e^{-n\eps_4}
    + 2e^{-(\eta_3-\sqrt{2/n})_{+}^2n/2} 
    \\
    & + e^{-n\eta_3^2/2} 
    +
    4(2\pi L_k^2+4)^{-1/2} + (L_k+(L_k^2+2)^{-1/2})^{-2}
    .
    \end{split}
\end{equation}
where $L_k=\sqrt{2\log(p/k)}$. 
\end{lemma} 
\Cref{lm-7-probability-of-Omega_1-Omega_2} is proved in
\Cref{sec:proof-probabilistic-lemma}.
Equipped with the result that the events $\Omega_1$ and $\Omega_2$ have overwhelming probability, we are now ready to bound
$\Rem_\nu$ in \eqref{new-2}.

\subsection{An intermediate result}
Before proving the main result (\Cref{thm:main}) in the next
subsections, we now
prove the following intermediate result.

\begin{theorem}
    \label{theorem:slepian-expansion}
    There exists a constant $\bar M>0$ that depends on $M_1,M_2,M_4,M_5$ only such that the following holds.
    Let $F_\theta = 1/(\sigma C_0)^2$ be the Fisher information as in (\ref{Fisher}), and 
    $T_n = \sqrt{nF_\theta}\langle \bz_0,\bep\rangle/\|\bz_0\|_2^2$ so that $T_n$ has the $t$-distribution with 
    $n$ degrees of freedom. 
    Let $\Omega_1$ and $\Omega_2$ be the events defined in \eqref{def-Omega_1} and \eqref{def-Omega_2}.
    Define random variables ${\Rem_I}$ and ${\Rem_{II}}$ by
    \begin{align*}
        {\Rem_I} &= \sqrt{n F_\theta}
        (\htheta_{\nu=0}-\theta)
        - T_n
        - \sqrt{F_\theta / n } \int_0^{\pi/2}(\sin t)\left(
        |\Shat(t)|\langle \ba_0,\bh(t) \rangle 
    \right)
    dt,
        \\
        {\Rem_{II}} &=
        \sqrt{n F_\theta}
        (\htheta_{\nu=0}-\theta)
        - T_n
        -  \sqrt{F_\theta / n} \left\langle \ba_0,\bhlasso \right\rangle \int_0^{\pi/2}(\sin t)\left(|\Shat(t)|\right) dt
        .
    \end{align*}
    Then
    for any $u\in\R$ such that $|u|\le \sqrt n / \bar M$,
    \bes
    \max\bigg\{
    \E\left[I_{\Omega_1\cap\Omega_2}\exp\left(\frac{u{\Rem_I}}{\lambda_0\sqrt{s_*}}\right)\right], 
    \E\left[I_{\Omega_1\cap\Omega_2}\exp\left(\frac{u{\Rem_{II}}}{\lambda_0\sqrt{s_*}}\right)\right]\bigg\} 
    \le 2 \exp\left(\bar M^2 u^2\right).
    \ees
\end{theorem}

We now gather some notation and lemmas to prove 
\Cref{theorem:slepian-expansion} .
Recall that the degrees-of-freedom adjusted LDPE is 
\bes
\htheta_{\nu} = \left\langle \ba_0, \lasso \right\rangle + \frac{\big\langle \bz_0,\by - \bX\lasso\big\rangle}{(1-\nu /n)\|\bz_0\|_2^2}, 
\ees
with $\bz_0 = \bX\bu_0$, where $\bu_0 = \bSigma^{-1}\ba_0/\langle\ba_0,\bSigma^{-1}\ba_0\rangle$ 
is the direction of the least favorable one-dimensional sub-model for the estimation of $\langle\ba_0,\bbeta\rangle$. 
Recall that the Fisher information for the estimation of $\langle\ba_0,\bbeta\rangle$ is 
$F_\theta = \sigma^{-2}/\langle\ba_0,\bSigma^{-1}\ba_0\rangle$, and that $\E \|\bz_0\|_2^2/n = \sigma^2F_\theta = 1/C_0^2$. 
We note that the estimation of $\theta = \langle \ba_0,\bbeta\rangle$ is scale equi-variant under the transformation 
\bel{scale-equi-variant}
\big\{\ba_0,\theta,\htheta_{\nu},\bu_0,\bz_0,F_\theta\big\} 
\to \big\{c\ba_0,c\theta,c\htheta_{\nu},\bu_0/c,\bz_0/c,F_\theta/c^2\big\}. 
\eel
Thus, without loss of generality, we may take the scale $\langle\ba_0,\bSigma^{-1}\ba_0\rangle = 1$ in which 
\bel{standard}
\quad & \bu_0 = \bSigma^{-1}\ba_0, \quad \ \bz_0 = \bX\bu_0\sim N({\bf 0},\bI_n),\quad\ 
F_\theta = \sigma^{-2},\quad\ C_0=1.
\eel
Furthermore, for any subset $A\subset\{1,...,p\}$ we have
\bel{upper-bound-submatrix-ba_0}
\big\|\bSigma_{A,A}^{-1/2}(\ba_0)_{A}\big\|_2^2
&=& \big\|\bSigma_{A,A}^{-1/2}(\bSigma^{1/2})_{A,*}\bSigma^{-1/2}\ba_0\big\|_2^2
\cr &\le& C_0^2\phi_{\max}\Big(\bSigma_{A,A}^{-1/2}(\bSigma^{1/2})_{A,*}(\bSigma^{1/2})_{*,A}\bSigma_{A,A}^{-1/2}\Big)
\cr &\le& C_0^2\phi_{\max}\Big(\bSigma_{A,A}^{-1/2}\bSigma_{A,A}\bSigma_{A,A}^{-1/2}\Big)
\cr &=& C_0^2.
\eel
Let ${\dot f}(t) = (\pa/\pa t)f(t)$ for 
all functions of $t$. By construction of the interpolation path \eqref{bz(t)}, we have
\begin{equation}
    \label{zdot}
{\dot\bz}_0(t) = \bP_\bep^\perp \left[ (-\sin t)\bz_0 + (\cos t) \bg \right], 
\end{equation}
so that $\langle \bep , \dot\bz_0(t) \rangle = 0$ holds for every $t$.
Conditionally on $\bep$, the random vector $(\bX(t),\dot\bz_0(t))$
is jointly normal and $\dot\bz_0(t)$ is independent of $\bX(t)$, so that
the conditional distribution of $\dot\bz_0(t)$ given $(\bX(t),\bep)$ is
\begin{equation}
\calL \left( \dot\bz_0(t) \big| \bX(t), \bep\right)
= N\left({\bf 0}, (1/C_0)^2 \bP_\bep^\perp\right).
\label{conditional-distribution-of-dot-bz_0}
\end{equation}

Here is an outline of the proof of \Cref{theorem:slepian-expansion}.
\begin{enumerate}
    \item Starting from the expansion \eqref{new-2}, 
the key to our analysis is to bound 
the remainder in \eqref{new-2}
by differentiating 
the continuous solution path
\eqref{bz(t)}-\eqref{beta(t)}
from $\lasso$ to $\tildelasso$.
    \item \Cref{lemma:lipschitzness} shows that the function $t\to\hbbeta(t)$
        is Lipschitz in $t$, hence differentiable almost everywhere along the path.
    \item Next, \Cref{lm-gradient-of-Lasso} computes the gradient
        of $t\to\hbbeta(t)$ along the path.
        To compute the gradient, we make use of Lemma~\ref{lemma:kkt-strict}
        which shows that the KKT conditions of the Lasso hold strictly almost everywhere.
    \item Finally, we write
        $\langle \bz_0,\bX\bQ_0(\tildelasso - \lasso)\rangle$
        as an integral from $0$ to $\pi/2$ of the derivative of the function
        $t\to\langle \bz_0,\bX\bQ_0  \hbbeta(t) \rangle$
        and the Lemmas \ref{lemma:W}, \ref{lemma:W'} and \ref{lemma:W''} bound from 
        above this derivative on the event $\Omega_1\cap\Omega_2$,
        thanks to the conditional distribution 
        \eqref{conditional-distribution-of-dot-bz_0} of $\dot\bz_0(t)$ given
        $(\bX(t),\bep)$.
\end{enumerate}

\begin{restatable}[Lipschitzness of regularized least-squares with respect to
    the design]{lemma}{lipschitzLemma}
    \label{lemma:lipschitzness}
    Let $\bep\in\R^n$ and $\bbeta\in\R^p$. Let $\bX$ and $\tilde\bX$ 
    be two design matrices of size $n\times p$ in a compact convex set $\tilde K$.
    Let $h$ be a norm in $\R^p$.
    Let $\hbbeta$ and $\tbbeta$ be the minimizers
    \bes
        \hbbeta
        = \argmin_{\bb\in\R^p}\left\{ L(\bX,\bb) + h(\bb)
        \right\},
        \qquad
        \tbbeta
        = \argmin_{\bb\in\R^p}\left\{ L(\tilde\bX,\bb) + h(\bb)
        \right\}
    \ees
    where $L(\bM,\bb) = \|\bep+\bM\bbeta - \bM\bb\|_2^2/(2n)$ for all $\bM\in\R^{n\times p}$ 
    and $\bb\in\R^p$.
    Then
    $$ 
    \|\bX(\hbbeta-\tbbeta)\|^2
    +
    \|\tilde\bX(\hbbeta-\tbbeta)\|^2
    \le
    C(\tilde K,h,\bep,\bbeta)
    \|\bX-\tilde\bX\|_{op}
    \|\hbbeta - \tilde\bbeta\|_2,
    $$
    where
    $C(\tilde K,h,\bep,\bbeta)$ is a quantity that depends on $\tilde K,h,\bep,\bbeta$ only.
\end{restatable}

\begin{lemma}
    \label{lemma:kkt-strict}
Consider a random design matrix $\bX\in\mathbb R^{n\times p}$
and independent random noise $\bep$ such that both $\bX$
and $\bep$
admit a density with respect to the Lebesgue measure.
Then with probability one, the KKT conditions of the Lasso hold strictly, that is,
$\mathbb P(\forall j\in\Shat, \quad |\bx_{j}^\top(\by - \bX \lasso) | < 1 ) = 1.$
\end{lemma}
\begin{proof}
    Since the distribution of $\bX$ is continuous,
    the assumption of \cite[Proposition 4.1]{bellec_zhang2018second_order_stein}
    is satisfied almost surely with respect to $\bX$
    and the result follows by conditionaning on $\bX$.
\end{proof}

\begin{restatable}{lemma}{gradientLassoLemma}
    \label{lm-gradient-of-Lasso} 
    Let $\bh(t) = \hbbeta(t) - \bbeta$.
    In the event $\Omega_1$ defined by \eqref{def-Omega_1}, 
\begin{equation}
\tildelasso - \lasso
= \int_0^{\pi/2} \bD^\top (t){\dot \bz}_0(t)dt
\label{lasso-X(t)}
\end{equation}
almost surely, where $\bD(t)$ is an $n\times p$ matrix given 
by $\bD_{\Shat^c(t)}(t) =0$ and 
\bes
&& \bD_{\Shat(t)}^\top(t) 
\cr &=&  \Big(\bX^\top(t)\bX(t)\Big)_{\Shat(t),\Shat(t)}^{-1}
\Big((\ba_0)_{\Shat(t)}\big(\bep - \bX(t)\bh(t)\big)^\top
- \bX_{\Shat(t)}^\top(t)\big\langle \ba_0, \bh(t)\big\rangle\Big).
\ees
\end{restatable}

It follows from \eqref{bz(t)} and \eqref{zdot} that conditionally on $\bep$,
the random vector ${\dot\bz}_0(t)$ is independent of $(\bX(t),\bh(t),\bD(t),I_{\Omega_2(t)})$
and the conditional distribution of $\dot\bz_0(t)$ given $(\bep,\bX(t))$ is
given by \eqref{conditional-distribution-of-dot-bz_0}.
Furthermore, by \eqref{zdot} we always have
$\langle \dot\bz_0(t), \bep \rangle = 0$ so that
$(\bep-\bX(t)\bh(t))^\top \dot\bz_0(t) = - (\bX(t)\bh(t))^\top \dot\bz_0(t)$
which simplifies the expression
$\bD_{\Shat(t)}^\top(t)\dot\bz_0(t)$.
Furthermore on $\Omega_2(t)$ defined in \eqref{def-Omega_2}, by the Cauchy-Schwarz inequality, 
\begin{equation}
    \begin{split}
        | \langle \ba_0, \bh(t) \rangle | &\le C_0 \|\bSigma^{1/2}\bh(t)\||_2 \le C_0 M_1 \sigma \lambda_0\sqrt{s_*},
        \\
        \|\bX\bQ_0\bh(t)\|_2/\sqrt n &\le (M_1 +  M_5 M_2)\sigma\lambda_0\sqrt{s_*},
        \\
         \|\bw_0(t)\|_2^2
        &\le (M_4/n) \|\bSigma_{ \Shat(t),\Shat(t) }^{-1/2} (\ba_0)_{\Shat(t)}\|_2^2 \le (M_4/n) C_0^2
        \label{properties-on-Omega_2(t)}
    \end{split}
\end{equation}
with $\bw_0(t)=\bX_{\Shat(t)}(t)\big(\bX_{\Shat(t)}^\top(t)\bX_{\Shat(t)}(t)\big)^{-1}(\ba_0)_{\Shat(t)}$, 
thanks to (\ref{conds-limited}) and \eqref{upper-bound-submatrix-ba_0}.
We will use these properties several times in the following
lemmas in order to bound $\Rem_\nu$ in \eqref{new-2}.

\begin{restatable}{lemma}{lemmaW}
    \label{lemma:W}
    The quantity
    \begin{equation}
    W =
    C_0 \sqrt n \left(\frac{\big\langle\bz_0,\bX\bQ_0(\tildelasso - \lasso) \big\rangle}{C_0^2 \|\bz_0\|_2^2}
    -
    \frac{\big\langle\bz_0,\bX\bQ_0(\tildelasso - \lasso) \big\rangle}{n}
    \right)
        \label{def-W}
    \end{equation}
    satisfies for any $u\in\R$
    \begin{equation}
        \E\left[I_{\Omega_1\cap\Omega_2}
        \exp\left(\frac{u W}{\sigma\lambda_0 \sqrt{s_*} }\right)\right] 
        \le \exp(C|u| + C u^2 )
        \label{eq:upper-bound-W}
    \end{equation}
    for some constant $C=C(M_1,M_2,M_5)>0$ that depends on $M_1, M_2, M_5$ only.
\end{restatable}

\begin{restatable}{lemma}{lemmaWprime}
    \label{lemma:W'}
    The quantity
    \begin{equation}
        \label{def-W'}
        W' =
    \frac{C_0\big\langle\bz_0,\bX\bQ_0(\tildelasso - \lasso) \big\rangle}{\sqrt n}
-
\int_0^{\pi/2} (\sin t)\frac{|\Shat(t)|\big\langle\ba_0,\bh(t)\big\rangle}{C_0\sqrt n} dt 
    \end{equation}
    satisfies
    \begin{equation}
    \E\left[
        I_{\Omega_1\cap\Omega_2} \exp \left( \frac{u W'}{\sigma\lambda_0\sqrt{s_*}} \right)
    \right] 
    \le 
    \exp\left( |u|C'/\sqrt n + \frac{ u^2 C'}{1 - |u|C'/\sqrt n}\right)
    \label{eq:upper-bound-W'}
    \end{equation}
    for any $u\in\R$ such that $|u| < \sqrt n / C'$,
    for some constant $C'=C'(M_1,M_2,M_4, M_5 )>0$ that depends on $M_1, M_2, M_4, M_5$ only.
\end{restatable}

\begin{restatable}{lemma}{lemmaWprimePrime}
    \label{lemma:W''}
    The quantity
    \begin{equation}
        \label{def-W''}
        W'' =
        \frac{1}{C_0\sqrt n}\int_0^{\pi/2} (\sin t)|\Shat(t)|\big\langle\ba_0,\bh(t)\big\rangle dt
        -
        \frac{\big\langle\ba_0,\bhlasso\big\rangle}{C_0\sqrt n}
        \int_0^{\pi/2} (\sin t)|\Shat(t)| dt
        \end{equation}
    satisfies for all $u\in\ R$
    \begin{equation}
        \E\left[\exp\left(\frac{u W''}{\sigma\lambda_0 \sqrt{s_*} } \right)\right]
        \le 2 \exp(C'' u^2)
        \label{eq:upper-bound-W''}
    \end{equation}
    for some constant $C''=C''(M_1,M_2,M_4, M_5 )>0$ that depends on $M_1, M_2, M_4, M_5$ only.
\end{restatable}

\begin{restatable}{lemma}{lemmaWprimePrimePrime}
    \label{lemma:W'''}
    The quantity
    \begin{equation}
        W''' = - \frac{\sqrt n C_0 \langle \bz_0, \bX\bQ_0\tildehlasso\rangle}{C_0^2\|\bz_0\|_2^2}
        \label{def-W'''}
    \end{equation}
    satisfies for all $u\in\R$
    \begin{equation}
        \E\left[I_{\Omega_1\cap\Omega_2} 
        \exp\left(\frac{u W'''}{\sigma\lambda_0\sqrt{s_*}}\right) \right]
        \le 2 \exp(C''' u^2)
        \label{eq:upper-bound-W'''}
    \end{equation}
    for some constant $C'''=C'''(M_1, M_2, M_5)$ that depends on $M_1, M_2, M_5$ only.
\end{restatable}

We are now ready to combine the above lemmas to prove \Cref{theorem:slepian-expansion}.
\begin{proof}[Proof of \Cref{theorem:slepian-expansion}]
    The random variables ${\Rem_I}$ and ${\Rem_{II}}$ in \Cref{theorem:slepian-expansion}
    satisfy
    $$\sigma{\Rem_I} = W''' + W + W', 
    \qquad \sigma{\Rem_{II}} = \sigma{\Rem_I} + W'' = W''' + W + W' + W''.$$
    where 
    $W,W',W''$ and $W'''$ are defined in \eqref{def-W}, \eqref{def-W'},
    \eqref{def-W''} and \eqref{def-W'''}.
    By \Cref{lemma:W,lemma:W',lemma:W'',lemma:W'''},
    there exists a constant $\bar M > 0$ that depends only on $M_1,M_2, M_4, M_5$
    such that for all $u\in\R$ with $|u| < \sqrt n/\bar M$,
    \begin{equation}
        \max_{V\in\{W,W',W'',W'''\} }\E\left[I_{\Omega_1\cap\Omega_2}\exp\left(\frac{u V}{\sigma\lambda_0\sqrt{s_*}}\right)\right]
    \le 2 \exp\left(\bar M^2 u^2\right)
    \label{previous-display-Theorem-6.2-proof}
    \end{equation}
    because one can always increase $\bar M$ so that
    the right hand side of the previous display
    is larger than the right hand side of
    \eqref{eq:upper-bound-W},
    \eqref{eq:upper-bound-W'}
    \eqref{eq:upper-bound-W''}
    and
    \eqref{eq:upper-bound-W'''}.
    By Jensen's inequality,
    \begin{align*}
        \E\left[
        I_{\Omega_1\cap\Omega_2}
        \exp(\frac{u {\Rem_I}}{\lambda_0\sqrt{s_*}})
        \right]
        \le
        \frac 1 3
        \E\left[
        I_{\Omega_1\cap\Omega_2}
        \left(
            {e}^{\frac{3u W'''}{\sigma\lambda_0\sqrt{s_*}}}
        +
            {e}^{\frac{3u W}{\sigma\lambda_0\sqrt{s_*}}}
        +
            {e}^{\frac{3u W'}{\sigma\lambda_0\sqrt{s_*}}}
        \right)
        \right].
    \end{align*}
    The right hand side is bounded from above thanks to \eqref{previous-display-Theorem-6.2-proof}.
    We apply the same technique to obtain the desired bound on ${\Rem_{II}}$, using \Cref{lemma:W''}
    for $W''$.
\end{proof}

\subsection{Proof of \Cref{thm:main}}
\label{sec:proof-main-part-(ii)}
From \Cref{theorem:slepian-expansion},
in order to complete prove \Cref{thm:main} we will need the following additional lemma.
\begin{restatable}{lemma}{lemmaLastTermDivergence}
    \label{lemma:last-term-divergence}
    The upper bound
    \begin{align*}
        &\E\left[
    I_{\Omega_1\cap\Omega_2} \left(\int_0^{\pi/2}(\sin t)(|\Shat(t)| - |\Shat(0)|)dt \right)^2\right]
    &\le
    n \left( \lambda_0^2s_* C'''' + 6( 3 + 2M_1^2 \lambda_0^2 s_*)\right)
    \end{align*}
    holds,
    where $C'''' = 3( M_5 M_1 + M_2  M_5 M_4)^2$.
\end{restatable}

\begin{proof}[Proof of \Cref{thm:main}]
    Thanks to the scale equivariance \eqref{scale-equi-variant},
    we take the scale $C_0=\|\bSigma^{-1/2}\ba_0\|_2=1$ without loss of generality, so that \eqref{standard} holds.
Let ${\Rem_{II}}$ be defined in \Cref{theorem:slepian-expansion}. 
Then for any degrees-of-freedom adjustment $\nu$ we have
\begin{align*}
    &\sqrt{F_\theta n}(1-\nu/n)(\htheta_\nu-\theta)
    -
    T_n
    +
    \sqrt{F_\theta/n}\langle \ba_0, \bhlasso \rangle (\nu - |\Shat|)  \\
    &= \sqrt{F_\theta/n}\langle \ba_0, \bhlasso \rangle
    \int_0^{\pi/2}(\sin t)(|\Shat(t)|-|\Shat|)dt
    +{\Rem_{II}}.
\end{align*}
Denote by $\Rem_{final}$ the above quantity. Then
\begin{equation*}
\E\left[
I_{\Omega_1\cap\Omega_2}
\left|\frac{\Rem_{final}}{\lambda_0\sqrt{s_*}}\right|^2
\right]
\le\begin{cases}
    2 M_2^2 \E\left[I_{\Omega_1\cap\Omega_2}
        \left(\int_0^{\pi/2}  (\sin t)(|\Shat(t)|-|\Shat|)  n^{-1/2}  dt\right)^2
\right] \\
+ 2 \E\left[I_{\Omega_1\cap\Omega_2}{\left\{ \Rem_{II}/(\lambda\sqrt{s_*}) \right\}  }^2\right].
\end{cases}
\end{equation*}
By \Cref{theorem:slepian-expansion},
$\E\left[I_{\Omega_1\cap\Omega_2}{\Rem_{II}}^2\right]$ is bounded 
by a constant that depends on $M_1,M_2,M_4, M_5$ only.
By
\Cref{lemma:divergence}
and the assumption $\lambda_0\sqrt{s_*}\le 1$ in \Cref{assumption:main},
the same holds for the first term.
Observe that since $\P(\Omega_1\cap\Omega_2)\to 1$, any random variable $Y$
such that $\E[I_{\Omega_1\cap\Omega_2}Y^2] \le C \lambda_0^2 s_*$ for some constant $C$
satisfies $Y=O_\P(\sqrt{s_*} \lambda_0)$ by Markov's inequality.
This shows that $\Rem_{final} = O_\P(\lambda_0\sqrt{s_*})$ and the proof is complete.
\end{proof}

\subsection{Proof of \Cref{corollary:unadjusted}}
\label{sec:proof-corollary-unadjusted}
On $\Omega_2$ we have $|\Shat|\le {s_*}$ 
and $|\langle \ba_0,\bhlasso\rangle | \le M_2 \sigma \lambda_0 \sqrt{s_*}$
so the claim of \Cref{corollary:unadjusted} follows from the same
argument as the previous subsection.

\section*{Funding}
P.C.B. was partially supported supported by the NSF Grants DMS-1811976
    and DMS-1945428.

C-H.Z. was
partially supported by the NSF Grants DMS-1513378,
  IIS-1407939, DMS-1721495, IIS-1741390 and CCF-1934924.


\bibliographystyle{amsalpha}
\bibliography{dof-slepian}

\newpage

\section*{Supplement}

\appendix

\section{Bounds for the false positive and proof of \Cref{lm-7-probability-of-Omega_1-Omega_2}}
\label{sec:proof-probabilistic-lemma}

We require first a few lemmas.
The following \Cref{lemma:full-rank-2m} shows that with probability one,
$\bX_A(t)$ is full-rank for all $t\ge 0$ and all sets $A$ of small enough cardinality.
\Cref{lm-5-new-iso,lemma:chi2-(t)}
provide uniform bounds for sparse eigenvalues of the random matrix family 
$\{\bX^\top(t)\bX(t)/n,\ t\ge0\}$ and some closely related quantity. 
\Cref{proposition:false-positive-probability-bound}
provides some tail-probability bound for the noise uniformly over all $t\ge0$,
as well as a bound on the number of false positives in $\supp(\hbbeta(t))$.
\Cref{lm-7-probability-of-Omega_1-Omega_2} will be finally proved in
\Cref{sec:finally-proof-of-lemma-probability-bounds}. 

\begin{restatable}{lemma}{lemmaFullRank}
    \label{lemma:full-rank-2m}
    (i) Almost surely, $\bz_0(t)\ne 0\ \forall t$, 
    that is, $\P(\inf_{t\ge 0}\|\bz_0(t)\|_2 > 0) = 1$.

    (ii)
    If  \eqref{condition-for-full-rank-2m} holds
     and $\rank(\bSigma_{A,A}) = |A|$ for all sets $A$ such that
     $|A\setminus S|\le 2(m+k)$, then for all such sets $A$,
$\E\big\{\P\big[\inf_{t\ge 0}\,\rank(\bX_A(t)) = |A| \,\big|\bep \big]\big\} = 1$.
\end{restatable}

\begin{restatable}{lemma}{lemmaChiSquare}
    \label{lemma:chi2-(t)}
    Let $\Omega_{\chi^2}(\ba_0)$ be the event
    $$\Omega_{\chi^2}(\ba_0) = \left\{
        \max_{0\le t\le \pi/2} \Big|C_0\|\bz_0(t)\|_2 - \sqrt{n} \Big| \le \eta_3\sqrt{n} \right\}.$$
        Then, $1 - \P(\Omega_{\chi^2}(\ba_0)) \le 2e^{-(\eta_3-\sqrt{2/n})_{+}^2n/2}$. 
\end{restatable}

\begin{restatable}{lemma}{lemmaIso}
    \label{lm-5-new-iso}
    Let $\{m,n,p\}$ be positive integers and 
    positive reals $\constants$ such that
\eqref{conditions-epsilons-1}
and \eqref{conditions-epsilons-2}
hold.
Let $\{\tau_*,\tau^*\}$ be defined \Cref{assumption:main}.
Then the event $\Omega_{iso}(\ba_0)$
defined by
\begin{align}
    &\Omega_{iso}(\ba_0) = \Big\{\tau_*  \le \big\|\bX(t)\bu\big\|_2^2/n \le \tau^*,\ \forall \bu\in \scrU(m+k,S;\bSigma),\ t>0\Big\} 
\label{def-Omega_iso}
\\
&\text{ satisfies }\quad
\P(\Omega_{iso}(\ba_0)) \ge 1 - 3e^{-n\eps_4} 
\nonumber
\end{align}
where $\scrU(m+k,S;\bSigma) = \{\bu: \|\bSigma^{1/2}\bu\|_2 =1, |\supp(\bu)\setminus S| \le m+k\}$. 
\end{restatable}

\subsection{Deterministic bounds on the false positives}
\label{sec:proof-false-positive-1}
In this subsection, the argument is fully deterministic.
Recall the definition of the sparse condition number in \eqref{def-condition-number}.
Consider the following condition: 
for $\bar S\subset \{1,\ldots, p\}$, $\bSigmabar$ a $p\times p$ 
positive semi-definite matrix, $0\le \eta_1<\infty$ and $\eta_2\in [0,1)$ and integer $m\le p-|\bar S|$,
\bel{SRC-tilde-S}
|\bar S| < \frac{ 2(1-\eta_2)^2m}{(1+\eta_1)^2\big\{\phi_{\rm cond}(m;\bar S,\bSigmabar)- 1\big\}} 
.
\eel

\begin{restatable}{proposition}{propositionFPbetabar}
    \label{proposition:eta_1-eta_2-bar-beta}
     Let $\eta_1>0$, $\eta_2\in(0,1)$, $\mu_0>0$ be constants.
    Assume that for some subset $\bar S\subset [p]$
    and vector $\bar\bbeta$ we have
    \begin{align}
      &  \bar S \supseteq \supp(\bar\bbeta)\,\cup\{j\in[p]: |\bar\bx_j^\top(\by-\bar\bX\bbetabar)/n|  \ge \eta_2\mu_0\}, 
    \label{def-S-bar} \\
    & \|\bar\bX_{\bar S}^\top(\by - \bar\bX\bbetabar)\|_2/n \le \eta_1 \mu_0|\bar S|^{1/2}.
    \label{condition-eta_1}
    \end{align}
    If condition \eqref{SRC-tilde-S} holds for $\bar S$, $\bSigmabar=\bar\bX^\top\bar\bX/n$
    and some $m$, then for any tuning parameter $\lambda\ge \mu_0$,
    the Lasso estimator ${\lasso}$ with response $\by$ and design $\bar\bX$ satisfies
    \begin{equation}
        \label{conclusion-support-lasso}
        |\supp({\lasso})\setminus \bar S | \le\frac{\{\phi_{\rm cond}(m;\bar S,\bSigmabar)-1\}| \bar S |}{2(1-\eta_2)^2/(1+\eta_1)^2}
        <
        m.
    \end{equation}
\end{restatable}

Since the argument in \Cref{proposition:eta_1-eta_2-bar-beta} is purely deterministic,
we may later apply this proposition to random $\eta_1,\eta_2,\mu_0,\bar S$.
In this case the conclusion \eqref{conclusion-support-lasso} holds
on the intersection of the events \eqref{SRC-tilde-S}, \eqref{def-S-bar} and \eqref{condition-eta_1}. 
The main ingredient to prove the above proposition is the following.

\begin{restatable}[Deterministic Lemma]{lemma}{lemmaFPdeterministic}
    \label{lem-false-positive} Suppose the SRC holds with $\bSigmabar$ in (\ref{SRC-tilde-S}) 
    replaced by $\bSigma$ and $\bar S$ replaced by $S$. Then, 
    \bes
    \frac{\|\bu_{{S}^c}\|_1\phi_{\max}(\bSigma_{{S},{S}})}{1-\eta_2}
    + \big|\supp(\bu)\setminus {S}\big|  
    \le \frac{\{\phi_{\rm cond}(m;{S},\bSigma)-1\}|{S}|}{2(1-\eta_2)^2/(1+\eta_1)^2} <m. 
    \ees
    for all $\bu\in \scrU_0({S},\bSigma;\eta_1,\eta_2)$ where $\scrU_0({S},\bSigma;\eta_1,\eta_2)$ is given by 
    \bel{srcU_0-norm-2}
    && \scrU_0({\bar S}, \bSigmabar, \eta_1,\eta_2) \\
    &=&\Big\{\bu: 
    \big| u_j (\bSigmabar\bu)_j + |u_j|\big| \le \eta_2 |u_j|\ \forall j\not\in \bar S,\;\; 
    \big\|(\bSigmabar\bu)_{\bar S}\big\|_2 
    \le (1+\eta_1)|{\bar S}|^{1/2}\Big\}. 
    \nonumber
    \eel
\end{restatable}

Lemma \ref{lem-false-positive} improves upon Lemma 1 of 
\cite{Zhang10-mc+}
in the special case of 
the Lasso by including the term $\|\bu_{{S}^c}\|_1\phi_{\max}(\bSigma_{{S},{S}})/(1-\eta_2)$ 
on the left-hand side and allowing general $\eta_1>0$ not dependent on $\bSigma$.  
The proof there, which covers concave penalties as well as the Lasso, 
is modified to keep the two additional items as follows. 

\subsection{Tail-probability bounds for the false positives}
\label{sec:proof-false-positive-2}

Note that on the event $\Omega_{iso}(\ba_0)$ defined in \eqref{def-Omega_iso},
the empirical condition number does not expand by more than $\tau^*/\tau_*$,
i.e., for all $t\ge0$,
\begin{equation}
    \label{condition-number-X(t)}
    \phi_{\rm cond}(m+k; S, \bX(t)^\top\bX(t) /n) 
    \le (\tau^*/\tau_*) 
    \phi_{\rm cond}(m+k; S, \bSigma).
\end{equation}
We can now give a bound on the false positives of $\hbbeta(t)$ uniformly over all $t\ge0$ and with high probability.

\begin{restatable}{proposition}{propositionFPprobability}
    \label{proposition:false-positive-probability-bound}
    Let $\lambda,\bep,\bz_0(t), \bX(t), \hbbeta(t)$ be as in \Cref{sec:proof-path}.
    Let $m,k > 0$ and assume that \eqref{SRC-population-final} holds.

    (i) Let $\eta_2\in (0,1)$ and define for some $L_k>0$ the random variable
    \begin{equation}
        \label{def-mu_0}
        \mu_0 = \eta_2^{-1} (\|\bep\|_2/n) (L_k + (L_k^2 + 2)^{-1/2}).
    \end{equation}
    Consider the two events
    \bel{event-noise-1}
    &\Omega_{noise}^{(1)} = &\Big\{
            \sum_{j=1}^p \left(|\bx_j^\top\bep| - \|\bep\|_2L_k\right)_+^2
            < \frac{k \|\bep\|_2^2}{L_k^2 + 2}
        \Big\},
        \\
        &\Omega_{noise}^{(2)} = &\Big\{
            \|\bX_S^\top \bep\|_2 < \|\bep\|_2 |S|^{1/2}(L_k+(L_k^2+2)^{-1/2})
        \Big\}
        .
        \label{event-noise-2}
    \eel
    On the intersection of the four events $\{\mu_0 \le \lambda \}$,
    $\Omega_{iso}(\ba_0)$, \eqref{event-noise-1} and \eqref{event-noise-2},
    the set $\tilde S = S \cup \{j\in [p]: |\bx_j^\top\bep|/n \ge \eta_2 \mu_0 \}$ satisfies
\begin{equation}
    |\tilde S \setminus S| < k
    \qquad\text{and}\qquad
    \max_{t\ge 0}\big|\supp(\hbbeta(t))\setminus \tilde S\big| < m.
    \label{conclusion-support-lasso-purpose-present-paper}
\end{equation}

(ii) If $L_k = \sqrt{2\log(p/k)}$ then $\Omega_{noise}^{(1)}\cap\Omega_{noise}^{(2)}$
has probability at least
\begin{equation}
    1-4(2\pi L_k^2+4)^{-1/2} - (L_k+(L_k^2+2)^{-1/2})^{-2}.
    \label{bound-probability-event-supprt-lasso}
\end{equation}
\end{restatable}

The probability in \eqref{bound-probability-event-supprt-lasso}
decreases logarithmically in $p/k$. Although for simplicity we do not try
to improve this probability, let us mention some known techniques that can be applied to improve it.
A first approach uses
$L_k=\sqrt{(1+\alpha)2\log(p/k)}$ with $\alpha>0$ 
as in the proof of \Cref{proposition:false-positive-probability-bound}(ii)
in which case the right hand side decreases polynomially in $p/k$.
Another approach is to use probability bounds in \cite[Proposition 10]{sun2013sparse}
which requires the upper sparse eigenvalue of $\bSigma$ to be bounded.
Finally, for prediction and estimation bounds, the argument of \cite[Theorem 4.2]{bellec2016slope}
can be used to derive exponential probability bounds from
bounds on the median (i.e., with probability $1/2$).

\subsection{Proof of \Cref{lm-7-probability-of-Omega_1-Omega_2}}
\label{sec:finally-proof-of-lemma-probability-bounds}
Note that for any $\eta_2\in(0,1)$, $\mu_0$ in \eqref{def-mu_0} with $L_k=\sqrt{2\log(p/k)}$
satisfies
\begin{equation}
    \label{upper-bound-mu_0}
\eta_2\mu_0 =  (\|\bep\|_2/n) (L_k + (L_k^2 + 2)^{-1/2})
\le (\|\bep\|_2 / n) \sqrt{2\log(8p/k)}.
\end{equation}
Hence if $\lambda = (1+\eta_3)\sigma\eta_2 \sqrt{(2/n)\log(8p/k)}$ as in
\eqref{eq:lambda-larger-than-sigma-lambda_0-over-alpha}
then $\lambda\ge \mu_0$ holds on the event 
\begin{equation}
    \label{event-noise-3}
    \Omega_{noise}^{(3)} = 
    \left\{ 
        \|\bep\|_2 < (1+\eta_3)\sigma\sqrt n 
    \right\}
    \text{ for which }
    \P(\Omega_{noise}^{(3)}) \ge 1 - e^{-n\eta_3^2/2}.
\end{equation}
Here, the probability bound is a classical deviation bound for $\chi^2$
random variables with $n$ degrees of freedom.
We are now ready to prove \Cref{lm-7-probability-of-Omega_1-Omega_2}. 

\begin{proof}[Proof of Lemma \ref{lm-7-probability-of-Omega_1-Omega_2}]
    Define the event $\Omega(\ba_0)$ by
    \bes
        &&\Omega(\ba_0)=\Omega_{noise}^{(1)}
    \cap\Omega_{noise}^{(2)}
    \cap\Omega_{noise}^{(3)}
    \cap\Omega_{iso}(\ba_0)
    \cap\Omega_{\chi^2}(\ba_0),
    \ees
    where 
    $\Omega_{noise}^{(1)},
    \Omega_{noise}^{(2)},
    \Omega_{noise}^{(3)},
    \Omega_{iso}(\ba_0)$ and 
    $\Omega_{\chi^2}(\ba_0)$
    are defined in \eqref{event-noise-1}, \eqref{event-noise-2},
    \eqref{event-noise-3}, \eqref{def-Omega_iso}
    and \Cref{lemma:chi2-(t)}.
    By \eqref{event-noise-3}, \Cref{proposition:false-positive-probability-bound}(ii),
    \Cref{lemma:chi2-(t),lm-5-new-iso} and the union bound,
    $1-\P(\Omega(\ba_0))$ is bounded from above by \eqref{upper-bound-proba-Omega_1_cap_Omega_2}.

In the rest of the proof, we prove $\Omega(\ba_0)\subset\Omega_2$ for the given 
$\{M_1,\ldots,M_8\}$ by checking the conditions in (\ref{conds-limited}), 
and prove that $\P\{\Omega(\ba_0)\setminus \Omega_1\} = 0$. 
Assume $\Omega(\ba_0)$ happens hereafter. 

On $\Omega_{noise}^{(3)}$ we have $\lambda \ge \mu_0$ where $\mu_0$ is defined in \eqref{def-mu_0}.
Let $\Shat(t)=\supp(\hbbeta(t))$.
The conditions of \Cref{proposition:false-positive-probability-bound}(i) are satisfied
hence for all $t>0$, 
\bes
\big|\Shat(t)\setminus S\big| \le m+k. 
\ees
This gives $\big|\Shat(t)|\le s_*$ in (\ref{conds-limited}). The specified $M_{3}$ is allowed 
as $\phi_{\rm cond}(m+k;S,\bSigma)\le \phi_{\rm cond}(p;\emptyset,\bSigma)$ in 
\eqref{SRC-population-final}. 
Consequently, thanks to $\Omega_{iso}(\ba_0)$ in \eqref{def-Omega_iso} 
we have $M_4=1/\tau_*$ in (\ref{conds-limited}).
We note that $M_5=1/(1-\eta_3)$ 
thanks to  the event 
$\Omega_{noise}^{(3)}\cap \Omega_{\chi^2}(\ba_0)$ in \eqref{event-noise-3} and \Cref{lemma:chi2-(t)}.
As $s_0+2(m+k)+1\le (n-1)\wedge ({p+1})$, for any $t,t'\ge0$ the set $B=\Shat(t)\cup\Shat(t')$
satisfies $|B\setminus S | \le 2{(m+k)}$ so
$\P\{\Omega(\ba_0)\setminus \Omega_1\}=0$ by \Cref{lemma:full-rank-2m}. 
It remains to give $M_1$ and $M_2$ in (\ref{conds-limited}). 

Let $A = \Shat(t)\cup S$ and note that $|A|\le |S|+m+k$. The KKT conditions imply
\bes
\Big(\bX_A^\top(t)\bX_A(t)/n\Big)\bh_A(t) = \bX_A^\top(t)\bep/n - \lam\pa\|\hbbeta_A(t)\|_1. 
\ees
Multiplying both sides by $\bh_A(t)=\hbbeta_A(t)-\bbeta_A$, we find that 
\bes
 \|\bX(t)\bh(t)\|_2^2/n 
&=& \big\langle \bX_A(t)\bh_A(t),\bep\big\rangle\big/n - \lam\big\langle \bh_S(t),\pa\|\hbbeta_S(t)\|_1\big\rangle
- \lambda \|\hbbeta_{\Shat\setminus S}\|_1
\cr &\le& \|\bh(t)\|_2\|{\bX_A^\top}(t)\bep\|_2/n + \lam\|\bh_S\|_2\sqrt{|S|}. 
\ees
By \eqref{def-mu_0},
\eqref{event-noise-1},
\eqref{event-noise-2},
\eqref{conclusion-support-lasso-purpose-present-paper},
$\|\bX_{{A}}^\top(t)\bep\|_2/n \le\eta_2\mu_0(|S|+k+m)^{1/2}$
as in the proof of \Cref{proposition:false-positive-probability-bound}(i) 
[cf. \eqref{pf-dim-bd-2} there with $\tilde S$ replaced by $A$].
Thus, as $\mu_0\le\lam$ in the event $\Omega_{noise}^{(3)}$ in \eqref{event-noise-3}, 
\bes
\|\bX(t)\bh(t)\|_2^2/n \le \|\bh(t)\|_2\eta_2\lam(|S|+m+k)^{1/2}+\lam\|\bh_S\|_2\sqrt{|S|}.
\ees
On $\Omega_{iso}(\ba_0)$ in \eqref{def-Omega_iso} we have
$\|\bh(t)\|_2 \le \|\bX(t)\bh(t)\|_2/(\sqrt{n \tau_*\rho_*})$.
This gives $M_1= (1+\eta_2)\lambda/(\sigma\lambda_0 \sqrt{\tau_*\rho_*})$ in \eqref{conds-limited}.
Thanks to $\Omega_{iso}(\ba_0)$ we get
$M_2 = M_1/\sqrt{\tau_*}$.
The same argument applies verbatim to $\bbeta^{(noiseless)}(t)$ and $\bh^{(noiseless)}(t)$
which provides the second line in \eqref{conds-limited}.
The proof is complete. 
\end{proof}

\section{Proof of \Cref{thm:unadjusted-ell1-additional-assumption} and \Cref{thm:infty-norm}}
\label{sec:proof-thm-infty-norm}

\begin{proof}[Proof of \Cref{thm:unadjusted-ell1-additional-assumption}] 
We assume without loss of generality $\|\bSigma^{-1/2}\ba_0\|_2=1$, so that $C_0=1$, $F_\theta = 1/\sigma^2$ and 
$\|\bu_0\|_1=\|\bSigma^{-1}\ba_0\|_1\le \min(K_{0,n,p},K_{1,n,p}\sqrt{n/s_*})$. 
    By the definition of $\htheta_{\nu}$ and \Cref{theorem:slepian-expansion} 
    with $\nu=0$ we have 
    \bes
        &&\sqrt{nF_\theta} 
        \left\langle \ba_0, \lasso - \bbeta \right\rangle
        \Big(1-\int_0^{\pi/2}\frac{\sin(t)|\Shat(t)|}{n}dt\Big)
     \\   &=& T_n + {\Rem_{II}}
        - \sqrt{n F_\theta} \langle\bz_0, \by - \bX\lasso\rangle \|\bz_0\|_2^{-2}
        .
    \ees
    where $T_n = n^{1/2}\bz_0^\top \bep /(\sigma C_0 \|\bz_0\|_2^2)$ 
    and ${\Rem_{II}} = O_{\P}(\lam_0\sqrt{s_*}/\sigma)$. 
    Let $Z(\ba_0)=\bz_0^\top \bep/(\sigma \|\bz_0\|_2)\sim N(0,1)$.
    By the KKT conditions of the Lasso,
    $|\langle\bz_0, \by - \bX\lasso\rangle| \le \|\bu_0\|_1\lam n.$
    Furthermore, on $\Omega_1\cap\Omega_2$,
    inequality $|\Shat(t)| \le {s_*}$ holds all $t\ge 0$ as well
    as $1/\|\bz_0\|_2 \le  M_5 /\sqrt n$.
    Hence, on $\Omega_1\cap\Omega_2$ we have proved
    that
    \bel{Rem-9-upper-bound-l_infty} 
        && \left| \left\langle \ba_0, \lasso - \bbeta
        \right\rangle \right| 
        \cr &\le& 
 \frac{ M_5^2  \|\bSigma^{-1}\ba_0\|_1 \lambda}{1- {s_*}/n}
        +
        \frac{\sigma (M_5|Z(a_0)| + |{\Rem_{II}}|)}{(1- {s_*}/n)\sqrt{n}}
\\ \nonumber &\le& 
        \frac{M_5^2\lam \min(K_{0,n,p},K_{1,n,p}\sqrt{n/s_*})}{1- {s_*}/n}
        +
        \frac{O_{\P}(\sigma)}{(1- {s_*}/n)\sqrt n}, 
    \eel
    where under \Cref{assumption:main}, 
    $s_*/n\le \epsilon_1^2/2$ thanks to \eqref{conditions-epsilons-2} 
    so that $1/(1-{s_*}/n)=O(1)$.  
    Thus, $\sqrt{F_\theta}\langle \ba_0, \lasso - \bbeta\rangle = O_{\P}(\lam_0 K_{0,n,p} + n^{-1/2})$ 
    and 
    \bes
    && \Big|\sqrt{n F_\theta}(\htheta_{\nu=0}-\theta)  - T_n\Big|
    \cr &=&\bigg| {\Rem_{II}} + \sqrt{F_\theta / n} 
    \left\langle \ba_0,\bhlasso \right\rangle \int_0^{\pi/2}(\sin t)\left|\Shat(t)\right| dt\bigg|
    \cr &=&
    O_{\P}\Big( \lam_0\sqrt{s_*} + n^{-1/2}\big(\lambda_0 K_{1,n,p}\sqrt{n/s_*} 
        + n^{-1/2}\big) s_* \Big), 
    \ees
    which is of the order $O_{\P}\big((1+K_{1,n,p})\lam_0\sqrt{s_*} + s_*/n\big)=o_{\P}(1)$.   
\end{proof}


\begin{proof}[Proof of  \Cref{thm:infty-norm}]
    For each $j=1,...,p$, we define
    an interpolation path as in \eqref{bz(t)}
    for $\ba_0=\be_j$ (so that we have $p$ different interpolation paths),
    and define the events $\Omega_1(\be_j)$ and $\Omega_2(\be_j)$
    to be the events \eqref{def-Omega_1} and \eqref{def-Omega_2} when $\ba_0=\be_j$.
    Similarly, define the events $\Omega_{iso}(\be_j)$
    and $\Omega_{\chi^2}(\be_j)$ as in \Cref{lm-5-new-iso}
    and \Cref{lemma:chi2-(t)} with $\ba_0=\be_j$.
    Note that the events $\Omega_{noise}^{(i)}, i=1,2,3$
    from \Cref{proposition:false-positive-probability-bound} and \eqref{event-noise-3} 
    do not depend on $\ba_0$. 
    Define 
    $$\bar\Omega_1 = \cap_{j=1}^p \Omega_1(\be_j),
    \qquad
    \bar\Omega_2 = \cap_{j=1}^p \Omega_2(\be_j)
    $$
    as well as
    $$\bar\Omega
    =
        \Omega_{noise}^{(1)}
    \cap\Omega_{noise}^{(2)}
    \cap\Omega_{noise}^{(3)}
    \cap\left(\cap_{j=1}^p\Omega_{iso}(\be_j) \right)
    \cap\left(\cap_{j=1}^p\Omega_{\chi^2}(\be_j) \right)
    .
    $$
    We established in the proof of \Cref{lm-7-probability-of-Omega_1-Omega_2} that for each $j=1,...,p$,
    $$
    \Big(
        \Omega_{noise}^{(1)}
    \cap\Omega_{noise}^{(2)}
    \cap\Omega_{noise}^{(3)}
    \cap\Omega_{iso}(\be_j)
    \cap\Omega_{\chi^2}(\be_j)
    \Big)
    \subset
    \Omega_2(\be_j),$$
    which implies the inclusion
    $\bar\Omega \subset \bar\Omega_2 \subset \Omega_2(\be_j)$.
    We also established in \Cref{lm-7-probability-of-Omega_1-Omega_2} that
    $$\P\left(
    \left(
    \Omega_{noise}^{(1)}
    \cap\Omega_{noise}^{(2)}
    \cap\Omega_{noise}^{(3)}
    \cap\Omega_{iso}(\be_j)
    \cap\Omega_{\chi^2}(\be_j)
    \right)
    \setminus \Omega_1(\be_j)\right) = 0,$$
    hence $\P(\bar\Omega \setminus \bar\Omega_1(\be_j))
    = 0$ and $\P(\bar\Omega \setminus \bar\Omega_1)=0$.
    Finally,
    $\P(\bar\Omega_1^c \cup \bar\Omega_2^c) \le \P(\bar\Omega^c)$
    and we bound the probability of $\bar\Omega^c$ with the union bound over $j=1,...,p$
    to obtain
    \begin{equation*}
        \begin{split}
        1-\P(\bar\Omega_1\cap\bar\Omega_2) \le
        &\quad p\Big(3 e^{-n\eps_4}
        + 2e^{-(\eta_3-\sqrt{2/n})_{+}^2n/2}\Big)
        \\
        &+ e^{-n\eta_3^2/2} 
        +
        4(2\pi L_k^2+4)^{-1/2} + (L_k+(L_k^2+2)^{-1/2})^{-2}
        .
        \end{split}
    \end{equation*}
    Indeed, since each $\Omega_{noise}^{(i)}$ is independent of $\ba_0$, the factor $p$
    from the bound is only paid for $\Omega_{iso}(\be_j)$ and $\Omega_{\chi^2}(\be_j)$.

    For each $j=1,...,p$, define the quantity ${\Rem_{II}}(\be_j)$ as the quantity
    ${\Rem_{II}}$ from \Cref{theorem:slepian-expansion} when $\ba_0=\be_j$. 
    Thanks to \eqref{Rem-9-upper-bound-l_infty} applied to $\ba_0=\be_j$,
    on $\bar\Omega_1\cap\bar\Omega_2$ we have
    simultaneously for all $j=1,...,p$,
    \[
        \left| \hbeta^{(lasso)}_j - \beta_j \right| 
        \le
        \frac{ M_5^2  \|\bSigma^{-1}\be_j\|_1 \lambda}{1- {s_*}/n}
        +
        \frac{\sigma\|\bSigma^{-1/2}\be_j\|_2}{(1- {s_*}/n)\sqrt n}\left(
        M_5|Z(\bfe_j)| + |{\Rem_{II}}(\be_j)|\right)
    \]
    It remains to bound  $|Z(\be_j)|$ and $|{\Rem_{II}}(\be_j)|$ uniformly over all $j=1,...,p$.
    For any $\be_j$, $Z(\be_j)\sim N(0,1)$ so
    \begin{equation}
        \P(\max_{j=1,...,p}|Z(\be_j)| > 2\sqrt{\log p}) \le 1/p.
        \label{max-of-gaussians-event}
    \end{equation}
    To bound $\max_{j=1,...,p}|{\Rem_{II}}(\be_j)|$,
    by \Cref{theorem:slepian-expansion}
    we have
    for any $u\in(0,\sqrt n /\bar M)$ and by Markov's inequality,
    \begin{align}
        \label{max-of-Rem_j-event}
        &\P\left(
    \bar\Omega_1\cap\bar\Omega_2 \cap
    \left\{
        \max_{j=1,...,p} \frac{|{\Rem_{II}}(\be_j)|}{\lambda_0\sqrt{s_*}} >  3 \bar M \sqrt{\log p}
    \right\}
    \right)
    \\
    &\le
    e^{-u 3 \bar M \sqrt{\log p }}
    \sum_{j=1}^p
    \E\left[
    I_{\bar\Omega_1\cap\bar\Omega_2}
    \exp\left({ \frac{u|{\Rem_{II}}(\be_j)|}{\lambda_0\sqrt{s_*}}}\right)
    \right]
    \nonumber
    \\
    &\le 
    e^{-u 3 \bar M \sqrt{\log p }}
    4p
    \,
    \exp\left(\bar M^2 u^2 \right).
    \nonumber
    \end{align}
    For $u=\sqrt{\log p}/\bar M$, the right hand side of the previous display
    equals $4/p$.
    The union bound of $(\bar\Omega_1\cap\bar\Omega_2)^c$,
    \eqref{max-of-gaussians-event}, \eqref{max-of-Rem_j-event} shows that \eqref{bound-Lasso-on-coefficient-j}
    holds on an event of probability at least
    $1-5/p -\P((\bar\Omega_1\cap\bar\Omega_2)^c)$. 
\end{proof}

\section{Proofs of the results for unknown $\bSigma$}
\label{sec:proofs-unknown-Sigma}

The results of \Cref{sec:unknown-Sigma}
are restated before their proofs for convenience.

\PropDantzigUnknownSigma*
\begin{proof}[Proof of \Cref{prop:dantzig}] 
The KKT conditions for the Lasso estimator 
    \bes
    \hbgamma = \bQ \hbb
        \quad\text{ where }\quad\hbb 
    = \argmin_{\bb\in\R^p}\Big\{\|\bz - \bX\bQ\bb\|_2^2/(2n) + \lam\|\bQ\bb\|_1\Big\} 
    \ees
can be written as 
\bes
\big\langle \bX\bQ\bh,\bz - \bX\bQ\hbgamma \big\rangle\big/n 
\in \lam \big\langle \pa\|\bQ\hbgamma\|_1, \bQ\bh\big\rangle\,\qquad \forall \bh\in \R^p,
\ees
where $\pa \| \bv\|_1= \{\tbv\in\R^p : \|\tbv\|_\infty = 1, \langle \tbv,\bv\rangle = {\|\bv\|_1}\}$
is the subdifferential of the $\ell_1$ norm at $\bv$.
With $\bh=\hbgamma - \bgamma$ 
this implies the basic inequality 
\bes
&& \|\bX\bQ(\hbgamma-\bgamma)\|_2^2/n + (1-\eta)\lam \|(\bQ\hbgamma-\bQ\bgamma)_{T^c}\|_1
    \cr &\le & (1+\eta)\lam \|(\bQ\hbgamma-\bQ\bgamma)_{T}\|_1 + 2\lam\|(\bQ\bgamma)_{T^c}\|_1, 
\ees
when $\|(\bX\bQ)^\top\bz_0/n\|_\infty\le \eta\lam$ for some fixed $\eta\in (0,1)$. 
As the the analysis in \cite{sun2012scaled} completely relies on the above basic inequality, 
\eqref{eq:hbu-ell_1-rate} 
follows from the same analysis, provided that 
$\max_{1\le j\le p} \|\bX\bQ\bfe_j\|_2/n^{1/2} \le \eta A$
with high probability. This condition holds as $A>2$ and 
\begin{equation}
\label{upper-bound-XQ-e_k}
\forall k\in[p],\quad
\E \|\bX\bQ\bfe_{k}\|_2^2/n 
= \big(\bfe_{k} - \bfe_{j_0}a_{k}/a_{j_0}\big)^\top\bSigma \big(\bfe_{k} - \bfe_{j_0}a_{k}/a_{j_0}\big) \le 4. 
\end{equation}
Moreover,
the KKT conditions at the realized $\lam = \htau A\lamuniv$ automatically 
provide \eqref{eq:hbu-ell_infty} since $\htau= O_\P(C_0^{-1})$. 
\end{proof}

\thmUnknownSigma*
\begin{proof}[Proof of \Cref{thm:unknown-sigma}]
    We prove separately the two cases
    \begin{enumerate}
    \item
        $r_n = \|\bbeta\|_0 \log(p) /\sqrt n$, 
        \item 
            $r_n = \min\big\{\|\bu_0\|_0 \log(p)/\sqrt n , ~ C_0\|\bu_0\|_1 \sqrt{\log p} \big\}$. 
    \end{enumerate}
    Let $\bhlasso=\lasso-\bbeta$.  As preliminaries for the proof,
    we note that the following standard bounds for the Lasso $\lasso$ 
    hold:
    \bel{standard-bounds-unknown-Sigma}
        \big\|\lasso\big\|_0
        &=& O_\P(\|\bbeta\|_0),
        \label{eq:sparsity-bound-hbbeta}
        \\
        \sqrt{F_\theta}
        |\langle \ba_0,\bhlasso\rangle|
                            & = &
                            O_\P(\sqrt{\|\bbeta\|_0\log(p)/n}),
        \label{eq:prediction-population-hbbeta-a_0}
        \\
        {\|\bhlasso\|_1} &=& O_\P( \sigma \|\bbeta\|_0 \sqrt{\log(p)/n}),
        \label{eq:ell_1-hbbeta}
        \\
        \|\bX^\top(\by-\bX{\lasso})\|_\infty &\le& n \lambda  =O(\sigma\sqrt{\log(p)/n}).  
        \label{eq:KKT-hbbeta}
    \eel
    Inequality \eqref{eq:sparsity-bound-hbbeta} is the sparsity bound for ${\lasso}$ (cf.
    \Cref{proposition:false-positive-probability-bound}),
    \eqref{eq:prediction-population-hbbeta-a_0}
    follows from the population prediction bound
    $\|\bSigma^{1/2}\bhlasso\|_2^2
    =O_\P(\sigma^2 \|\bbeta\|_0 \log(p)/n)$
    for ${\lasso}$,
    \eqref{eq:ell_1-hbbeta} is the $\ell_1$ rate bound for
    ${\lasso}$ and
    \eqref{eq:KKT-hbbeta} {follows} from the KKT conditions for ${\lasso}$.
    Since $F_\theta = (C_0\sigma)^{-2}$ in \eqref{Fisher} we also have
    by the law of large numbers 
    that
    \begin{equation*}
        \|\bz_0\|_2 / \sqrt{ \sigma^2 n F_\theta} = C_0\|\bz_0\|_2/\sqrt n \to^\P 1,
    \end{equation*}
    Moreover, as $\bX\bQ$ and $\bz_0$ are independent,
        $\|\bz_0\|_2^{-1}\|(\bX\bQ)^\top \bz_0\|_\infty = O_\P(\sqrt{\log p})$
     thanks to \eqref{upper-bound-XQ-e_k}, so that
    \eqref{eq:hbu-ell_infty} and \eqref{eq:hbu-ell_1-rate} yield
    \bel{eq:consistency-hbz-X-a0}
    &&  \|\bz_0\|_2^{-2}
        |\langle\hbz,\bz\rangle- \|\bz_0\|_2^2|
          \cr&=&
    \|\bz_0\|_2^{-2}
    |\langle\hbz,\bz - \bz_0\rangle + \langle\bz_0, \hbz-\bz_0\rangle|
          \cr&\le&
          O_\P(\lamuniv) C_0 \|\bQ\bgamma\|_1
        +
        O_\P(\lamuniv) C_0 \|\bQ(\bgamma-\hbgamma)\|_1
          \cr&=&
          O_\P(\lamuniv) C_0 \|\bu_0\|_1
          \\&=&
          O_\P(1) \min\big\{
              \sqrt{\|\bu_0\|_0\log(p)/n},~ C_0\|\bu_0\|_1\sqrt{\log(p)/n}
          \big\}
        = o_\P(1) 
        \nonumber
    \eel
    thanks to $C_0\|\bu_0\|_1\le O(1)\|\bu_0\|_0^{1/2}$ and
    \eqref{sparsity-conditions-consistency} for the last line.
    Similarly, by \eqref{eq:hbu-ell_infty},
the triangle inequality 
$\|\bQ\hbgamma\|_1 \le \|\bQ\bgamma\|_1 + \|\bQ(\hbgamma-\bgamma)\|_1$,
\eqref{eq:hbu-ell_1-rate} and \eqref{bu-properties},
    \bes
     \|\bz_0\|_2^{-2}
        |\langle\hbz,\bz\rangle- \|\hbz\|_2^2|
          &=&
    \|\bz_0\|_2^{-2}
    |\langle\hbz,\bz - \hbz\rangle|
          \cr&\le&
          O_\P( \lamuniv ) C_0 \|\bQ\hbgamma\|_1
           \cr&=&
           O_\P( \lamuniv ) C_0 \|\bu_0\|_1
           = o_\P(1).
    \ees
    The three previous displays imply that
    both 
    \begin{equation}
        C_0^2 \langle \hbz,\bz\rangle/n = 1+o_\P(1),\qquad
        C_0^2 \|\hbz\|_2^2/n = 1+o_\P(1),
        \label{eq:variance-consistency-hbz-bz}
    \end{equation}
    hold i.e., $\langle \hbz,\bz\rangle/n$ and $\|\hbz\|_2^2/n$
    are consistent estimate of the noise variance $C_0^{-2}$. 

{\it Proof under regime (i).} 
    We have the decomposition
    \bes
    &&\sqrt{F_\theta n}\big\{
        (1-\nu/n)(\htheta_{\nu,\hbz}-\theta)
        - \langle\hbz,{\bz}\rangle^{-1}\hbz^\top \bep 
    \big\}
  \\&=&\sqrt{F_\theta n}
  \big\{
      (1- \nu/n) \langle \ba_0,\bhlasso\rangle 
      + \langle\hbz,{\bz}\rangle^{-1}
      \big(
          \hbz^\top(\by-\bX{\lasso})
          - \hbz^\top \bep
      \big)
  \big\}
    \\&=& 
     \sqrt{F_\theta n}\big\{
        -(\nu/n) \langle \ba_0, \bhlasso\rangle
    - 
    \langle\hbz,{\bz}\rangle^{-1}
    ~ \hbz^\top \bX{\bQ} \bhlasso\rangle
    \big\}.
    \ees
    We now show that the last line is {$O_\P(r_n)$}.
    The first term is {$O_\P(r_n)$} due to
    \eqref{eq:sparsity-bound-hbbeta} 
    and
    \eqref{eq:prediction-population-hbbeta-a_0}.
    The second term is {$O_\P(r_n)$} thanks to
    \eqref{eq:hbu-ell_infty}, \eqref{bu-properties}, \eqref{eq:ell_1-hbbeta}
    and \eqref{eq:consistency-hbz-X-a0}.
    On the first line,
    $\sqrt{F_\theta n}
    \langle\hbz,\bz\rangle^{-1}
    \hbz^\top \bep$ is normal conditionally
    on $\bX$ with  
    conditional variance $
    C_0^{-2} n \|\hbz\|_2^2|\langle \hbz,\bz\rangle|^{-2}
    = 1+o_{\P}(1)$ 
    thanks to  \eqref{eq:variance-consistency-hbz-bz}.

{\it Proof under regime (ii).} 
    In the following, 
    $\htheta_\nu$ is the estimate \eqref{LDPE-df} when $\bSigma$ is known
    with the ideal score vector $\bz_0$. We have the decomposition
    \bes
      &&\sqrt{nF_\theta}(1-\nu/n)(\htheta_{\nu,\hbz}-\theta)
      - \sqrt{nF_\theta}(1-\nu/n)(\htheta_{\nu}-\theta)
    \\&=&
    \sqrt{nF_\theta}
      \Big\langle
          \frac{\hbz}{\langle \hbz, {\bz}\rangle } - \frac{\bz_0}{\|\bz_0\|_2^2},
          \by-\bX{\lasso}
      \Big\rangle
    \\&=&
    \sqrt{nF_\theta}
    \Big(
        \frac{1}{\langle \hbz, {\bz}\rangle 
        }
        - \frac{1}{\|\bz_0\|_2^{2}}
    \Big) 
    \Big\langle \bz_0, \by-\bX{\lasso} \Big\rangle
      +
      \frac{\sqrt{n F_\theta}}{\langle \hbz, {\bz}\rangle}
    \Big\langle \hbz-\bz_0, \by-\bX{\lasso}\Big\rangle.
    \ees
    Since $\sqrt{nF_\theta}(1-\nu/n)(\htheta_{\nu}-\theta)\to^d N(0,1)$
    when $\nu=|\Shat|$
    by \Cref{thm:teaser}, it is sufficient to prove that the two terms
    in the last line are {$O_\P(r_n)$}.
    For the first term, first note that
    $
    |\langle \bz_0, \by-\bX{\lasso} \rangle|
    \le \|\bu_0\|_1 \lambda n$
    and we use
    \eqref{eq:consistency-hbz-X-a0}
    to bound $|\langle \hbz, \bz\rangle/\|\bz_0\|_2^2-1|$. 
    This and \eqref{eq:variance-consistency-hbz-bz} provide 
    that the first term in the above decomposition is 
    $O_\P(C_0^2\|\bu_0\|_1^2 \log(p)/\sqrt n)$ with
    \begin{equation*}
        \frac{C_0^2\|\bu_0\|_1^2 \log p}{\sqrt n}
        =
        \begin{cases}
            O(\|\bu_0\|_0 \log(p)/\sqrt n) = O(r_n) 
            & \text{ if } r_n = \|\bu_0\|_0\log(p)/\sqrt n,
            \\
            O(r_n \eps_n) = o(r_n) 
            &\text{ if } r_n= C_0\|\bu_0\|_1 \sqrt{\log p}
        \end{cases}
    \end{equation*}
    where 
    we used $C_0\|\bu_0\|_1 = O(1) \|\bu_0\|_0{}^{1/2}$ for the first line
    and the upper bound 
    \eqref{sparsity-conditions-consistency} for the second line.
    For the second term, 
     we have 
    by \eqref{eq:KKT-hbbeta}, \eqref{eq:hbu-ell_1-rate} and \eqref{bu-properties},     
    \bes
    |\langle \hbz-\bz_0, \by-\bX{\lasso}\rangle/n|
    &\le& \lam \|\bQ(\hbgamma-\bgamma)\|_1
    \cr &{\le}&     O_\P(\lam C_0^{-1}) 
    \min\big\{\|\bu_0\|_0 \lamuniv, C_0\|\bu_0\|_1\big\}. 
    \ees
    Hence by \eqref{eq:variance-consistency-hbz-bz}
    the second term is also {$O_\P(r_n)$}.
\end{proof}

\thmUnknownSigmaLowerBound*
\begin{proof}[Proof of \Cref{thm:unknown-sigma-lb}]
We have established in the proof of \Cref{thm:unknown-sigma} that for any 
$(\bbeta,\ba_0)\in \scrB_n$ and $0\le \nu\le \|\lasso\|_0$,
$$
\sqrt{nF_\theta}(1-\nu/n)(\htheta_{\nu,{\hbz}}-\theta)
- \sqrt{nF_\theta}(1-\nu/n)(\htheta_{\nu}-\theta) = {o_\P(1+r_n)}
$$
so that
$\sqrt{nF_\theta}(1-\nu/n)(\htheta_{\nu,\bz}-\theta){/(1+r_n)}$ is unbounded 
if and only if $\sqrt{nF_\theta}(1-\nu/n)(\htheta_{\nu}-\theta){/(1+r_n)}$ is unbounded. 
{The proof then follows from the following proposition.}

\propositionUnbounded*
\begin{proof}[Proof of \Cref{prop:unbounded}]
For equality \eqref{second-term-lower-bound}, by simple algebra we find
\bes
    && \sqrt{nF_\theta}(\htheta_{\nu}-\theta)
  \cr&=&
\sqrt{nF_\theta}(\htheta_{\nu=|\Shat|}-\theta)
+ \sqrt{nF_\theta}\bigl((1-\nu/n)^{-1}-(1-|\Shat|/n)^{-1}\bigr) \|\bz_0\|^{-2} \langle \bz_0,\by-\bX{\lasso}\rangle
  \cr&=&
  (1-|\Shat|/n)^{-1}T_n + o_\P(1) + n^{-1}(\nu-|\Shat|) \Lambda_\nu
  \cr&=&
  T_n + o_\P(1) +  n^{-1}(\nu-|\Shat|) \Lambda_\nu
\ees
using $|\Shat|=o_\P(n)$ for the last equality.
We now derive lower bounds on $\Lambda_\nu$ 
given $s_\Omega\le s_0\lll n/\log(p/s_0)$ and $\ba_0$. 
Since $\|\bSigma^{-1/2}\ba_0\|_2=1$, we have $F_\theta = 1/\sigma^2$ and $\bu_0=\bSigma^{-1}\ba_0$. 
Choose $\bbeta$ with $\|\bbeta\|_0=s_0$ 
and $\langle \bu_0,\sgn(\bbeta)\rangle = \|\bu_0\|_1$. 
By the KKT condition for $\lasso$,
\bel{pf-lower-bd-1}
\langle \bz_0,\by -\bX{\lasso}\rangle/(\lam n)
&=& \langle \bu_0, \pa \|\lasso\|_1\rangle
\cr &\ge& \|\bu_0\|_1 - 2\sum_{j=1}^p |(\bu_0)_j|
I\big\{\hbeta^{\text{\tiny (lasso)}}_j\beta_j\le 0\big\}. 
\eel
As $\|\bbeta\|_0\log(p/s_0) \lll n$, there exists constant $C_1$ such that 
\bes
\P\Big\{ \|{\lasso}-\bbeta\|_2^2 \ge C_1\|\bbeta\|_0\log(p/s_0)\big/\big(n\phi_{\min}(\bSigma)\big)\Big\} 
=\eps_n\to 0. 
\ees
Choose the nonzero $\beta_j$ satisfying 
$\min_{j: \beta_j\neq 0} |\beta_j|\pb{^2} \ge C_1\|\bbeta\|_0\log(p\pb{/s_0})\big/\big(n\phi_{\min}(\bSigma)\big)$, 
\bes
\P\Big\{\hbeta^{\text{\tiny (lasso)}}_j\beta_j \le 0 \hbox{ for some }\beta_j>0\Big\} 
\le \P\Big\{\|{\lasso}-\bbeta\|_2 \ge \min_{j: \beta_j\neq 0} |\beta_j| \Big\} \le \pb{\eps_n}. 
\ees
This inequality and \eqref{pf-lower-bd-1} yield that for $0\le \nu < n$, 
\bes
\P\Big\{|\Lambda_\nu|\ge \sqrt{nF_\theta}\|\bz_0\|^{-2}\|\bu_0\|_1n\lam,\quad \|{\lasso}\|_0\ge s_0\Big\} \ge 1 - \eps_n\to 1. 
\ees
This yields \eqref{eq:Lambda_nu-approaching-one-conclusion} 
as $F_\theta=1/\sigma^2$ and $\|\bz_0\|_2^2\sim \chi^2_n$. 
To maximize $\|\bSigma^{-1}\ba_0\|_1$ we take $\bv_0$ satisfying $\bv_0 \in \{-1,0,1\}^p$ 
and $\|\bv_0\|_0=s_\Omega$ 
and set $\ba_0 = \bSigma \bv_0/\|\bSigma^{1/2}\bv_0\|_2$, so that 
$\|\bSigma^{-1/2}\ba_0\|_2=1$ and $\|\bSigma^{-1}\ba_0\|_1=s_\Omega/\|\bSigma^{1/2}\bv_0\|_{2}
\ge s_\Omega^{1/2}/\|\bSigma\|_{op}^{1/2}$. 
\end{proof} 
\end{proof} 

\section{Proofs of Lemmas}
\label{appendix:1}
We give here the proofs of previously used lemmas. The lemmas are restated before their proofs for convenience.

\lipschitzLemma*

\begin{proof}[Proof of \Cref{lemma:lipschitzness}]
    As the objective function at the minimizer
is smaller than the objective function at $\bbeta$, 
$\max(h(\tbbeta),h(\hbbeta)) \le \|\bep\|^2/2n + h(\bbeta)$. 
It follows that $\|\tbbeta\|_2\vee\|\hbbeta\|_2\le C(h,\bep,\bbeta,n)$. 
The strong convexity of the loss $\bb\to L(\bX,\bb)$ (resp. $\bb\to L(\tilde\bX,\bb)$) 
with respect to the metric $\bb\to\|\bX\bb\|_2$
(resp. $\bb\to \|\tilde\bX\bb\|_2$) yield that
\begin{align*}
\|\bX(\hbbeta-\tbbeta)\|_2^2
&\le
\|\bX\tbh-\bep\|_2^2 - \|\bX \bh-\bep\|_2^2+ 2n(h(\tbbeta) - h(\hbbeta)),
\\
\|\tilde\bX(\hbbeta-\tbbeta)\|_2^2
&\le
\|\tilde\bX\bh-\bep\|_2^2 - \|\tilde\bX\tbh-\bep\|_2^2+ 2n(h(\hbbeta) - h(\tbbeta))
\end{align*}
where $\bh=\hbbeta-\bbeta$ and $\tbh = \tbbeta-\bbeta$.
Summing the above two inequalities, we find that 
\bes
&&
\|\bX(\hbbeta-\tbbeta)\|_2^2 + \|\tilde\bX(\hbbeta-\tbbeta)\|_2^2
\\
&\le&
\|\bX\tbh-\bep\|_2^2 - \|\bX\bh-\bep\|_2^2
+
\|\tilde\bX\bh-\bep\|_2^2 - \|\tilde\bX\tbh - \bep\|_2^2
\\
&=&
(\tbh - \bh)^\top\left[\bX^\top\bX - \tilde\bX^\top\tilde\bX\right](\bh + \tbh)
+ 2\bep^\top(\bX-\tilde\bX)(\bh - \tbh)
.
\ees
Since 
$\bX^\top\bX - \tilde\bX^\top \tilde\bX = (\bX-\tilde\bX)^\top\bX + \tilde \bX^\top(\bX - \tilde \bX)$,
the conclusion follows as $\|\bep\|_2,\|\bX\|_{op},\|\tilde\bX\|_{op},\|\bh + \tbh\|_2$ 
are all bounded. 
\end{proof}

\gradientLassoLemma*

\begin{proof}[Proof of \Cref{lm-gradient-of-Lasso}] 
Let $\mu'>0$ the value of the infimum in $\Omega_1$
and $R=\max_{t\ge 0}\|\bX(t)\|_{op}$.
By \Cref{lemma:lipschitzness} with the compact set $\tilde
K=\{\bM\in\R^{n\times p}: \|\bM\|_{op}\le R\}$ we get
$$
    \mu'\|\hbbeta(t)-\hbbeta(t')\|_{2}^2
    \le
    C(R,\bep,\bbeta)
    \|\bX(t)-\bX(t')\|_{op}
    \|\hbbeta(t) - \bbeta(t')\|_2.
    $$
for some constant
$C(R,\bep,\bbeta)<+\infty$ depending only on $(R,\bep,\bbeta)$ only.
Since $
\|\bX(t)-\bX(t')\|_{op}
=\|\ba_0\|_2 \|\bz_0(t)-\bz_0(t')\|_2
$ and $t\to\bz_0(t)$ is a Lipschitz function, we conclude that
the function $t\to\hbbeta(t)$
is Lipschitz continuous with finite (random) Lipschitz norm over $0\le t\le \pi/2$.
Hence on $\Omega_1$, the map
$t\to\hbbeta(t)$ is differentiable Lebesgue almost everywhere in $[0,\pi/2]$

For each $t$ let $\Omega_0(t)$ be the event that the KKT conditions hold strictly,  
\bel{pf-lm-1-1}
\big\langle \bX_j(t), \bep - \bX(t)\bh(t)\big\rangle\big/n \begin{cases} = \lam\sgn(\hbeta_j(t)), 
& \hbeta_j(t)\neq 0 \cr \in (-\lam,\lam), & \hbeta_j(t)=0. \end{cases}
\eel
Let $\Theta$ be uniformly distributed on $[0,\pi/2]$ independently of $(\bep,\bX,\tilde\bX)$
and let $\Omega_0$ be the event that the KKT conditions hold strictly
for the lasso solution $\hbbeta(\Theta)$, i.e., the lasso solution
\eqref{lasso-X(t)} with random $t=\Theta$.
By \Cref{lemma:kkt-strict} we have $\mathbb P(\Omega_0)=1$
since the joint distribution of
$(\bX(\Theta),\bX(\Theta)\bbeta+\bep)$ admits a density with respect to
the Lebesgue measure.
In the event $\Omega_0$, by the Fubini Theorem, (\ref{pf-lm-1-1}) holds in a random set $J\subset [0,\pi/2]$ 
such that $[0,\pi/2]\setminus J$ has zero Lebesgue measure. 
If the KKT conditions hold striclty at $t_0$, it must also hold
strictly on a neighborhood of $t_0$ by continuity of $t\to\langle \bX_j(t),\bep - \bX(t)\bh(t)\rangle$,
hence $J$ is an open set.
Moreover, for each $t_0\in J$, $\sgn(\hbeta_j(t))$ is unchanged in some open interval 
containing $t_0$, so that $\sgn(\hbeta_j(t))$ has zero derivative in $J$. 
Consequently, for any $t\in J$, (\ref{pf-lm-1-1}) yields 
${\dot\bh}_{\Shat^c(t)}(t) = \bf 0$ and
$\big({\dot \bX}^\top(t)\big(\bep - \bX(t)\bh(t)\big) 
- \bX^\top(t){\dot\bX}(t)\bh(t) - \bX^\top(t)\bX(t){\dot\bh}(t)\big)_{\Shat(t)} = {\bf 0}$
for $t\in J$.
As $\bX(t) = \bX\bQ_0 + \bz_0(t)\ba_0^\top$, we have 
${\dot \bX}(t) = {\dot\bz}_0(t)\ba_0^\top$, so that 
\bes
&& \big(\bX^\top(t)\bX(t)\big)_{\Shat(t),\Shat(t)}{\dot\bh}_{\Shat(t)} (t)   
\cr &=&\big(\ba_0\big\langle {\dot\bz}_0(t), \bep - \bX(t)\bh(t)\big\rangle
- \bX^\top(t){\dot\bz}_0(t)\big\langle \ba_0, \bh(t)\big\rangle \big)_{\Shat(t)}
\cr &=&\big((\ba_0)_{\Shat(t)}\big(\bep - \bX(t)\bh(t)\big)^\top
- \big\langle \ba_0, \bh(t)\big\rangle \bX_{\Shat(t)}^\top(t)\big){\dot\bz}_0(t)
\cr &=&\big(\bX^\top(t)\bX(t)\big)_{\Shat(t),\Shat(t)}\bD_{\Shat(t)}^\top(t){\dot\bz}_0(t),\quad t\in J.
\ees
Thus, ${\dot\bh}(t) = \bD^\top(t){\dot\bz}_0(t)$ almost everywhere in $t$ in $\Omega_0\cap\Omega_1$. 
The conclusion follows from the Lipschitz continuity of $t\to \hbbeta(t)$ in the event $\Omega_1$. 
\end{proof}

\section{Bounds on $W, W', W'', W''$}
\label{appendix:2}

\lemmaW*

\begin{proof}[Proof of \Cref{lemma:W}]
    Thanks to the scale equivariance \eqref{scale-equi-variant},
    we take the scale $C_0=\|\bSigma^{-1/2}\ba_0\|_2=1$ without loss of generality, so that \eqref{standard} holds.
    Write $\lasso - \tildelasso = \bhlasso - \tildehlasso$ so that
    by the Cauchy-Schwarz inequality
    \bes
    |W|
    &\le&
    2 \max_{\bh}\sqrt n\left|\left(\frac 1 n - \frac{1}{\|\bz_0\|_2^2}\right) 
    \langle \bz_0, \bX\bQ_0\bh \rangle \right|
    \\ &\le&
    \left|\|\bz_0\|_2 - \sqrt n\right|(\|\bz_0\|_2 + \sqrt n)n^{-1/2}\|\bz_0\|_2^{-1}
    \max_{\bh}\|\bX\bQ_0\bh\|_2
    \ees
    where the maxima are taken over $\bh\in\{\bhlasso,\tildehlasso\}$.
    Thanks to \eqref{properties-on-Omega_2(t)} for $t=0$ and $t=\pi/2$,
    we bound the right hand side on $\Omega_1\cap\Omega_2$ to obtain
    $|W| \le \left|\|\bz_0\|_2 - \sqrt n\right| M'\sigma\lambda_0\sqrt{s_*}$
    where $M'=2 M_5^2 (M_1+ M_5 M_2)$.
    The function $\bz_0\to|\sqrt n - \|\bz_0\|_2|$ is 1-Lipschitz
    with expectation at most 1,
    hence
    \begin{align*}
    \E\left[I_{\Omega_1\cap\Omega_2}
    \exp\left(\frac{u W}{\sigma\lambda_0 \sqrt{s_*} }\right)\right] 
    \le \E \exp\left(|u| M' |\sqrt n - \|\bz_0\|_2 | \right) 
    \le  e^{|u|M' + u^2 (M')^2 / 2}
    \end{align*}
    by the Gaussian concentration theorem \cite[Theorem 5.5]{boucheron2013concentration}.
\end{proof}

\lemmaWprime*

\begin{proof}[Proof of \Cref{lemma:W'}]
    Thanks to the scale equivariance \eqref{scale-equi-variant},
    we take the scale $C_0=\|\bSigma^{-1/2}\ba_0\|_2=1$ without loss of generality, so that \eqref{standard} holds.
Since $\dot\bz_0(t) = \bP_\bep^\perp \dot\bz_0(t)$,  we write
$
\bA(t) = \bX\bQ_0\bD^\top(t) \bP_\bep^\perp
+ \hat\bP(t) \bP_\bep^\perp \big\langle \ba_0, \bh(t)\big\rangle  
$ and notice that $W'=W_1'+W_2'$ where
\begin{equation}
    \begin{split}
        &W_1 ' = \frac {1}{\sqrt n}
\int_0^{\pi/2} \Big\langle 
    \bz_0,
    \bA(t)
    {\dot\bz}_0(t)
\Big\rangle dt 
 , \\
 &W_2' = \frac{1}{\sqrt n}\int_0^{\pi/2}
\big\langle \ba_0, \bh(t)\big\rangle
\left[
-
    \Big\langle
    \bz_0,
    \hat\bP(t) \bP_\bep^\perp   
    {\dot\bz}_0(t)
\Big\rangle 
-
(\sin t)|\Shat(t)| 
\right]dt.
\label{two-integrals-A(t)}
\end{split}
\end{equation}
We now bound from above the two above integrals separately,
starting with $W_1'$.
By the Cauchy-Schwarz's inequality, on $\Omega_2$,
$$
|W_1| 
\le \int_0^{\pi/2} n^{-1/2}\|\bz_0\|_2 \|\bA(t)\dot\bz_0(t)\|_2
\le M_5 \|\bA(t)\dot\bz_0(t)\|_2
.
$$
Next, we bound $\|\bA(t)\dot\bz_0(t)\|_2$ conditionally on $(\bX(t),\bep)$.
As $(\bX\bQ_0)_{\Shat(t)} = \bX_{\Shat(t)}(t) - \bz_0(t)(\ba_0)_{\Shat(t)}^\top$, 
and $\bep^\top \dot\bz_0(t)=0$ by construction of the path $\bz_0(t)$,
the definition of $\bD(t)$ in Lemma \ref{lm-gradient-of-Lasso} gives
\bel{eq:def-A(t)}
\bA(t)
&=& - \Big\{\bw_0(t) - \bz_0(t)\big\|\bw_0(t)\|_2^2\Big\}
\big( \bX(t)\bh(t)\big)^\top \bP_\bep^\perp
\cr && + \bz_0(t)\big(\bw_0(t)\big)^\top\bP_{\bep}^\perp
\big\langle \ba_0,\bh(t)\big\rangle, 
\eel
as in (\ref{matrix-formula}).
Since each of the three terms in the right hand side of \eqref{eq:def-A(t)} is rank 1, 
their Frobenius norm equals their operator norm and on $\Omega_2(t)$ we have
\bes
    I_{\Omega_2(t)}\|\bA(t)\|_F  & \le &
I_{\Omega_2(t)}\left[\|\bw_0(t)\|_2 + \|\bz_0(t)\|_2\big\|\bw_0(t)\|_2^2 \right]\|\bX(t)\bh(t)\|_2 \\
&& + I_{\Omega_2(t)}\|\bz_0(t)\|_2\|\bw_0(t)\|_2 \big| \big\langle \ba_0, \bh(t) \big\rangle \big|
\\
& \le & 
\tilde M \sigma\lambda_0\sqrt{s_*}
\ees
by (\ref{properties-on-Omega_2(t)}),
where $\tilde M = ( M_4^{1/2} M_1   +  M_5  M_4 M_1 + M_5  M_4 M_2 )$.
Conditionally on $(\bX(t),\bep)$,
the vector $\dot\bz_0(t)$ is normal $N({\bf 0},\bP_\bep^\perp)$
and the function $\dot\bz_0(t)\to\|\bA(t)\dot\bz_0(t)\|_2$
is Lipschitz with Lipschitz constant at most $\|\bA(t)\|_{op} \le \|\bA(t)\|_F$,
and the conditional expectation satisfies
$\E[\|\bA(t)\dot\bz_0(t)\|_2 \; | \bX(t),\bep]\le \|\bA(t)\|_F$.
Hence by the Gaussian concentration theorem 
(e.g. \cite[Theorem 5.5]{boucheron2013concentration}),
for any $u\in\R$,
\begin{align*}
    I_{\Omega_2(t)}\E\left[\exp\left(\frac{u \|\bA(t)\dot\bz_0(t)\|_2}{\sigma\lambda_0\sqrt{s_*}}\right)\; \Big| \bX(t), \bep\right] 
&\le
I_{\Omega_2(t)}
\exp\left(
    \frac{|u| \|\bA(t)\|_F}{\sigma\lambda_0\sqrt{s_*}}
    + \frac{u^2 \|\bA(t)\|_F^2}{2\sigma^2\lambda_0^2{s_*}}
\right)
\\
&\le
\exp\left(
    |u| \tilde M
    + u^2 \tilde M^2/2
\right).
\end{align*}
By Jensen's inequality
with respect to the Lebesgue measure over $[0,\pi/2]$, the Fubini theorem 
and the fact that $I_{\Omega_1\cap\Omega_2} \le I_{\Omega_2(t)}$, we have
\begin{align*}
   & \E \left[I_{\Omega_1\cap\Omega_2} \exp \left(
            \frac{u W_1'}{
                \sigma\lambda_0 \sqrt{s_*} }
\right)
\right] \\
&\le
\frac{2}{\pi}\int_0^{\pi/2}
\E\left\{I_{\Omega_2(t)}\E\left[\exp\left(\frac{u \|\bA(t)\dot\bz_0(t)\|_2}{\sigma\lambda_0\sqrt{s_*}}\right)\;
 \bigg|\bX(t),\bep \right]\right\}
dt,
\\  &\le 
\exp\left(
    |u| (\pi/2) \tilde M  M_5
    + u^2 (\pi/2)^2 \tilde M^2 M_5^2 /2
\right).
\end{align*}

We now bound the second integral in \eqref{two-integrals-A(t)}.
We decompose $\bz_0$ as 
\begin{equation*}
    \bz_0 = \bP_\bep \bz_0(t) + \bP_\bep^\perp [(\cos t) \bz_0(t) - (\sin t){\dot\bz}_0(t)]
\end{equation*}
where for the first term we use that
$\bP_\bep\bz_0(t)=\bP_\bep\bz_0$ is the same for every $t$.
For any $t\ge0$, the integrand of $W_2'$ in \eqref{two-integrals-A(t)}
can be written as a polynomial of degree 2 in $\dot\bz_0(t)$ as follows
$$
W_2' = \frac{2}{\pi}\int_0^{\pi/2}
\Big(
\langle \dot\bz_0(t),\bQ(t) \dot\bz_0(t)\rangle - \trace\,\bQ(t)
+ \langle \bv(t), \dot\bz_0(t)\rangle 
+ \mu(t)
\Big)
dt
$$
where
\begin{align*}
    \bQ(t) & = n^{-1/2} (\pi/2) \langle \ba_0,\bh(t)\rangle (\sin t) \bP_\bep^\perp \hat\bP(t)\bP_\bep^\perp
    ,\\
    \bv(t) & = n^{-1/2} (\pi/2)\langle \ba_0,\bh(t)\rangle
    \bP_\bep^\perp \hat\bP(t)[ - \bP_\bep\bz_0(t) - (\cos t)\bP_\bep^\perp \bz_0(t)]
    , \\
    \mu(t) &= n^{-1/2} (\pi/2) \langle \ba_0,\bh(t)\rangle (\sin(t)) [\trace(\bP_\bep^\perp\hat\bP(t)\bP_\bep^\perp) - |\Shat(t)|].
\end{align*}
Conditionally on $(\bX(t),\bep)$, the coefficients $\bQ(t),\bv(t)$ and $\mu(t)$ are fixed
and $\dot\bz_0(t)$ is normal $N({\bf 0},\bP_\bep^\perp)$.
Furthermore, the value of the integrand is unchanged if $\dot\bz_0(t)$
is replaced by $\ba(t) = \dot\bz_0(t) + Z\frac{\bep}{\|\bep\|_2}$ which
has $N({\bf 0},\bI_n)$ distribution if $Z\sim N(0,1)$ is independent of $(\bX(t),\bp,\dot\bz_0(t))$.

By Jensen's inequality over the Lebesgue measure on $[0,\pi/2]$, the Fubini Theorem
and conditioning on $(\bep,\bX(t))$,
the expectation $\E[I_{\Omega_1\cap\Omega_2} e^{u W_2'}]$
is bounded from above by
$$
\frac 2 \pi \int_0^{\pi/2}
\E\left\{ I_{\Omega_2(t)} \E\left[
        e^{u
        \big(
\langle \dot\bz_0(t),\bQ(t) \dot\bz_0(t)\rangle - \trace\,\bQ(t)
+ \langle \bv(t), \dot\bz_0(t)\rangle 
+ \mu(t)
    \big)}
    \Big| \bX(t),\bep\right]
\right\}
dt
.
$$
If $\bQ,\bv,\mu$ are deterministic with the same dimension as above
and $\ba\sim N({\bf 0},\bI_n)$ is standard normal then
for all $u\in \R$ with $|u|<1/(2\|\bQ\|_{op})$,
$$
\E e^{u(\ba^\top \bQ \ba - \trace(\bQ) + \ba^\top\bv + \mu)}
\le
\exp\left(u\mu + \frac{u^2(\|\bQ\|_F^2 + \|\bv\|^2/2)}{1-2\|\bQ\|_{op} |u|}\right).
$$
This upper bound is proved by diagonalizing $\bQ$ and using
the rotational invariance of the normal distribution, cf., for instance
the proofs in
\cite[Lemma 2.4]{hsu2012tail} or
\cite[Proposition 8.1]{bellec2014affine}.
For $\bv=0$ see also \cite[Lemma 1]{laurent2000adaptive}.
By applying this bound conditionally on $(\bep,\bX(t))$, we get
$$
\E\left[I_{\Omega_1\cap\Omega_2} e^{u W_2'}\right]
\le
\frac 2 \pi \int_0^{\pi/2}
\E\left\{
    I_{\Omega_2(t)}
\exp\left(
u\mu(t)
+\frac{u^2(\|\bQ(t)\|_F^2 + \|\bv(t)\|_2^2/2)))}{1-2\|\bQ(t)\|_{op} |u|}
\right)
\right\}
dt
$$
for any $u\in\R$ such that $|u|\le 1/(2\sup_{t\ge 0}\|\bQ(t)\|_{op})$.
The quantity $\mu(t)$ as well as the norms of $\bQ(t)$ and $\bv(t)$
can be readily bounded on $\Omega_2(t)$ thanks to \eqref{properties-on-Omega_2(t)}.
For $\mu(t)$,
since $|\Shat(t)| = \trace(\hat\bP(t))$ and $\bP_\bep$ is a rank-1 orthogonal projection 
$$
|\Shat(t)| - \trace(\bP_\bep^\perp \hat\bP(t)\bP_\bep^\perp) 
= \trace(\bP_\bep\hat\bP(t) \bP_\bep) \in [0,1],
$$
and hence
$|\mu(t)|\le 3n^{-1/2}(\pi/2) M_2 \lambda_0 \sqrt{s_*}$.
For $\bQ(t)$, by properties of the operator norm
and the fact that the operator norm of projectors is a most 1,
\begin{align*}
\|\bQ(t)\|_{op}
&\le n^{-1/2} (\pi/2) M_2 \sigma\lambda_0\sqrt{s_*}, \\
\|\bQ(t)\|_F
&\le 
n^{-1/2}(\pi/2)M_2 \sigma\lambda_0\sqrt{s_*} |\Shat(t)|^{1/2}
\le
(\pi/2) M_2 \sigma\lambda_0\sqrt{s_*}
\end{align*}
where we used that $\|\hat\bP(t)\|_F^2 = |\Shat(t)|\le n$.
Finally, for $\bv(t)$,
$$
\|\bv(t)\|_2 \le
n^{-1/2} (\pi/2) M_2 \sigma\lambda_0 \sqrt{s_*}\|\bz_0(t)\|_2
\le
(\pi/2) M_2 \sigma\lambda_0\sqrt{s_*}  M_5.
$$
We have established that
$$
\E\left[
I_{\Omega_1\cap\Omega_2} \exp\left(\frac{u W_2'}{\sigma\lambda_0\sqrt{s_*} }\right)
\right]
\le
\exp\left(
|u| \frac{3\pi M_2}{2\sqrt n}
+ \frac{u^2 (\pi^2M_2^2/4 + \pi^2M_2^2 M_5^2 / 8 )}{1- |u|(\pi/2) M_2 / \sqrt n}
\right)
$$
for any $u$ such that $|u| \le 2\sqrt{n} / (\pi M_2)$.
To complete the proof, we combine the bound on $W_1'$ and the bound on $W_2'$ using that for all $v\in\R$,
by Jensen's inequalty,
$$\E[I_{\Omega_1\cap\Omega_2}
    e^{v(W_1' + W_2')}
]
\le
(1/2)
\E[I_{\Omega_1\cap\Omega_2} e^{2v W_1'}] 
+ (1/2)
\E[I_{\Omega_1\cap\Omega_2} e^{2v W_2'}] 
.
$$
\end{proof}

\lemmaWprimePrime*

\begin{proof}[Proof of \Cref{lemma:W''}]
    Thanks to the scale equivariance \eqref{scale-equi-variant},
    we take the scale $C_0=\|\bSigma^{-1/2}\ba_0\|_2=1$ without loss of generality, so that \eqref{standard} holds.
    By simple algebra and the condition $\Shat(t)\le s_*$ in \eqref{conds-limited} on $\Omega_2$,
    $$W'' \le \frac{s_*}{\sqrt n} \int_0^{\pi/2}|\langle \ba_0, \hbbeta(t) - \hbbeta(0)\rangle | dt
    \le
    \frac{s_*}{\sqrt n} \int_0^{\pi/2} \int_0^t|\langle \ba_0, \bD^\top(x) \dot\bz_0(x)\rangle | dx dt
    .$$
    The integrand is non-negative so
    the function $t \to \int_0^t|\langle \ba_0, \bD^\top(x) \dot\bz_0(x)\rangle | dx$ defined on $[0,\pi/2]$
    is maximized at $t=\pi/2$.
    By Jensen's inequality, Fubini's theorem and the law of total expectation,
$\E[I_{\Omega_1\cap\Omega_2}
    e^{u W''}
    ]$ is bounded from above by
    $$\frac{2}{\pi}
    \int_0^{\pi/2}
    \E
    \left\{
    I_{\Omega_2(x)}
        \E 
        \left[
        \exp\left(\frac{u (\pi/2)^2 s_*}{\sqrt n} 
        \Big|\left\langle \ba_0, \bD^\top(x)\dot\bz_0(x) \right\rangle\Big|\right)
        \Big| \bep, \bX(x)
        \right]
    \right\}
    dx
    .
    $$
Conditionally on $(\bep,\bX(x))$, the random variable $\langle \ba_0, \bD^\top(x)\dot\bz_0(x) \rangle$
is normal with variance $\|\bP_\bep^\perp \bD(x) \ba_0\|_2^2$.
It follows from (\ref{conds-limited}) and the definition of $\bD(x)$ in Lemma \ref{lm-gradient-of-Lasso} that in the event $\Omega_2(x)$,
\bes
    \big\|\bP_\bep^\perp \bD(x)\ba_0\big\|_2
    &\le& (M_4/n)\|\bX(x)\bh(x)\|_2 + (M_4/n)^{1/2}\|\bSigma^{1/2}\bh(x)\|_2
    \\ &\le& M''\sigma \lambda_0\sqrt{s_*}/\sqrt{n} 
\ees
for some constant $M''$ that depends on $M_1, M_2, M_4, M_5$ only.
If $Z$ is centered normal, then $\E[e^{|vZ|}]\le \E[e^{vZ}] + \E[e^{-vZ}]=2 \exp(v^2\E[Z^2]/2)$. Combining the two previous displays, we have proved that
$$
    \E[I_{\Omega_1\cap\Omega_2}
    e^{u W''}
    ]
    \le 2
    \exp \left[
        {u^2 (\pi/2)^4(M'')^2 \sigma^2 \lambda_0^2 s_*^3/(2n^2)}
    \right]
    .
$$
Using that $s_*^2 \le n^2$ completes the proof with the scale change 
$u\to u/(\sigma\lam_0\sqrt{s_*})$.
\end{proof}

\lemmaWprimePrimePrime*

\begin{proof}[Proof of \Cref{lemma:W'''}]
    Thanks to the scale equivariance \eqref{scale-equi-variant},
    we take the scale $C_0=\|\bSigma^{-1/2}\ba_0\|_2=1$ without loss of generality, so that \eqref{standard} holds.
    Note that $Z = \langle \bep, \bz_0 \rangle /\|\bep\|_2$
    is standard normal since $\bz_0=\bX \bu_0$ is independent of $\bep$
    and $\bep/\|\bep\|_2$ is uniformly distributed on the sphere.
    On $\Omega_2$ we have 
    $|\langle \bP_\bep\bz_0, \bX\bQ_0\tildehlasso\rangle|
    \le
    \|\bX\bQ_0\tildehlasso\|_2 |Z|$.
    Hence on $\Omega_2$,
    \begin{align*}
        & \sqrt n \|\bz_0\|_2^{-2} |\langle \bz_0, \bX\bQ_0\tildehlasso\rangle | \\
        &\le M_5^2 n^{-1/2}
    \left(
    |\langle \bP_\bep\bz_0, \bX\bQ_0\tildehlasso\rangle\rangle|
    +
    |\langle \bP_\bep^\perp\bz_0, \bX\bQ_0\tildehlasso\rangle\rangle|
    \right),\\
    &\le M_5^2 n^{-1/2}
    \left(
    |Z| \|\bX\bQ_0\tildehlasso\|_2
    +
    |\langle \bP_\bep^\perp\bz_0, \bX\bQ_0\tildehlasso\rangle\rangle|
    \right).
    \end{align*}
    On $\Omega_2(\pi/2)$, quantity
    $\|\bX\bQ_0\tildehlasso\|_2$ is bounded from above thanks to \eqref{properties-on-Omega_2(t)} for $t=\pi/2$.
    For any $u\in\R$, by Jensen's inequality and $I_{\Omega_1\cap\Omega_2}\le I_{\Omega_2(\pi/2)}$,
    \bes
         2\E\left[I_{\Omega_1\cap\Omega_2} 
        \exp\left(\frac{u W'''}{\sigma\lambda_0\sqrt{s_*}}\right) \right]
        &\le&
            \E\Big[\exp\Big\{ 2|u| M_5^2 (M_1+M_2 M_5 )|Z|\Big\}\Big] 
          \cr  && \quad + 
            \E\left[
                I_{\Omega_2(\pi/2)}\exp\left\{2 |u|  M_5^2  \frac{|\langle \bP_\bep^\perp \bz_0, \bX\bQ_0\tildehlasso\rangle|}{\sqrt n \sigma \lambda_0 \sqrt{s_*}}\right\}
            \right]
            .
        \ees
    Since $|Z|$ is the absolute value of a standard normal,
    we use $\E[e^{|vZ|}]\le 2 e^{v^2/2}$ for the first line of the right hand side
    with $v=2u M_5^2 (M_1+M_2 M_5 )$.
    For the second line, since $\bP_\bep^\perp\bz_0$ and $\bX\bQ_0\tildehlasso$
    are independent and the conditional distribution of $\bP_\bep^\perp\bz_0$
    given $(\bep,\bX(\pi/2))$ is $N({\bf 0},(1/C_0^2)\bP_\bep^\perp)$,
    for any $v\in\R$ we have
    $$\E\left[\exp(|v\langle \bP_\bep^\perp \bz_0, \bX\bQ_0\tildehlasso\rangle|)
    \Big| \bep,\bX(\pi/2) \right]
    \le
    2 \exp(v^2 \|\bP_\bep^\perp \bX\bQ_0\tildehlasso\|_2^2/2).
    $$
    On $\Omega_2(\pi/2)$, the squared norm in the right hand side
    is bounded from above thanks to \eqref{properties-on-Omega_2(t)} for $t=\pi/2$.
    Combining the above bounds completes the proof.
\end{proof}

\lemmaLastTermDivergence*

\begin{proof}[Proof of \Cref{lemma:last-term-divergence}]
    Using $(a+b+c)^3 \le 3a^2 + 3b^2+3c^3$,
    $$\left(\int_0^{\pi/2}\sigma^2(|\Shat(t)| - |\Shat(0)|) (\sin t)dt\right)^2
    \le
    \begin{cases}
        &3 \left(\int_0^{\pi/2}(\sin t)(\bep^\top\bX(0)(\bh(0) - \bh(t))dt\right)^2
        \\
        + &3 \left(\int_0^{\pi/2}(\sin t)(\bep^\top\bX(t)\bh(t) - \sigma^2|\Shat(t)|) dt\right)^2 \\
        + &3 \left(\bep^\top\bX(0)\bh(0) - \sigma^2|\Shat(0)|\right)^2.
    \end{cases}
    $$
    where we used that $\bep^\top\bX(t)=\bep^\top\bX(0)$ by construction of the path $\bz_0(t)$ in \eqref{bz(t)}.
    Next, write 
    $\bep^\top\bX(0)(\bh(t) - \bh(0))
    = \int_0^t \bep^\top \bX(0)\bD(x)^\top\dot\bz_0(x) dx
    = \int_0^t \bep^\top \bX(x)\bD(x)^\top\dot\bz_0(x) dx.
    $
    The function $f(x)=|\bep^\top \bX(x)\bD(x)^\top\dot\bz_0(x)|$ is non-negative
    thus
    $\int_0^{\pi/2}(\sin t)(\int_0^t f(x)dx)dt\le \int_0^{\pi/2}f(x)dx$.
    By Jensen's inequality applied to each of the three terms above,
    the previous display is bounded from above on $\Omega_1\cap\Omega_2$ by
    \begin{equation}
        \begin{split}
          &3 I_{\Omega_1\cap\Omega_2} \int_0^{\pi/2} \frac{2}{\pi} \left[\left(\frac{\pi}{2}\bep^\top\bX(x)\bD^\top\dot\bz_0(x)\right)^2\right]dx
        \\
        + &3 I_{\Omega_1\cap\Omega_2} \Big( \int_0^{\pi/2} W(t) (\sin t) dt + W(0) \Big).
        \label{three-lines}
        \end{split}
    \end{equation}
    where $W(t) =(\bep^\top\bX(t)\bh(t) - \sigma^2|\Shat(t)| )^2$ for all $t\ge 0$.
    To bound the expectation of the first line, 
    we use the Fubini Theorem and the fact that for any $x\in[0,\pi/2]$,
    $I_{\Omega_1\cap\Omega_2}\le I_{\Omega_2(x)}$ and
    \begin{align*}
    \E\left[I_{\Omega_1\cap\Omega_2}\left(\frac{\pi}{2}\bep^\top\bX(x)\bD^\top\dot\bz_0(x)\right)^2\right]
    &\le
    \E\left[I_{\Omega_2(x)} \E\left[\left(\frac{\pi}{2}\bep^\top\bX(x)\bD^\top\dot\bz_0(x)\right)^2 \Big| \bX(x),\bep\right]  \right], \\
    & = 
    \E\left[
        I_{\Omega_2(x)} \left(\frac \pi 2\right)^2 
        \|\bP_\bep^\perp \bD(x) \bX(x)^\top \bep\|_2^2
    \right].
    \end{align*}
    By \Cref{lm-gradient-of-Lasso} that computes $\bD(x)$ 
    and the inequalities in \eqref{properties-on-Omega_2(t)}, on $\Omega_2(x)$ we have
    \begin{align*}
        &\quad \|\bP_\bep^\perp \bD(x) \bX(x)^\top \bep\|_2\\
        &\le \|\hat\bP(x) \bep\|_2 |\langle \ba_0,\bh(x)\rangle |
        + \|\bP_\bep^\perp \bX(x)\bh(x)\|_2 \|\bep\|_2  \|\bw_0(x)\|_2 \\
        &\le \sigma^2 \lambda_0\sqrt{s_*}( M_5 M_1 + M_2  M_5 M_4^{1/2})\sqrt n.
    \end{align*} 

    We now bound the expectation of the second line in \eqref{three-lines}.
    For any $t$,
    let $\Omega_3(t) = \{ \|\bX(t)\bh^{(noiseless)}(t)\|_2 \le \sqrt n M_1 \sigma \lambda_0 \sqrt{s_*} \}$
    where $\bh^{(noiseless)}(t)$ is the error vector of the noiseless Lasso for $\bX(t)$ defined in \eqref{noiseless}
    and notice that $\Omega_1\cap\Omega_2\subset\Omega_2(t)\subset\Omega_3(t)$.
    Consider a random variable
    $\Theta$ independent of all other random variables, valued in $[0,\pi/2]$ 
    with distribution $\P(\Theta = 0) = 1/2$ 
    and $\P(\Theta \in (a,b)) = \int_a^b(\sin t)dt/2$ for any $0<a<b\le \pi/2$.
    In other words, $\Theta$ is the mixture of a dirac at 0 and a continuous distribution with density $t\to \sin t$ on $[0,\pi/2]$.
    Since $\Omega_1\cap\Omega_2\subset\Omega_3(t)$ for all $t\ge0$, 
    the expectation of the second second \eqref{three-lines} is bounded from above by
    $$
    3\E\Big[\int_0^{\pi/2} I_{\Omega_3(t)} W(t) (\sin t) dt + I_{\Omega_3(0)} W(0)\Big]
     =
     6\E\left\{I_{\Omega_3(\Theta)} \E\left[ W(\Theta) \big|\bX(\Theta) \right] \right\}
    .
    $$
    where we used the law of total expectation and the fact that $I_{\Omega_3(\Theta)}$ is a measurable function of $\bX(\Theta)$.
    Recall that $W(\cdot)$ is defined after \eqref{three-lines}.
    The random design matrix $\bX(\Theta)$ has iid $N({\bf 0},\bSigma)$ rows
    and admits a density with respect to the Lebesgue measure on $\R^{n\times p}$.
    Furthermore, $\bX(\Theta)$ is independent of $\bep$ so by 
    \Cref{lemma:divergence} below the previous display is 
    bounded from above by
    \begin{equation*}
        6 n\sigma^4 + 6\sigma^2 \E[I_{\Omega_3(\Theta)}\|\bX(\Theta)\bh(\Theta)\|_2^2].
    \end{equation*}
    The function $\bep\to \|\bX(\Theta)\bh(\Theta)\|_2$ is 1-Lipschitz (see, e.g., \cite{bellec2016bounds})
    hence on $\Omega_3(\Theta)$ we have
    $$\|\bX(\Theta)\bh(\Theta)\|_2 \le \|\bep\|_2 + \|\bX(\Theta)\bh^{noiseless}(\Theta)\|_2
    \le \|\bep\|_2 + M_1 \sqrt n \sigma \lambda_0\sqrt{s_*}.
    $$
    Using $(a+b)^2\le 2a^2 + 2b^2$, this shows that 
    $\E[I_{\Omega_3(\Theta)}\|\bX(\Theta)\bh(\Theta)\|_2^2] \le 2n\sigma^2 + 2 M_1^2 \sigma^2 n \lambda_0^2 s_*$
    and the proof is complete.
\end{proof}

\begin{lemma}[Section 4 of \cite{bellec_zhang2018second_order_stein}]
    \label{lemma:divergence}
    Let $\bar\bX$ be a random design matrix 
    that admits a density with respect to the Lebesgue measure on $\R^{n\times p}$.
    Consider the lasso estimator $\check\bbeta$ with design $\bar\bX$ and response vector
    $\bar\by = \bar\bX\bbeta + \bep$ where $\bep\sim N({\bf 0},\sigma^2\bI_n)$ is independent of $\bar\bX$.
    Let $\check S=\supp(\check\bbeta)$. 
    Then with probability one with respect to the probability distribution of $\bar\bX$ we have
    \begin{align*}
    \E\left[
        (\bep^\top\bar\bX(\check\bbeta - \bbeta) - \sigma^2|\check S|)^2
    \Big| \bar\bX
    \right]
    \le
    \sigma^2
    \E
    \left[
        \|\bar\bX({\check\bbeta}-\bbeta)\|_2^2 
        \Big|\bar\bX
    \right]
    +\sigma^4
    n.
    \end{align*}
\end{lemma}

\section{Proofs of bounds on sparse eigenvalues}
\label{appendix:3}

\lemmaFullRank*

\begin{proof}[Proof of \Cref{lemma:full-rank-2m}]
    (i). Let $\sigma_{\min}(\cdot)$ denote the smallest singular value of matrix. 
    Thanks to the scale equivariance \eqref{scale-equi-variant},
    we take the scale $C_0=\|\bSigma^{-1/2}\ba_0\|_2=1$ 
    without loss of generality, so that \eqref{standard} holds. 
    By construction of the path \eqref{bz(t)}, 
    $\bz_0(t) = \bP_{\bep}\bz_0+\bP_{\bep}^\perp\{(\cos t)\bz_0+(\sin t)\bg\}$ 
    where $(\bz_0,\bg)\in \R^{n\times 2}$ is standard Gaussian and independent of $\bep$. 
    Let $\bar t = \argmin_t \|\bP_{\bep}^\perp\bz_0(t)\|_2$. As $\bP_{\bep}(\bz_0,\bg)$ is 
    independent of $\{\bP_{\bep}^\perp\bz_0(t), t>0\}$, $\min_t\|\bz_0(t)\|_2^2$ has the same 
    distribution as 
    $\|\bP_{\bep}^\perp\bz_0(\bar t)\|_2^2+\|\bP_{\bep}((\cos \bar t)\bz_0+(\sin \bar t)\bg)\|_2^2$, 
    which is no smaller than $\min_{t>0} \|(\cos t)\bz_0+(\sin t)\bg\|_2^2 =\sigma_{\min}^2(\bz_0,\bg)$. 
    Thus, 
    \bes
    \P\Big(\inf_{t\ge 0}\|\bz_0(t)\|_2 = 0\Big) = \P\Big(\sigma_{\min}(\bz_0,\bg)=0\Big) = 0. 
    \ees 
    (ii). We will prove that for any $\bep\ne0$, conditionally on $\bep$,
    for any set $A$ as above, 
    $$\P\big(\forall \bu: \supp(\bu)=A, \quad\bSigma^{1/2}\bu\ne 0 \Rightarrow \bX(t)\bu \ne 0 \big| \bep\big) =1.$$
    This implies that $\bX_A(t)$ is of full rank $|A|$ with probability one.

    Let $\tba_0=\bSigma^{-1/2}\ba_0$. As $\|\tba_0\|_2=C_0=1$, there exists 
a matrix $\tbQ_0\in\R^{(p-1)\times p}$ with $(p-1)$ orthonormal rows such that
$\tbQ{}_0^\top \tbQ_0 = \bI_{p}-\tba_0\tba_0^\top = \bSigma^{1/2}\bQ_0\bSigma^{-1/2}$
    and $\tbQ_0\tbQ{}_0^\top = \bI_{p-1}$. Similarly for 
    any $\bep\ne 0$, $\bP_\bep^\perp$ is an orthogonal projection onto a subspace of dimension $n-1$
    and there exists a matrix $\bM_\bep\in\R^{(n-1)\times n}$
    with $(n-1)$ orthonormal rows such that
    $\bM_\bep^\top \bM_\bep = \bP_\bep^\perp$
    and $\bM_\bep\bM_\bep^\top = \bI_{n-1}$. 
    Conditionally on $\bep$, 
    $\bP_\bep^\perp\bX\bQ_0\bSigma^{-1/2}=\bP_\bep^\perp\bX\bSigma^{-1/2}\tbQ{}_0^\top \tbQ_0$, 
    $\bP_\bep^\perp\bz_0 = \bP_\bep^\perp\bX\bSigma^{-1/2}\tba_0$ and 
    $\bP_\bep^\perp\tbz_0 = \bP_\bep^\perp\bg$ are mutually independent,
    and we may write 
    $\bW = \bM_\bep(\bX\bSigma^{-1/2}\tbQ{}_0^\top,\bz_0,\bg)$ as a 
    standard Gaussian matrix in $\R^{(n-1)\times(p+1)}$. It follows that 
     \bes 
     \bP_{\bep}^\perp\bX(t)\bu 
&=& \bP_{\bep}^\perp\big\{\bX\bQ_0 + (\cos t) \bz_0\ba_0^\top 
+ (\sin t) \bg\ba_0^\top\big\}\bu 
\cr &=& \bM_\bep^\top \bM_\bep\big(\bX\bSigma^{-1/2}\tbQ{}_0^\top, \bz_0,\bg\big)
\begin{pmatrix}
    \tbQ_0\bSigma^{1/2}\bu
    \cr (\cos t)\langle \ba_0,\bu\rangle
    \cr (\sin t) \langle \ba_0,\bu\rangle
\end{pmatrix}
\\ \nonumber &=& \bM_\bep^\top\bW\bv(t)
\ees 
    with $\bv(t)= \big(\big(\tbQ_0\bSigma^{1/2}\bu\big){}^\top, (\cos t)\langle \ba_0,\bu\rangle, 
    (\sin t) \langle \ba_0,\bu\rangle\big){}^\top$.
    We note that 
    \bes 
    \quad & 
    \|\bv(t)\|_2^2 = \E\|\bW\bv(t)\|_2^2/(n-1) = \E\| \bP_{\bep}^\perp\bX(t)\bu\|_2^2/(n-1)
    = \|\bSigma^{1/2}\bu\|_2^2. 
    \ees 
    When $\supp(\bu)\subseteq A$ with $|A\setminus S| \le 2(m+k)$, $\{\bv(t), t>0\}$ lives in a subspace of 
    dimension $2(m+k)+|S|+1$ in $\R^{p+1}$. Since $\bW$ is standard Gaussian, 
    $\bW$ is full rank in this subspace almost surely when \eqref{condition-for-full-rank-2m} holds. 
    In this event, $\bSigma^{1/2}\bu\neq 0$ implies 
    $\inf_{t}\|\bX(t)\bu\|_2\ge\inf_{t}\|\bW\bv(t)\|_2 > 0$
    and the second claim is proved.
\end{proof}

\lemmaChiSquare*

\begin{proof}[Proof of \Cref{lemma:chi2-(t)}]
    Again we take the scale $C_0=\|\bSigma^{-1/2}\ba_0\|_2=1$ without loss of generality. 
    Similar to the proof of \Cref{lemma:full-rank-2m} (i), 
    a standard bound on the singular value of the standard Gaussian matrix 
    $(\bz_0,\bg)\in \R^{n\times 2}$ 
    (cf. \cite[Theorem II.13]{DavidsonS01}) yields 
    \bes
    \P\Big(\inf_{t\ge 0}\|\bz_0(t)\|_2 \le \sqrt{n} - \sqrt{2} - t\Big) 
    \le \P\Big(\sigma_{\min}(\bz_0,\bg)\le \sqrt{n} - \sqrt{2} - t\Big) \le e^{-t^2/2}. 
    \ees
This inequality and its counterpart for $\sup_{t\ge 0}\|\bz_0(t)\|_2$ completes the proof. 
\end{proof}

\lemmaIso*

\begin{proof}[Proof of \Cref{lm-5-new-iso}]
Let $A\subset[p]$ such that $|A\setminus S|\le m+k$
and let $\scrU_A=\{\bu\in\R^p: \bu_{A^c}={\bf 0}, \|\bSigma^{1/2}\bu\|_2=1 \}$.
Define $\bW\in\R^{(n-1)\times (p+1)}$
as in the proof
of \Cref{lemma:full-rank-2m}.
Similarly, for a given $\bu$ define $\bv(t)$ as in the proof of \Cref{lemma:full-rank-2m}.
Then $\bW$ has iid $N(0,1)$ entries and
$\|\bP_\bep^\perp \bX(t)\bu\|_2 = \|\bW \bv(t)\|_2$.
If $\bu\in \scrU_A$ then $\bv(t)$ has unit norm and
lives in a linear subspace of dimension
$|A|+1=s_*+1\le n-1$. By \cite[Theorem II.13]{DavidsonS01},
$$
    \big|\|\bP_\bep^\perp\bX(t)\bu\|_2-(n-1)^{1/2}\big|\le (s_*+1)^{1/2}+\eps_2n^{1/2} /2
\qquad
\forall t,\forall \bu\in\scrU_A
$$
with probability at least $1- 2e^{-n \eps_2^2/8}$.
The elementary inequality $|n^{1/2}-(n-1)^{1/2}|+(s_*+1)^{1/2} \le (s_*+2)^{1/2}$ holds
for $s_*\in[1,n-2]$, which is granted by \eqref{conditions-epsilons-2}.
Hence on the event of the previous display,
\begin{equation}
    \label{intersection-event-1-iso}
\big|\|\bP_\bep^\perp\bX(t)\bu\|_2- n^{1/2}\big|\le (s_*+2)^{1/2}+\eps_2n^{1/2} /2
\le (\eps_1+\eps_2)n^{1/2}/2
\end{equation}
for all $t\ge0$ and $\bu\in\scrU_A$
thanks to \eqref{conditions-epsilons-2}.
Since $\bep^\top\bX(t)=\bep^\top\bX$ for all $t\ge0$ and $\|\bSigma^{1/2}\bu\|_2=1$
for $\bu\in \scrU_A$,
the supremum $\sup_{t\ge 0, \bu\in\scrU_A}\|\bP_\bep \bX(t)\bu\|_2$ 
is a 1-Lipschitz function of
the random variable $\|\bep\|_2^{-1}\bep^\top\bX\bSigma^{-1/2}$
which has standard normal $N(0,\bI_p)$ distribution.
By the Gaussian concentration theorem
(e.g. \cite[Theorem 5.5]{boucheron2013concentration}),
\begin{equation}
    \label{intersection-event-2-iso}
    0 \le \|\bP_\bep \bX(t)\bu\|_2 \le (s_*+1)^{1/2} + \eps_2 n^{1/2} /2
    \qquad \forall t\ge0, \forall \bu\in\scrU_A
\end{equation}
has probability at least $1- e^{-n\eps_2t^2/8}$.
On this event, $\|\bP_\bep\bX(t)\bu\|_2\le(\eps_1+\eps_2)n^{1/2}/2$.
Consequently, $|\|\bX(t)\bu\|_2 - n^{1/2}|\le (\eps_1+\eps_2)n^{1/2}$ holds
simultaneously for all $t \ge0$ and $\bu\in\scrU_A$ on the intersection
of of \eqref{intersection-event-1-iso} and \eqref{intersection-event-2-iso}.
By the union bound, this intersection
has probability at least $1-3e^{-n \eps_2^2/8}$.

Since there are $\binom{p-|S|}{k+m}$ possible
sets $A\subset [p]$ with $|A\setminus S|\le k+m$, the event \eqref{def-Omega_iso} holds with probability at least
$1-\binom{p-|S|}{k+m} 3e^{-n\eps_2^2/8}\ge 1- 3e^{-\eps_4n}$
since $\log \binom{p-|S|}{k+m} \le \eps_3 n$
and $\eps_4+\eps_3=\eps_2^2/8$ are provided
by \eqref{conditions-epsilons-1} and \eqref{conditions-epsilons-2}.
\end{proof}

\section{Proofs for bounds on false positives}
\label{appendix:4}

\propositionFPbetabar*

\begin{proof}[Proof of \Cref{proposition:eta_1-eta_2-bar-beta}]
    The Lasso estimator must satisfy the KKT condition 
    $\bg\in \lam\pa \|{\lasso}\|_1$
    where $\bg = \bar\bX^\top(\by - \bar\bX{\lasso})/n$ is the negative gradient of the loss 
    $\|\by-\bar\bX\bb\|_2^2/(2n)$. 
    For $j\not\in \bar S$, the KKT conditions implies 
    \bes
    \bar\bx_j^\top(\by - \bar\bX\bbetabar)/(n\lam) - \bar\bx_j^\top\bar\bX\bu/n 
    = \bar\bx_j^\top(\by - \bar\bX{\lasso})/(n\lam) = \pa |u_j|, 
    \ees
    where $\bu=({\lasso}-\bbetabar)/\lambda$, so that for $\lam\ge\mu_0$ 
    \bes
    \Big| u_j (\bSigmabar\bu)_j + |u_j|\Big|= \Big| u_j \bar\bx_j^\top(\by - \bar\bX\bbetabar)/(n\lam)\Big| 
    \le \eta_2 |u_j| 
    \ees
    due to $u_j = {\lasso}_j/\lam$ for $j\not\in \bar S$. Moreover, using \eqref{condition-eta_1} for $\lam\ge\mu_0$ we get
    \bes
    \big\|(\bSigmabar\bu)_{\bar S}\big\|_2 
    \le \big\|\bg_{\bar S}\big\|_2/\lam + \big\|\bar\bX_{\bar S}^\top(\by-\bar\bX\bbetabar)\big\|_2/(n\mu_0)
    \le (1+\eta_1)|{\bar S}|^{1/2}. 
    \ees
    Hence for all $\lam\ge \mu_0$, vector $\bu$ belongs to the set 
    $\scrU_0({\bar S}, \bSigmabar, \eta_1,\eta_2)$ in \eqref{srcU_0-norm-2}, 
    so that \eqref{conclusion-support-lasso} follows from \Cref{lem-false-positive}. 
\end{proof}

\lemmaFPdeterministic*

\begin{proof}[Proof of Lemma \ref{lem-false-positive}.]
Let $\scrU_1=\scrU_0({S},\bSigma;\eta_1,\eta_2)$.  
For each $\bu\in \scrU_1$, there exists a small $\eps>0$ for which  
$\bu\in \scrU_0({S},\bSigma + \eps^2\bI_{p\times p};\eta_1+\eps,\eta_2+\eps)$. 
Thus, as $\phi_{\rm cond}(m;{S},\bSigma + \eps^2\bI_{p\times p}) 
\le \phi_{\rm cond}(m;{S},\bSigma)$ and the conclusion is continuous in 
$(\eta_1,\eta_2)$, we assume without loss of generality that $\bSigma$ is positive definite. 

Let $B_{\bu} = \{j\in {S}^c: |(\bSigma \bu)_j|\ge 1-\eta_2\}$. 
We have $\supp(\bu)\setminus {S} \subseteq B_{\bu}$.  
Define 
\bes
k^* = \max\Big\{|B_{\bu}|: \bu\in\scrU_1\Big\},\quad 
t^* = \frac{\{\phi_{\rm cond}(m;{S},\bSigma)-1\}|{S}|}{2(1-\eta_2)^2/(1+\eta_1)^2}. 
\ees
We split the proof into two-steps. In the first step, we prove that 
for any integer $k\in [0, k^*]$, there exists a vector $\bu\in\scrU_1$ and $A$ satisfying 
\bel{sec3eq37}
\qquad
k = |A\setminus{S}|,\quad {S}\cup\supp(\bu)\subseteq A \subseteq {S}\cup 
\{j: |(\bSigma \bu)_j|\ge 1-\eta_2\}.\eel
In the second step, we prove that when (\ref{sec3eq37}) holds with $k\le m$, 
\bel{sec3eq38}
\|\bu_{{S}^c}\|_1\phi_{\max}(\bSigma_{{S},{S}})/(1-\eta_2) + k \le t^*. 
\eel
As $t^* < m$ by the SRC, $k\le m$ implies $k<m$ by the second step, 
so that $m\not\in [0,k^*]$ by the first step. The conclusion follows as $|\supp(\bu)\setminus{S}|\le k^*$. 

{\it Step 1.} Let $\bu^*\in\scrU_1$ with $|B_{\bu^*}|=k^*$. 
Let $B=B_{\bu^*}$. 
Define a vector $\bz$ by 
\bes
\bz_{B} = (\bSigma\bu^*)_B+\sgn(\bu^*_B),\quad \bz_{{S}} = (\bSigma\bu^*)_{S}. 
\ees 
As $\bu^*\in\scrU_1$, we have $\|\bz_B\|_\infty\le \eta_2$ and 
$\|\bz_{S}\|_2\le (1+\eta_1)|S|^{1/2}$.   
Consider an auxiliary optimization problem
\bes
\bb(\lam) = \argmin_{\bb\in\R^p}\Big\{\bb^T\bSigma\bb/2 - \bz^T\bb + \lam\|\bb_B\|_1: 
\supp(\bb)\subseteq {S}\cup B \Big\}. 
\ees
The KKT conditions for $\bb(\lam)$ can be written as 
\bes
\begin{cases}
|\big(\bSigma\bb(\lam) - \bz\big)_j|\le\lam, & j\in B, \cr
b_j(\lam)\big(\bSigma\bb(\lam) - \bz\big)_j  + \lam\,|b_j(\lam)|=0, & j\in {S}^c, \cr
\big(\bSigma\bb(\lam) - \bz\big)_j  =0, & j\in {S},\cr 
b_j(\lam)=0, & j\not\in {S}\cup B.
\end{cases}
\ees
Note that $b_j$ is penalized only for $j\in B$, but $j\in B$ does not guarantee $b_j(\lam)\neq 0$. 
Due to the positive-definiteness of $\bSigma$, the objective function of the auxiliary minimization 
problem is strictly convex, so that $\bb(\lam)$ is uniquely defined by the KKT conditions 
and continuous in $\lam$. 

Let $\bu(\lam)=\bb(\lam)/\lam$. 
For $\lam\ge 1$, the KKT conditions imply
\bes
\big|u_j(\lam)\big(\bSigma\bu(\lam)\big)_j  + |u_j(\lam)|\big|
= |b_j(\lam)z_j|/\lam^2
\le \eta_2|u_j(\lam)|,\  \forall\ j\in B, 
\ees 
and $\|\big(\bSigma\bu(\lam)\big)_{S}\|_2 = \|\bz_{S}\|_2 /\lam {\le\|\bz_{S}\|_2} \le(1+\eta_1)|S|^{1/2}$, 
so that $\bu(\lam)\in\scrU_1$. Let 
\bes
B(\lam) = \{j\in B: |(\bSigma \bu(\lam))_j|\ge 1-\eta_2\}. 
\ees
For $\lam=1$, the KKT conditions yield $\bb(1)=\bu^*$. Let 
\bes
\lam^* = \|\bSigma_{B,{S}}\bSigma_{{S},{S}}^{-1}\bz_{S}- \bz_B\|_\infty. 
\ees
For $\lam \ge \lam^*$, the solution is given by 
\bes
\bb_{S}(\lam)=\bSigma_{{S},{S}}^{-1}\bz_{S},\quad \bb_B(\lam)=0. 
\ees
Thus, $\bu(\lam)$ is a continuous path in $\scrU_1$ with 
$\supp(\bu(\lam^*))={S}$, $\bu(1)=\bu^*$ and 
$\supp\big(\bu(\lam)\big) \subseteq {S}\cup B(\lam)$. 
Let $k \in [0,k^*]$ and 
\bes
\lam_k = \sup\big\{\lam\in [1,\lam^*]: |B(\lam)| \ge k\ \hbox{ or }\ \lam = \lam^*\big\}. 
\ees
If $B(\lam^*)\ge k$, then (\ref{sec3eq37}) is feasible with $\bu=\bu(\lam^*)$ due to 
$\supp(\bu(\lam^*))={S}$. Otherwise, $\lam_k \in [1,\lam^*)$, 
$\supp(\bu(\lam_k))\subseteq {S}\cup B(\lam_k+)$, $|B(\lam_k+)|<k$, and $|B(\lam_k)|\ge k$ 
due to the continuity of $\bu(\lam)$ and the fact that $k \le k^*=|B(1)|$. 
Thus, (\ref{sec3eq37}) is feasible with $\bu=\bu({\lam_k})$. 

{\it Step 2.} Suppose (\ref{sec3eq37}) holds for certain $A$, $k\le m$ and $\bu\in\scrU_1$. 
We need to prove (\ref{sec3eq38}). 
Let $B=A\setminus {S}$, $\bv = (\bSigma\bu)_A\in \R^A$, 
$\bv_{({S})}=(v_jI\{j\in {S}\}, j\in A)\in \R^A$ and $\bv_{(B)}=\bv - \bv_{({S})}$. 
By algebra, 
\bes
\bv^{T}\bSigma_{A,A}^{-1}\bv+\bv_{(B)}^{T}\bSigma_{A,A}^{-1}\bv_{(B)}
- \bv_{({S})}^{T}\bSigma_{A,A}^{-1}\bv_{({S})}
= 2\bv^{T}\bSigma_{A,A}^{-1}\bv_{(B)}. 
\ees
Because $\bv^{T}\bSigma_{A,A}^{-1}\bv_{(B)}
= (\bv^{T}\bSigma_{A,A}^{-1})_{B}\bv_{B}= \bu_{B}^{T}(\bSigma\bu)_{B}\le - (1-\eta_2)\|\bu_B\|_1$, 
\bes
\frac{\|\bv_{(B)}\|_2^2+\|\bv\|_2^2}{\phi_{\max}(\bSigma_{A,A})}
&\le& \bv_{({S})}^{T}\bSigma_{A,A}^{-1}\bv_{({S})} - 2(1-\eta_2)\|\bu_B\|_1
\cr &\le& \frac{\|\bv_{({S})}\|_2^2}{\phi_{\min}(\bSigma_{A,A})} - 2(1-\eta_2)\|\bu_B\|_1.  
\ees
Since $\|\bv_{(B)}\|_2^2\ge (1-\eta_2)^2|B|$ and 
$\|\bv\|_2^2-\|\bv_{(B)}\|_2^2 = \|\bv_{({S})}\|_2^2\le(1+\eta_1)^2|{S}|$, 
\bes
&& 2(1-\eta_2)\|\bu_B\|_1\phi_{\max}(\bSigma_{A,A})+2(1-\eta_2)^2|B| 
\cr &\le& (1+\eta_1)^2|{S}|\Big(\frac{\phi_{\max}(\bSigma_{A,A})}{\phi_{\min}(\bSigma_{A,A})}-1\Big). 
\ees
As $|A\setminus {S}|=k\le m$ and ${S}\subseteq A$, 
$\phi_{\max}(\bSigma_{A,A})/\phi_{\min}(\bSigma_{A,A})\le \phi_{\rm cond}(m;{S},\bSigma)$ 
and $\phi_{\max}(\bSigma_{A,A})\ge \phi_{\max}(\bSigma_{{S},{S}})$. It follows that 
\bes
\|\bu_B\|_1\phi_{\max}(\bSigma_{{S},{S}})/(1-\eta_2)+|B| 
\le \frac{|{S}|\big(\phi_{\rm cond}(m;{S},\bSigma)-1\big)}{2(1-\eta_2)^2/(1+\eta_1)^2} = t^*. 
\ees
This completes Step 2 and thus the proof of the lemma.
\end{proof}

\propositionFPprobability*

\begin{proof}[Proof of \Cref{proposition:false-positive-probability-bound}]
    (i)
    Assume that the four events $\{\mu_0 \le \lambda \}$,
    $\Omega_{noise}^{(1)}$, $\Omega_{noise}^{(2)}$
    and $\Omega_{iso}(\ba_0)$ hold hereafter.
    By construction of the path in \eqref{bz(t)}, the vector $\bep^\top\bX(t)=\bep^\top\bX$
    is the same for all $t\ge0$ and both \eqref{event-noise-1} and \eqref{event-noise-2}
    also hold if $\bX$ is replaced by $\bX(t)$.
    The right hand side of \eqref{event-noise-1} satisfies
    ${k \|\bep\|_2^2}/(L_k^2 + 2)^{-1}
        =
        k \left(n\eta_2\mu_0 - \|\bep\|_2L_k\right)^2
    $
    by definition of $\mu_0$ in \eqref{def-mu_0}.
    Thus by the triangle inequality, on the event \eqref{event-noise-1},
    \bes
    \eta_2\mu_0 |\tilde S\setminus S|^{1/2} 
    &\le& \|\bX_{\tilde S\setminus S}(t)^\top \bep\|_2/n
    \\ &<& L_k\|\bep\|_2|\tilde S\setminus S|^{1/2}/n + (\eta_2\mu_0 - \|\bep\|_2 L_k/n) k^{1/2}
    ,
    \ees
    for all $t\ge0$
    which implies $|\tilde S\setminus S| < k$.  
    This gives the bound $|\tilde S|<|S|+k$, which we now improve further as follows.
    On the intersection of \eqref{event-noise-1} and \eqref{event-noise-2}
    we have
    \bel{pf-dim-bd-2}
    && \|\bX_{\tilde S}^\top\bep\|_2^2
    \\ \nonumber &=& \|\bX_{S}^\top\bep\|_2^2+\|\bX_{\tilde S\setminus S}^\top\bep\|_2^2  
    \\ \nonumber &<& \|\bep\|_2^2|S|\big(L_k+(L_k^2+2)^{-1/2}\big)^2 
    + \Big[(\|\bep\|_2L_k)|\tilde S\setminus S|^{1/2}+\frac{k^{1/2}\|\bep\|_2}{(L_k^2+2)^{-1/2}}\Big]^2
    \\ \nonumber & <&  \|\bep\|_2^2(|S|+k)\big(L_k+(L_k^2+2)^{-1/2}\big)^2 
    \\ \nonumber & =&  (n \eta_2 \mu_0)^2(|S|+k).
    \eel
    Hence, if we define $\eta_1$ by $\eta_1 = \|\bX_{\tilde S}^\top\bep\|_2/(n\mu_0|\tilde S|^{1/2})$, we have proved
    $(\eta_1/\eta_2)^2 |\tilde S| < |S|+ k$.
    Together with $|\tilde S\setminus S|<k$, this implies the improved bound
    \bel{size-tilde-S}
        \quad\qquad
        |\tilde S|(1+\eta_1)^2/(1+\eta_2)^2 
        \le
        |\tilde S| \max( (\eta_1/\eta_2)^2, 1 \big) 
        <
        |S|+k.
    \eel
    Next we apply \Cref{proposition:eta_1-eta_2-bar-beta} to $\bar S=\tilde S$
    to prove the second inequality in
    \eqref{conclusion-support-lasso-purpose-present-paper}  
    based on \eqref{size-tilde-S}.
    This means to check the following version of \eqref{SRC-tilde-S}, \eqref{def-S-bar} and
    \eqref{condition-eta_1}:  
    \bel{pf-dim-bd-1}
    & |\tilde S| < 2(1-\eta_2)^2m\big/
    \big[(1+\eta_1)^2\big\{\phi_{\rm cond}(m;\tilde S,\bSigmabar(t))- 1\big\}\big]\quad \forall t, 
    \\ \nonumber 
    & \tilde S \supseteq \supp(\bbeta) \cup\{j\in[p]: |\bx_j^\top(t)(\by(t)-\bX(t)\bbeta)/n|  \ge \eta_2\mu_0\} \quad \forall t, 
    \\ \nonumber 
    & \|\bX_{\tilde S}^\top(t)(\by(t) - \bX(t)\bbeta)\|_2/n \le \eta_1 \mu_0|\tilde S|^{1/2}\quad \forall t, 
    \eel
    with $\bar\bX = \bX(t)$, $\bbetabar = \bbeta$, $\bSigmabar(t)=\bX^\top(t)\bX(t)/n$
    and $\by(t)=\bep+\bX(t)\bbeta$.
    For all $t\ge0$,
    $\bX^\top(t)(\by(t)-\bX(t)\bbeta)=\bX^\top\bep$ so that the second line in \eqref{pf-dim-bd-1} holds with equality by definition of $\tilde S$ and
    the third line holds with equality by definition of $\eta_1$ given after \eqref{pf-dim-bd-2}.
    For the first inequality in \eqref{pf-dim-bd-1}, combining 
    \eqref{size-tilde-S}, \eqref{SRC-population-final} and \eqref{condition-number-X(t)} gives
    \bes 
        && |\tilde S| (1+\eta_1)^2/(1+\eta_2)^2 \\
        &<& |S|+k  \\
        &<& 2(1-\eta_2)^2m\big/
        \big[(1+\eta_2)^2\big\{(\tau^*/\tau_*)\phi_{\rm cond}(m+k;S,\bSigma)- 1\big\}\big] 
        \cr &\le& 2(1-\eta_2)^2m\big/
        \big[(1+\eta_2)^2\big\{\phi_{\rm cond}(m+k;S,\bSigmabar(t))- 1\big\}\big]
        \cr &\le& 2(1-\eta_2)^2m\big/
    \big[(1+\eta_2)^2\big\{\phi_{\rm cond}(m;\tilde S,\bSigmabar(t))- 1\big\}\big].  
    \ees
    Multiplying both sides by $(1+\eta_2)^2/(1+\eta_1)^2$ yields the first inequality
    in \eqref{pf-dim-bd-1}.
    
    (ii)
    For every $j=1,...,p$ the random variable $\|\bep\|_2^{-1}\bx_j^\top\bep$ has standard normal distribution
    hence by \cite[Lemma~B.1(ii)]{bellec_zhang2018second_order_stein} we have
    \bel{bound-expectation-sum-over-p}
    \E \sum_{j=1}^p \big(|\bx_j^\top\bep|\,\|\bep\|_2^{-1} -
    L_k \big)_+^2 
    \le
    \frac{4k \exp\left(\log(p/k)-L_k^2/2\right)
    }{(L_k^2 +2 )(2\pi L_k^2+4)^{1/2}}
    .
    \eel
    With $L_k=\sqrt{2\log(p/k)}$, the numerator of the right hand side equals $4k$
    and by Markov's inequality, event \eqref{event-noise-1} has
    probability at least $1-4/(2\pi L_k^2+4)^{1/2}$.
    Furthermore, $\E[\|\bX_S^\top \bep\|_2^2\|\bep\|_2^{-2}]= 
    \hbox{trace}(\bSigma_{S,S}) 
    \le |S|$.
    Hence by Markov's inequality, the probability of \eqref{event-noise-2}
    is at least
    $1- (L_k+(L_k^2+2)^{-1/2})^{-2}$.
    The union bound completes the proof.
\end{proof}

\section{Necessity of degrees-of-freedom adjustment}
\label{sec:proof-th-1}

\thmNecessityDofGaussianDesign*

\begin{proof}[Proof of \Cref{th-1}]
Thanks to the scale equivariance \eqref{scale-equi-variant},
we take the scale $C_0=\|\bSigma^{-1/2}\ba_0\|_2=1$ without loss of generality, so that \eqref{standard} holds.

Let $\bhlasso = \lasso-\bbeta$. 
As $\bz_0=\bX\bu_0$,  by simple algebra,
\bes
\big\langle \bz_0,\by - \bX\lasso\big\rangle 
= \big\langle \bz_0, \bep\big\rangle - \|\bz_0\|_2^2\big\langle \ba_0,\bhlasso\big\rangle
 - \big\langle \bz_0,\bX\bQ_0\bhlasso\big\rangle
\ees
with $\bQ_0 = \bI_{p\times p} - \bu_0\ba_0^\top$ as in (\ref{decomp-beta}). 
Thus, by (\ref{LDPE-df}), 
\begin{equation}
    \label{pf-th-1-1}
\quad
(1 - \nu/n)\big(\htheta_{\nu} - \theta\big) = 
\frac{\big\langle \bz_0, \bep\big\rangle}{\|\bz_0\|_2^2}
- (\nu/n)\big\langle \ba_0,\bhlasso\big\rangle 
- \frac{\big\langle \bz_0,\bX\bQ_0\bhlasso\big\rangle}{\|\bz_0\|_2^2}. 
\end{equation}
We note that as $\bu_0=\bSigma^{-1}\ba_0/\big\langle\ba_0,\bSigma^{-1}\ba_0\big\rangle$, 
$\bz_0 = \bX\bu_0$ is independent of $\bX\bQ_0$. However, $\bz_0$ is not independent of 
$\bX\bQ_0\bhlasso$. 
We will use throughout the proof that the operator norm of
$(\bX_S/\sqrt n)\bSigma_{S,S}^{-1/2}$ is $O_\P(1)$ so that
\begin{equation}
\|(\bX_S/\sqrt n)\bSigma_{S,S}^{-1/2}\|_{op} = O_\P(1), 
\|\bSigma_{S,S}^{-1/2}(\bX_S^\top\bX_S/n)\bSigma_{S,S}^{-1/2}\|_{op} = O_\P(1).
\label{operator-norm-X_S-davidson-szarek}
\end{equation}
This holds because the singular values of
a matrix of size $|S|\times n$ with standard normal entries and $|S|/n \lll 1$ are
bounded away from 0, cf. for instance \cite{DavidsonS01}.

We will also use throughout that when $\sgn(\lasso) = \sgn(\bbeta)$,
the KKT conditions can be equivalently written as one of
\begin{equation}
\label{pf-th-1-2}
\begin{split}
    \bX_S^\top\bep&= (n\lam)\sgn(\bbeta_S) + \bX_S^\top\bX_S \bhlasso, \\
    (\bX_S^\top\bX_S)^{-1}\bX_S^\top\bep&= (n\lam)(\bX_S\bX_S)^{-1}\sgn(\bbeta_S) + \bhlasso.
\end{split}
\end{equation}
We decompose (\ref{pf-th-1-1}) as follows
\begin{equation}
    \begin{split}
\label{pf-th-1-3}
& (1 - \nu/n)\big(\htheta_{\nu} - \theta\big) 
\\ &= 
\frac{\big\langle \bz_0, \bep\big\rangle}{\|\bz_0\|_2^2}
- \frac{s_0-\nu}{n}\Big\langle (\ba_0)_S,\lam(\bX_S^\top\bX_S/n)^{-1}\sgn(\bbeta_S)\Big\rangle +\sum_{j=1}^3 \Rem_j
    \end{split}
\end{equation}
with 
\begin{align*}
\Rem_1 &= \frac{s_0-\nu}{n}\Big(\big\langle \ba_0,\bhlasso\big\rangle 
+ \big\langle (\ba_0)_S,\lam(\bX_S^\top\bX_S/n)^{-1}\sgn(\bbeta_S)\big\rangle\Big), 
\\
\Rem_2 &= \bigg\{\frac{\big\langle \bz_0, \bP_S \bz_0\big\rangle}{\|\bz_0\|_2^2}  - \frac{s_0}{n}\bigg\}
\big\langle \ba_0,\bhlasso\big\rangle, 
\\
\Rem_3 &= \frac{\langle \bz_0, (\bX \bQ_0)_S \bhlasso \rangle + \langle \bz_0, \bP_S \bz_0 \rangle \langle \ba_0, \bhlasso \rangle}{\|\bz_0\|_2^2}
\end{align*}
where $\bP_S$ is the orthogonal projection onto the column space of $(\bX\bQ_0)_S$.
We now prove that each $(|\Rem_j|)_{j=1,2,3}$ is of order at most $\eta_n$, i.e., $|\Rem_j| = O_\P(\eta_n)$.
Since $\bep$ is independent of $\bX_S$,
the random variable 
$$Z= \langle (\bX_S^\dagger(\ba_0)_S, \bep\rangle
/\|\bX_S^\dagger(\ba_0)_S\|_2$$ has $N(0,\sigma^2)$ distribution and we have by
\eqref{pf-th-1-2}
\bes
\Big|\big\langle \ba_0,\bhlasso\big\rangle 
+  \big\langle (\ba_0)_S,\lam(\bX_S^\top\bX_S/n)^{-1}\sgn(\bbeta_S)\big\rangle\Big|
= |Z| n^{-1/2}\|(\bX_S /\sqrt n)^\dagger (\ba_0)_S\|_2
\ees
and $|Z| = O_\P(\sigma)$.
This proves that
$|\Rem_1| \le O_{\P}(\sigma|\nu-s_0|/n^{3/2})$.
Next, by \eqref{operator-norm-X_S-davidson-szarek} we get
$$\|(\bX_S /\sqrt n)^\dagger (\ba_0)_S\|_2^2=
(\ba_0)_S^\top \bSigma_{S,S}^{-1/2}(\bSigma_{S,S}^{-1/2}(\bX_S^\top\bX_S/n)\bSigma_{S,S}^{-1/2})^{-1}
\bSigma_{S,S}^{-1/2}(\ba_0)_S = O_\P(1)$$
due to $\|\bSigma_{S,S}^{-1/2}(\ba_0)_S\|_2\le \|\bSigma^{-1/2}\ba_0\|_2=C_0=1$
by \eqref{upper-bound-submatrix-ba_0}.
Furthermore, by definition of $C_{\bbeta}$ we have similarly
\bes
\Big|\lam \big\langle (\ba_0)_S,(\bX_S^\top\bX_S/n)^{-1}\sgn(\bbeta_S)\big\rangle\Big|
= O_{\P}(\lam C_{\bbeta}\sqrt{s_0}). 
\ees
Thus we have proved that
\begin{equation}
\label{pf-th-1-4}
\Big|\big\langle \ba_0,\bhlasso\big\rangle \Big|
= O_\P \big(\sigma/n^{1/2} +\lam C_{\bbeta}\sqrt{s_0}\big).
\end{equation}
As $\bz_0\sim N({\bf 0}, \bI_n)$
and $\bP_S$ is independent of $\bz_0$ we have
\begin{equation}
\label{pf-th-1-5}
\begin{split}
    &\|\bz_0\|_2 = \big\{\sqrt{n}+O_{\P}(1)\big\},\quad \|\bP_S\bz_0\|_2 = \big\{\sqrt{s_0}+O_{\P}(1)\big\}, \\
& \|\bP_S\bz_0\|^2/\|\bz_0\|^2 - s_0/n
    = O_\P(s_0^{1/2}n^{-1})
    .
\end{split}
\end{equation}
Applying (\ref{pf-th-1-4}) and (\ref{pf-th-1-5}) 
to bound the remainder term $|\Rem_2|$, we find that 
\bes
|\Rem_2| &\le& 
O_{\P}\left(s_0^{1/2}n^{-1}
\Big(\sigma n^{-1/2}+C_{\bbeta}\lam\sqrt{s_0}\Big)
\right)
\ees
We now bound $|\Rem_3|$.
Let $\bP_{\bep}$ be the orthogonal projection onto $\bep$
and let $\bP_{\bep}^\perp= \bI_n - \bP_{\bep}$.
Define $\tbz_0 = \bP_{\bep} \bz_0 + \bP_{\bep}^\perp \bg$,
where $\bg$ is independent of $(\bep,\bX)$ and $\bg$ is equal in distribution to $\bz_0$.
Hence, $\bep^\top \bz_0 = \bep^\top \tbz_0$ holds almost surely,
while conditionally on $\bep$, the two vectors $\bP_{\bep}^\perp \bz_0$ and $\bP_{\bep}^\perp
\tilde\bz_0$ are independent and identically distributed.
We define similarly 
$\tilde\bX = \bX \bQ_0 + \tilde\bz_0 \ba_0^\top$,
the Lasso estimator $\tilde\bbeta$ as the minimizer of
\begin{equation*}
    \tilde\bbeta
    = \argmin_{\bb\in\R^p}\Big\{ \|\bep +\tilde\bX\bbeta - \tilde\bX\bb\|_2^2/(2n) + \lam\|\bb\|_1\Big\}
\end{equation*}
and set $\tilde\bh = \tilde\bbeta - \bbeta$.
Note that $(\tilde\bz_0,\tilde\bX,\tilde\bbeta,\tilde\bh)$ has the same distribution as $(\bz_0,\bX,\lasso,\bhlasso)$ so that support recovery
\eqref{selection} is also granted to $\tilde\bbeta$.

On the event $\{\sgn(\lasso) = \sgn(\bbeta) = \sgn(\tilde \bbeta)\}$,
since $\bX^\top \bep = \tilde\bX^\top\bep$ holds,
the KKT conditions for the Lasso imply
\begin{equation}
\label{same-XTXh}
\bX_S^\top\bep - (n\lam)\sgn(\bbeta_S) 
= \bX_S^\top\bX \bhlasso 
= \tilde\bX_S^\top\tilde\bX \tilde\bh.
\end{equation}
Let $((\bX\bQ_0)_S^\top)^\dag$ be the Moore-Penrose generalized inverse 
of $(\bX\bQ_0)_S^\top$ and $\bP_S$ the orthogonal projection to the range of $(\bX\bQ_0)_S$ in $\R^n$. 
As $\supp(\bhlasso)\subseteq S$, 
$(\bX\bQ_0)\bhlasso$ lives in the range of $(\bX\bQ_0)_S$, so that 
\bes
 \big\langle \bz_0,(\bX\bQ_0)\bhlasso \big\rangle
+ \langle \bz_0, \bP_S \bz_0 \rangle \langle \ba_0, \bhlasso \rangle
&=& \langle \bz_0, \bP_S \bX_S \bhlasso \rangle. 
\ees
By \eqref{same-XTXh} and simple algebra we have
\begin{align*}
(\bX \bQ_0)_S^\top \bX \bhlasso
&= 
- (\ba_0)_S \bz_0^\top \bX \bhlasso
+\tilde\bX_S^\top \tilde\bX \tilde\bh
, \\
&=
 - (\ba_0)_S \bz_0^\top \bX \bhlasso
 + (\ba_0)_S \tilde\bz_0^\top \tilde \bX \tilde \bh
 + (\bX \bQ_0)_S^\top \tilde\bX \tilde\bh
 .
\end{align*}
Hence the quantity
    $| \langle \bP_S \bz_0, \bX \bhlasso \rangle |$
is bounded from above by
\begin{align*}
    \Big|
    \langle \bP_S \bz_0, \tilde \bX \tilde \bh \rangle
    \Big|
     +
     \Big| \langle \bP_S \bz_0, ( (\bX\bQ_0)_S^\top )^\dagger (\ba_0)_S \rangle \Big|
     \left( \|\bz_0\|_2\|\bX \bhlasso \|_2 \vee  \|\tilde\bz_0\|_2\|\tilde\bX \tilde\bh \|_2 \right).
\end{align*}
Note that $\|\bP_{\bep}\bz_0\|_2 = O_\P(1)$ while 
$\bP_{\bep}^\perp\bz_0$ is independent of $\bP_S\tilde\bX \tilde \bh$
and of $\bP_S((\bX\bQ_0)_S^\top )^\dagger (\ba_0)_S$.
Thus, since the operator norm of $\bP_s$ is at most $1$,
we have established that
\begin{multline*}
    \|\bz_0\|_2^2|\Rem_3| \le O_\P(1)
    \bigg[
    \|\tilde \bX \tilde \bh \|_2 
     \\ + \|((\bX\bQ_0)_S^\top )^\dagger (\ba_0)_S\|_2 
    \left( \|\bz_0\|_2\|\bX \bhlasso \|_2 \vee  \|\tilde\bz_0\|_2\|\tilde\bX \tilde\bh \|_2 \right)
    \bigg]
    .
\end{multline*}
Since $\bSigma^{1/2}\bQ_0\bSigma^{-1/2}$ is an orthogonal projection in $\R^p$,
\begin{equation*}
    \|((\bX\bQ_0)_S^\top )^\dagger (\ba_0)_S\|_2 
    =
    \|((\bX\bQ_0)_S^\top \bSigma^{-1/2} )^\dagger \bSigma^{-1/2}(\ba_0)_S\|_2 
    \le O_\P(1/\sqrt n).
\end{equation*}
Finally, using \eqref{pf-th-1-5} for $\bz_0$ and $\tbz_0$,
as well as $\|\bX\bhlasso\|_2+\|\tilde\bX\tilde\bh\|_2
\le O_\P(1) (\sigma/\sqrt n + \lambda C_{\bbeta} \sqrt{s_0})$,
we get that $|\Rem_3| \le (O_\P(1)/\sqrt n) (\sigma/\sqrt n + C_{\bbeta} \lambda\sqrt{s_0})$.

Combining the upper bounds for the remaining terms, we obtain
\bel{sum-of-Rem_j-th-1}
&& \bigg|\sum_{j=1}^3\Rem_j\bigg| 
\le \frac{O_{\P}(1)}{n^{1/2}}
\Big(\frac{\sigma|\nu-s_0|}{n^{3/2}} + \frac{\sigma}{n^{1/2}}+C_{\bbeta}\lam\sqrt{s_0}+\frac{\sigma\sqrt{s_0}}{\sqrt{n}}\Big)
.
\eel
If $\sqrt{(1\vee s_0)/n}+C_{\bbeta}\sqrt{s_0}(\lam/\sigma)\big) \le \eta_n$, 
(\ref{pf-th-1-3}) and \eqref{sum-of-Rem_j-th-1} complete the proof 
when $\htheta_\nu$ is given by \eqref{LDPE-df}.

Finally, we prove the equivalence of \eqref{LDPE-df}, \eqref{htheta-JM} and \eqref{est-ZZ}.
It follows from (\ref{pf-th-1-5}) that 
\bes
\frac{\langle \bz_0,\bX\ba_0\rangle}{\|\bz_0\|_2^2\|\ba_0\|_2^2}-1 
= \frac{\langle \bz_0,\bX\bQ_0\ba_0\rangle}{\|\bz_0\|_2^2\|\ba_0\|_2^2}
= O_{\P}(1)\frac{C_0\|\bSigma^{1/2}\ba_0\|_2}{n^{1/2}\|\ba_0\|_2^2} = O_{\P}(n^{-1/2})
\ees
when $C_0\|\bSigma^{1/2}\ba_0\|_2/\|\ba_0\|_2^2 = O(1)$. 
Let $\htheta_{\nu}$ be as in (\ref{LDPE-df}) and $\htheta_{\nu}'$ be the $\htheta_{\nu}$ in (\ref{est-ZZ}). 
It follows from \eqref{hat-theta-equivalence} that
\bes
 \sqrt{F_\theta n}(1 - \nu/n)\big(\htheta_{\nu} - \htheta_{\nu}'\big) 
 &\le& O_\P(1) |T_n|/\sqrt n + O_\P(1) \|\bX\bhlasso\|_2/(\sigma\sqrt n) \\
 &\le& O_{\P}(\eta_n). 
\ees
A similar argument yields the equivalence between (\ref{LDPE-df}) and (\ref{htheta-JM}).
\end{proof}

\section{Subgaussian design}
\label{sec:appendix-proof-subgaussian}
We provide here the proof of Theorem~\ref{thm:1-subgaussian},
    restated here for convenience.
\subGaussianTheoremDofNecessary*
\begin{proof}
On event \eqref{selection},
the Lasso and its error vector $\bh={\lasso}-\bbeta$
have a closed form expression, namely,
$\bh = (\bX_S^\top\bX_S)^{-1}[\bX_S^\top \bep - \lambda n \sgn(\bbeta)_S]$
thanks to the KKT conditions
$\bX_S^\top(\by-\bX{\lasso}) = \lambda n \sgn(\bbeta)_S$ on $S$.
If $\nu$ is a degrees-of-freedom adjustment
and $\hat\theta_\nu$ is the variant \eqref{htheta-JM},
we have
on the selection event
\bel{subgaussian-terms}
\nonumber
&&\sqrt n (1-\nu/n)(\hat\theta_\nu - \theta)
\\ &=&
\nonumber
\sqrt n (1-\nu/n)
\Big(
    \langle \ba_0,\bh\rangle + n^{-1}\langle \ba_0,\bSigma^{-1}\bX^\top(\by-\bX{\lasso})\rangle
\Big)
\\&=& 
\label{subgaussian-first-term}
\sqrt n (1-\nu/n)
\langle \ba_0, (\bX_S^\top\bX_S)^{-1} \bX_S^\top \bep\rangle
\\&&
- \lambda \sqrt n
(s_0-\nu)
\ba_0^\top\left[
(\bX_S^\top\bX_S)^{-1}
\right]
\sgn(\bbeta)_S
\label{subgaussian-second-term}
\\&&
- \lambda \sqrt n
\ba_0^\top\left[
(n-s_0)
(\bX_S^\top\bX_S)^{-1}
-
\bSigma^{-1}
\right]
\sgn(\bbeta)_S
.
\label{subgaussian-third-term}
\eel
The first term \eqref{subgaussian-first-term}
is unbiased asymptotically normal, while
the second \eqref{subgaussian-second-term} term represents the bias 
present when $|s_0-\nu|$ is large (i.e. an incorrect adjustment).
We now show that the third term \eqref{subgaussian-third-term} is negligible.
If $\ba_0=\bSigma\,\sgn(\bbeta)/\sqrt{s_0}$ the third term
\eqref{subgaussian-third-term} 
is equal to
\begin{equation}
    \label{subgaussian-quadratic-form-in-sgn-bbeta}
    \begin{split}
&-
\lambda \sqrt{n/s_0} (n-s_0) Q(\sgn(\bbeta)_S)
\quad\text{ where }
\\
&
Q(\sgn(\bbeta)_S)
=
\sgn(\bbeta)_S^\top\left[
\bSigma_{S,S}(\bX_S^\top\bX_S)^{-1}
-
(n-s_0)^{-1}
\bI_{S,S}
\right]
\sgn(\bbeta)_S
\end{split}
\end{equation}
The above $Q(\sgn(\bbeta)_S)$ is a quadratic form in the random vector $\sgn(\bbeta)_S$
and its expectation conditionally on $\bX$ is
\begin{equation}
\label{conditiona-expected-with-resect-to-signs-of-bbeta}
\E[Q(\sgn(\bbeta)_S)|\bX] = 
\trace\left[
\bSigma_{S,S}^{1/2}(\bX_S^\top\bX_S)^{-1}\bSigma_{S,S}^{1/2}
-
(n-s_0)^{-1}
\bI_{S,S}
\right].
\end{equation}
It holds that $\trace \bI_{S,S}= s_0$ for the second term,
while the first term, we must compute
$\|(\bX_S^\top \bX_S)^{-1/2}\bSigma^{1/2}_{S,S}\|_F^2 = \|(\bX_S\bSigma_{S,S}^{-1/2})^\dagger\|_F^2$.
Let $\bA\in\R^{n\times|S|}$ be the matrix
$\bA = \bX_S\bSigma_{S,S}^{-1/2}$, which has iid entries by assumption.
Since the Penrose pseudo-inverse satisfies
$\bA^\dagger \bA = \bI_{S,S}$,
the $j$-th row $\br_j$ of $\bA^\dagger$ satisfies
$\br_j^\top \ba_j = 1$
and 
$\br_j^\top \ba_k = 0$
for $k\in S\setminus \{j\}$
where $(\ba_j)_{j\in S}$ are the columns of $\bA$.
Furthermore by definition of the pseudo-inverse,
each row $\br_j$ of $\bA^\dagger$ belongs to the linear span
of the columns $\{\ba_k,k\in S\}$ of $\bA$.
This implies by algebra that $\br_j = \|\bQ_j\ba_j\|^{-2}\bQ_j\ba_j$
where $\bQ_j\in\R^{n\times n}$
is the projection onto the 
orthogonal complement of the span of $\{\ba_k, k\in S\setminus \{j\}\}$,
hence 
$$\trace[\bSigma_{S,S}(\bX_S^\top\bX_S)^{-1}] = \sum_{j\in S} \|\bQ_j\ba_j\|^{-2}.$$
Since $\bA$ has iid entries, the projector $\bQ_j$ is independent of $\ba_j$
and since $\ba_j$ has $n$ iid subgaussian entries,
$|\|\bQ_j\ba_j\| - \sqrt{n-s_0+1}| \le Kt$ with probability at least $1-2e^{-C t^2}$
where $K$ is the subgaussian norm of the entries of $\bA$,
cf. \cite[Theorem 6.3.2]{vershynin2018high}.
Here, we assume that $K$ is constant independent of $n,s_0,p$.
Set $t=C \log s_0$
for sufficiently large $C$, by the union bound
we have $\max_{j\in S}|\|\bQ_j\bx_j\|-\sqrt{n-s_0}| \le O_p(\sqrt{\log s_0})$
which implies
$\max_{j\in S}|\|\bQ_j\bx_j\|^{-2}-(n-s_0)^{-1}| \le O_p((n-s_0)^{-3/2}\sqrt{\log s_0})$.
Hence
\begin{equation}
\label{subgaussian-estimation-trace}
    \trace[\bSigma_{S,S}(\bX_S^\top\bX_S)^{-1}]
    = s_0(n-s_0)^{-1}\left(1 + O_P(\sqrt{\log(s_0)/(n-s_0)})\right),
\end{equation}
and the main terms of order $s_0(n-s_0)^{-1}$ cancel each other (this is the key!):
\bes
\trace[
\bSigma_{S,S}(\bX_S^\top\bX_S)^{-1}
-
(n-s_0)^{-1}
\bI_{S,S}]
&=&  O_P(s_0n^{-3/2}\sqrt{\log(s_0)}).
\ees
Hence the conditional expectation \eqref{conditiona-expected-with-resect-to-signs-of-bbeta}
is equal to $O_P(s_0n^{-3/2}\sqrt{\log(s_0)})$
and $\lambda\sqrt{n/s_0}(n-s_0)\E[Q(\sgn(\bbeta)_S)|\bX]$ 
is $O_P(\lambda\sqrt{s_0\log(s_0)})$.

If $\br$ has iid Rademacher  entries
and $\bM$ is a symmetric matrix then
$\Var[\br^\top\bM\br] = 2\|\bM-\diag(\bM)\|_F^2$
(cf., e.g., \cite[6.2.2]{Petersen2008matrix_cookbook}),
while if $\bM$ is not symmetric
$\Var[\br^\top\bM\br]
= \Var[\br^\top((\bM+\bM^\top)/2)\br]
= 2\|(\bM + \bM^\top)/2-\diag(\bM)\|_F^2
\le 2\|\bM - \diag(\bM)\|_F^2$. 
Let $\bB_S =\bSigma_{S,S}(\bX_S^\top\bX_S)^{-1}$.  
Since $\sgn(\bbeta)_S$ has iid Rademacher entries,
conditionally on $\bX$
\bes
&& \Var\big(Q(\sgn(\bbeta)_S) |\bX\big) 
\\ &\le& 2\|\bB_S - \diag(\bB_S)\|_F^2
\\ &\le& 2\|\bB_S - {\bI_S/n}\|_F^2
\\ &=& 2\big\|\bSigma_{S,S}^{1/2}
\big(\bI_S - \bSigma_{S,S}^{-1/2}(\bX_S^\top\bX_S/n)\bSigma_{S,S}^{-1/2}\big)
\bSigma_{S,S}^{1/2} (\bX_S^\top\bX_S)^{-1}\big\|_F^2
\\ & \le & 2\big\|\bI_S - \bSigma_{S,S}^{-1/2}(\bX_S^\top\bX_S/n)\bSigma_{S,S}^{-1/2}\big\|_F^2
\big\|\bSigma_{S,S}^{1/2} (\bX_S^\top\bX_S)^{-1}\bSigma_{S,S}^{1/2}\big\|_{op}^2
\phi_{\rm cond}(\bSigma_{S,S})
\\ & = & O_{\P}(s_0^2/n^3)\phi_{\rm cond}(\bSigma_{S,S})
\ees
thanks to
$\|\bI_S - \bSigma_{S,S}^{-1/2}(\bX_S^\top\bX_S/n)\bSigma_{S,S}^{-1/2}\|_{op}^2 = O_\P(s_0/n)$
\cite[Theorem 4.6.1]{vershynin2018high}
because the matrix $\bX_S\bSigma_{S,S}^{-1/2}$ has iid entries.
In summary, the conditional variance of $Q(\sgn(\bbeta)_S)$
given $\bX$
is at most 
$\phi_{\rm cond}(\bSigma_{S,S}) {O_\P(s_0^2/n^3)}
$
and the standard deviation of $\lambda\sqrt{n/s_0}(n-s_0)Q(\sgn(\bbeta)_S)$
is at most
$\phi_{\rm cond}(\bSigma_{S,S})^{1/2} {O_\P(\lambda \sqrt{s_0})}$.
By Chebyshev's inequality, the first claim is proved.

Finally,
$\lambda\sqrt n \ba_0^\top(\bX_S^\top\bX_S)^{-1}\sgn(\bbeta)_S
= \lambda\sqrt{n/s_0}\left(
    s_0(n-s_0)^{-1}+Q(\bbeta)
\right)$.
Due to the bound on the conditional expectation and variance of $Q(\bbeta)$,
this quantity is of order
$$\lambda\sqrt{n/s_0}\left(
    s_0(n-s_0)^{-1} 
    + {O_\P\Big(s_0n^{-3/2}\big(\sqrt{\log(s_0)} + \phi_{\rm cond}(\bSigma_{S,S})^{1/2}\big)\Big)}
    \right)
$$
which is equal to $\lambda\sqrt{s_0/n}(1-o_P(1))$.
\end{proof}

\end{document}